\documentclass[10pt]{amsart}

\usepackage[margin=1in]{geometry}
\usepackage{amsmath}
\usepackage{float}
\usepackage{booktabs}
\usepackage[dvipsnames]{xcolor}

\usepackage{ wasysym }

\usepackage{blkarray}

\usepackage{amscd,amsmath,amssymb,amsfonts,amsthm,ascmac}
\usepackage{enumerate,mathrsfs,stmaryrd,latexsym, comment, mathtools, mathdots} 

\usepackage[mathscr]{euscript}

\usepackage{caption}
\usepackage{subcaption}
\usepackage{hyperref}

\usepackage[all]{xy}

\usepackage{tikz-cd}
\usepackage{tikz}
\usetikzlibrary{snakes, 
	3d, matrix,decorations.pathreplacing,calc,decorations.pathmorphing, patterns}
\usetikzlibrary {positioning}
\usepackage{tikz-3dplot}

\numberwithin{equation}{section}
\allowdisplaybreaks[1]

\makeatletter
\newcommand{\leqnomode}{\tagsleft@true\let\veqno\@@leqno}
\newcommand{\reqnomode}{\tagsleft@false\let\veqno\@@eqno}
\makeatother

\newcommand{\defi}[1]{{\textit{#1}}}

\newcommand{\C}{{\mathbb{C}}}
\renewcommand{\P}{{\mathcal{P}}}
\newcommand{\R}{{\mathbb{R}}}

\newcommand{\Z}{{\mathbb{Z}}}
\newcommand{\Cstar}{{\C^{\ast}}}

\newcommand{\barn}{{\bar{n}}}

\newcommand{\p}{{\gamma}}
\newcommand{\ad}{{\delta}}
\newcommand{\cham}[1]{{\mathscr{C}_{#1}}}

\newcommand{\bfu}{{\mathbf{u}}}
\newcommand{\bfv}{{\mathbf{v}}}
\newcommand{\bfw}{{\mathbf{w}}}
\newcommand{\bfm}{{\mathbf{m}}}
\newcommand{\bfws}[1]{{\mathbf{w}_{{#1}}}}
\newcommand{\Bis}{{B_{\mathbf{i}}}}
\newcommand{\wS}{{\widehat{\Sigma}}}

\newcommand{\is}{{{\mathbf{i}}_{\ad}(I)}}
\newcommand{\isI}[2]{{\mathbf{i}_{{#1}}\left(#2 \right)}}

\newcommand{\Rn}[1]{{R(w_0^{(#1)})}}

\newcommand{\node}{\mathrm{node}}

\newcommand{\vs}{\vspace}

\DeclareMathOperator{\PC}{PC}
\DeclareMathOperator{\ind}{ind}
\DeclareMathOperator{\SL}{SL}
\DeclareMathOperator{\GL}{GL}

\DeclareMathOperator{\Cone}{Cone}

\DeclareMathOperator{\GC}{GC}

\newcommand{\GP}{{\mathcal{GP}}}

\newcommand{\A}{\mathsf{A}}
\newcommand{\D}{\mathsf{D}}

\newcommand{\mcal}{\mathcal}

\newtheorem{theorem}{Theorem}[section]
\newtheorem{lemma}[theorem]{Lemma}
\newtheorem{proposition}[theorem]{Proposition}
\newtheorem{corollary}[theorem]{Corollary}

\newtheorem{Question}[theorem]{Question}

\theoremstyle{definition}
\newtheorem{example}[theorem]{Example}
\newtheorem{definition}[theorem]{Definition}
\newtheorem{remark}[theorem]{Remark}

\begin{document}
	
\title[Small toric resolutions of string polytopes with small indices]{Small toric resolutions of  toric varieties of \\  string polytopes with small indices}
\author{Yunhyung Cho}
\address{Department of Mathematics Education, Sungkyunkwan University, Seoul, Republic of Korea}
\email{yunhyung@skku.edu}
\author{Yoosik Kim}
\address{Department of Mathematics, Brandeis University, Waltham, USA and Center of Mathematical Sciences and Applications, Harvard University, Cambridge, USA} 
\email{yoosik@brandeis.edu, yoosik@cmsa.fas.harvard.edu}
\author{Eunjeong Lee}
\address{Center for Geometry and Physics, Institute for Basic Science (IBS), Pohang 37673, Korea}
\email{eunjeong.lee@ibs.re.kr}
\author{Kyeong-Dong Park}
\address{Center for Geometry and Physics, Institute for Basic Science (IBS), Pohang 37673, Korea}
\email{kdpark@ibs.re.kr}


\begin{abstract}
	Let $G$ be a semisimple algebraic group over $\mathbb{C}$.
	For a reduced word $\bf i$ of the longest element in the Weyl group of $G$ and a dominant integral weight $\lambda$, one can construct the string polytope $\Delta_{\bf i}(\lambda)$, whose lattice points encode the character of the irreducible representation $V_{\lambda}$. 
	The string polytope $\Delta_{\bf i}(\lambda)$ is singular in general and 
	combinatorics of string polytopes heavily depends on the choice of $\mathbf i$. 
In this paper, we study combinatorics of string polytopes when $G = \SL_{n+1}(\C)$, and present a sufficient condition on $\mathbf i$ such that the toric variety $X_{\Delta_{\mathbf i}(\lambda)}$ of the string polytope $\Delta_{\mathbf i}(\lambda)$ has a small toric resolution.
Indeed, when $\mathbf i$  \defi{has small indices} and $\lambda$ is regular, we explicitly construct a small toric resolution of the toric variety~$X_{\Delta_{\bf i}(\lambda)}$ using a {\em Bott manifold}. 
Our main theorem implies that a toric variety of any string polytope admits a small toric resolution when $n < 4$.
As a byproduct, we show that if $\mathbf i$ has small indices then $\Delta_{\mathbf i}(\lambda)$ is integral for any dominant integral weight $\lambda$, which in particular implies that  
the anticanonical limit toric variety $X_{\Delta_{\bf i}(\lambda_P)}$ of a partial flag variety $G/P$ is Gorenstein Fano. 
Furthermore, we apply our result to symplectic topology of the full flag manifold $G/B$ and obtain 
a formula of the disk potential of the Lagrangian torus fibration on $G/B$ obtained from a flat toric degeneration of~$G/B$ to the toric variety $X_{\Delta_{\bf i}(\lambda)}$. 
\end{abstract}
\maketitle
\setcounter{tocdepth}{1} 
\tableofcontents
\date{\today}

\section{Introduction}
Let $X$ be a smooth projective variety over $\C$.
Lazarsfeld--Musta\c{t}\u{a}~\cite{LM} and Kaveh--Khovanskii~\cite{KaKho12} independently provided a systematic way of producing a semigroup $\Gamma$ and a convex body~$\Delta$
(called a {\em Newton--Okounkov body}) from the choice of a polarization $(X, \mcal{L})$ and a valuation $\nu$ on the ring of sections of $\mcal{L}$. 
When $\Gamma$ is finitely generated,  Anderson \cite{An13} showed that $\Delta$ is a rational convex polytope and constructed a toric degeneration of $X$ whose central fiber is the toric variety $X_{\Delta}$ of $\Delta$. This construction generalizes the previous works of \cite{GL, Cal, KM} for toric degenerations of Schubert varieties. 

An interesting classes of examples of Newton--Okounkov bodies are {\em string polytopes} introduced by Littelmann~\cite{Li} (see also~\cite{BeZe93, BeZe96}).
For a semisimple algebraic group $G$ over $\C$ of rank $n$, fix a reduced word $\mathbf i$ of the longest element of the Weyl group of $G$ and a dominant integral weight $\lambda$. Then one can 
obtain a {\em string polytope} $\Delta_{\bf i}(\lambda)$, a rational convex polytope in $\R^{\barn}$ where $\barn$ is the complex dimension of $G/B$, in which 
the lattice points parametrize elements of the dual crystal basis of the irreducible representation $V_\lambda$ of $G$ with highest weight 
$\lambda$ via the {\em string parametrization}.
Kaveh~\cite{Kav15} proved that the string parametrization associated to ${\bf i}$ induces a valuation $\nu_{\mathbf i}$ on the function field of $G/B$ such that 
the string polytope $\Delta_{\bf i}(\lambda)$ coincides with the Newton--Okounkov body for $(G/B, \mcal{L}_\lambda, \nu_{\mathbf i})$ where $\mathcal{L}_{\lambda}$ is a line bundle over $G/B$ determined by the weight $\lambda$. 

The theory of Newton--Okounkov bodies can be thought of as a bridge between algebraic theory and symplectic geometry.
Under the finite generatedness of $\Gamma$, Harada and Kaveh \cite{HaKa15} produced a completely integrable system $\Phi$ on $X$ making the following diagram commutes$\colon$
\begin{equation}\label{equ_toricdegcis}
	\xymatrix{
		  X  \ar[dr]_{\Phi} \ar[rr]^{ \phi}
                              & & X_0
      \ar[dl]^{\Phi_0} \\
  & \Delta &}
\end{equation}
where 
\begin{itemize}
\item $X_0 ( = X_{\Delta})$ is a projective toric variety of the Newton--Okounkov polytope $\Delta$ with a moment map $\Phi_0$,
\item $\phi$ is a continuous map (or a degeneration map) which is a symplectomorphism 
outside the singular loci of $X_0$. 
\end{itemize}
The system $\Phi$ leads to a Lagrangian torus fibration on $X$ over $\Delta$. 
Harada and Kaveh~\cite[Corollary~3.36]{HaKa15} proved the existence of the system $\Phi$ for the string polytope $\Delta_{\bf i}(\lambda)$.

A key step toward understanding Floer theory and deriving a local Landau--Ginzburg mirror complex chart of $\Phi \colon X \to \Delta$ is to compute the (Floer) disk potential of $\Phi$ introduced by Fukaya--Oh--Ohta--Ono. The disk potential arises from counting invariants of holomorphic disks bounded by its fiber, see \cite{COtoric,Aur, FOOO}.
According to the pioneering work of Nishinou--Nohara--Ueda \cite{NNU10, NNU12}, the toric degeneration of $\Phi$ in~\eqref{equ_toricdegcis} is very useful to compute the disk potential of $\Phi$ on $X$. Especially, if  $X$ is Fano and $X_0$ admits a {\em small toric resolution}\footnote{A resolution of a variety is called {\em small} if the exceptional loci have codimension greater 
than one.}, they proved that the disk potential of $X$ can be computed from the toric variety of a (generalized) conifold transition of $X$. 
Namely, the disk potential of $X$ agrees with the (toric) {\em Givental--Hori--Vafa potential}, a certain Laurent polynomial  which can be easily read off from the defining equations of the polytope $\Delta$.
In this regard, it is a meaningful question asking whether the toric variety of $\Delta$ has a small toric resolution.

In this manuscript, we focus on the case where $G = \mathrm{SL}_{n+1}(\C)$ and study the integrality of the string polytope $\Delta_{\mathbf i}(\lambda)$ for a dominant integral weight $\lambda$. 
It is proved in~\cite{An13} that if $\Delta_{\mathbf i}(\lambda)$ is integral, then the associated semigroup $\Gamma$ is finitely generated (and hence it yields the diagram~\eqref{equ_toricdegcis}). Moreover, for the weight $\lambda_P$ corresponding to the anticanonical bundle of a partial flag $G/P$, if the string polytope $\Delta_{\mathbf i}(\lambda_P)$ is integral, then the toric variety of $\Delta_{\mathbf i}(\lambda_P)$ is Gorenstein Fano (see \cite{Rusinko08, Steinert19}).
In fact, we will see that the integrality of $\Delta_{\mathbf i}(\lambda)$ holds when $\Delta_{\mathbf i}(\lambda)$ admits a small toric resolution (see Proposition~\ref{prop_small_resolution_implies_integral}).
Recently, Steinert~\cite[Example~7.5]{Steinert19} provides a non-integral string polytope. Accordingly, not every string polytope admits a small toric resolution.
In this regard, we address the following question.

\begin{Question}\label{Q_small}
	When does the toric variety of $\Delta_{\bf i}(\lambda)$ admit a small toric resolution? Can we construct the small toric resolution explicitly?
\end{Question}

The latter question is initiated from the observation of Batyrev, Ciocan-Fontanine, Kim, and van Straten \cite[Proposition~3.1.2]{BCKV}. 
They explicitly constructed a small toric resolution $\psi \colon B \to X_0$ of the toric variety $X_0$ of the {\em Gelfand--Cetlin} polytope where $B$ is a Bott manifold (see Definition~\ref{def_Bott_tower} for the definition of Bott manifolds).
Note that the Gelfand--Cetlin polytope associated to $\lambda$  is unimodularly equivalent\footnote{Two polytopes $P$ and $Q$ in $\R^{n}$ are \defi{unimodularly equivalent} if there exists an affine transformation $T \colon x \to Ax + v$ such that $A  \in \GL_{n}(\Z)$, $v \in \Z^{n}$, and $T(P) = Q$.} to the string polytope $\Delta_{\mathbf i_0}(\lambda)$ for the standard reduced word
\[
\mathbf{i}_0 \coloneqq (1,2,1,3,2,1,\dots,n,n-1,\dots,1).
\] 

In order to state our main result, we need to introduce some terminologies. For any reduced word ${\bf i}$, 
one can obtain a new reduced word obtained by interchanging two consecutive numbers $i$ and $j$ satisfying $|i-j| > 1$. We call such an operation a \defi{$2$-move}. 
We say that two reduced words ${\bf i}$ and ${\bf i}'$ are {\em equivalent} if one can be obtained from the other by applying a sequence of $2$-moves. 
Each equivalence class is 
called a {\em commutation class}. 
One important property of a reduced word $\mathbf i$ of the longest element of the Weyl group of $G$ is that by applying $2$-moves to $\mathbf i$ repeatedly, 
we obtain a new reduced word ${\bf i}_\D$ having the consecutive descending subsequence (denoted by $\D_n$) where
\[
\D_n \coloneqq (n,n-1,\dots,1).
\]
We can similarly produce a new reduced word ${\bf i}_\A$ equivalent to ${\bf i}$ which contains the consecutive ascending subsequence $\A_n = (1,2,\dots,n)$ 
(see Proposition~\ref{prop_index} or \cite[Proposition~3.2]{CKLP}).

Using the above properties, for each sequence $\delta \in \{\A,\D\}^n$ of letters consisting of `$\A$' and `$\D$' where each letter stands for `ascending' or `descending', respectively,
one can associate a non-negative integer vector $\mathrm{ind}_\delta({\bf i}) \in \Z^n$ called the {\em $\delta$-index} of ${\bf i}$ (see Definition~\ref{def_index_vector}). 
The present authors proved in \cite{CKLP} that a string polytope $\Delta_{\bf i}(\lambda)$ is unimodularly equivalent to the Gelfand--Cetlin polytope associated with $\lambda$
if and only if the $\delta$-index 
of ${\bf i}$ is the zero vector for some $\delta \in \{\A,\D\}^n$. 
We say that ${\bf i}$ has {\em small indices} if $\mathrm{ind}_\delta({\bf i}) = (0, \dots, 0, k)$ for some $\delta \in \{\A,\D\}^n$ and $k \leq \kappa(\delta_{n-1},\delta_n)$ 
where
	\begin{itemize}
		\item $\kappa(\delta_{n-1},\delta_n) = 2$ if $\delta_n = \delta_{n-1}$, and 
		\item $\kappa(\delta_{n-1},\delta_{n}) =n-1$ otherwise. 
	\end{itemize}	
See Definition~\ref{def_small} for the definition of small indices. Now we are ready to state our main theorem.

\begin{theorem}[{Theorem~\ref{thm_main}}]\label{thm_main_introduction}
	Let $\mathbf i$ be a reduced word of the longest element in the Weyl group of $\SL_{n+1}(\C)$ and $\lambda$ a regular dominant integral weight.
	If $\mathbf i$ has small indices, then the toric variety $X_{\Delta_{\mathbf i}(\lambda)}$ of the string polytope $\Delta_{\mathbf i}(\lambda)$ admits a small toric resolution $X_{\wS_{\bf i}}$. Moreover, the smooth projective toric variety $X_{\wS_{\bf i}}$ is isomorphic to a blow-up of a Bott manifold. 
\end{theorem}

Note that every reduced word of the longest element has small indices for $n \leq 3$. 
When $G = \SL_5(\C)$, there are $20$ commutation classes (out of $62$) having small indices, see Appendix~\ref{appendix_S5}.
It is worth while to mention that the condition $k \leq 2$ for small indices of the form $\mathrm{ind}_\delta({\bf i}) = (0, \dots, 0, k)$ is optimal to apply our construction.
Namely, there exists a choice of ${\bf i}$ such that its index is $(0, \dots, 0, 3)$ and its string polytope cannot be resolved using our construction (see Example~\ref{example_not_small}). 

As every string polytope unimodularly equivalent to the Gelfand--Cetlin polytope has small indices,  
Theorem~\ref{thm_main_introduction} generalizes \cite[Proposition~3.1.2]{BCKV}.

\begin{theorem}[{Corollaries~\ref{cor_integral_polytope}, \ref{cor_Fano}, and \ref{cor_disk_potential}}]\label{thm_main_intro_2}
Let $\mathbf i$ be a reduced word of the longest element in the Weyl group of $\SL_{n+1}(\C)$. Suppose that $\mathbf i$ has small indices. Then we have the following. 
\begin{enumerate}
	\item For any dominant integral weight $\lambda$, the string polytope $\Delta_{\mathbf i}(\lambda)$ is integral. 
	\item For a parabolic subgroup $P$ and the anticanonical weight $\lambda_P$,  the toric variety $X_{\Delta_{\mathbf i}(\lambda_P)}$ is Gorenstein Fano. 
	\item In case that $\lambda$ is regular, let $\Phi \colon G/B \to \Delta_{\mathbf{i}}(\lambda)$ be the completely integrable system 		given in \eqref{equ_toricdegcis}.
	Then the disk potential of the Lagrangian fiber $L(u)$ can be computed by the combinatorics of 
	$\Delta_{\mathbf i}(\lambda)$ for any interior point $u$ of $\Delta_{\mathbf i}(\lambda)$. 
\end{enumerate}
\end{theorem}

We note that the integrality of string polytopes has been conjectured in~\cite[Conjecture~5.8]{AB}, but it recently turned out that there exist non-integral string polytopes by~\cite[Example~7.5]{Steinert19}. 
Our main theorem presents a sufficient condition on the integrality of string polytopes.

This paper is organized as follows. 
In Section~\ref{section_deformation_and_Bott_manifolds}, we introduce some notations and well-known facts on toric varieties.
We also recall resolutions of toric varieties and {\em Bott manifolds} 
which will be used to construct a small toric resolution of the toric variety $X_{\Delta_{\mathbf i}(\lambda)}$. 
In Section~\ref{secDescriptionOfStringPolytopes}, we describe string polytopes in terms of explicit defining inequalities using Gleizer--Postnikov description.
In Section~\ref{secExtensionsAndContractions}, we explain certain operations on the set of reduced words, called the {\em extension} and {\em contraction}, introduced in 
the previous work \cite{CKLP} and illustrate how a string polytope changes when applying each operation.
In Section~\ref{section_comb_string_polytopes}, we study the combinatorics of string polytopes having $\delta$-indices of the form $(0,\dots,0,k)$ and describe the defining
inequalities in terms of {\em rigorous paths}.
In Section~\ref{section_small_resolution_main}, we give a construction of the small toric resolution of the toric variety $X_{\Delta_{\mathbf i}(\lambda)}$ of the string polytope ${\Delta_{\mathbf i}(\lambda)}$ when ${\bf i}$ 
has small indices and present the main theorem.  We also list some corollaries of our main theorem. 
Finally in Section~\ref{section_proof_of_main}, the proof of the main theorem will be provided. 

We have two appendices. The relation between Dynkin diagram automorphisms and combinatorics of string polytopes is explained in Appendix~\ref{appendix_automorphisms_and_string_polytopes}. 
In Appendix~\ref{appendix_S5}, we give the classification of reduced words of the longest element in $\mathfrak{S}_5$ having small indices. 

\subsection*{Acknowledgements} 
The authors thank to Naoki Fujita and Yonghwa Cho for their interests and valuable discussions.
The first author was supported by the National Research Foundation of Korea (NRF) grant funded by the Korea government (MSIP; Ministry of Science, ICT and Future Planning) (NRF-2020R1C1C1A01010972). 
The second author was supported by the Simons Collaboration Grant on
Homological Mirror Symmetry and Applications.
The third author was supported by IBS-R003-D1.
The fourth author was supported by  IBS-R003-Y1 and IBS-R003-D1.

\section{Bott manifolds and resolutions of singular toric varieties} 
\label{section_deformation_and_Bott_manifolds}
In this section, we recall some well-known facts on toric varieties and resolutions of singular toric varieties from~\cite{CLS11toric}. And then, we study certain smooth projective toric varieties, called \textit{Bott manifolds}, which will be used to construct small toric resolutions of toric varieties  of string polytopes. 
(See~\cite{GrKa94} and \cite[\S 7.8]{BP2015ToricTop2015} for more details on Bott manifolds.)

Let $n$ be a positive integer. Let $M$ be the character lattice of a torus $T \cong (\Cstar)^{n}$ and $N$ the lattice of one-parameter subgroups of $T$.
We denote 
\[
M_{\R} \coloneqq M \otimes_{\Z} \R  \quad \text{ and } \quad N_{\R} \coloneqq N \otimes_{\Z} \R
\]
so that $M_{\R} \cong \R^{n}$ and $N_{\R} \cong \R^{n}$. 
Let $X_{\Sigma}$ be the toric variety of a fan $\Sigma$ in $N_{\R}$ and $D = \sum_{\rho \in \Sigma(1)} a_{\rho} D_{\rho}$ a torus-invariant Cartier divisor on $X_{\Sigma}$.
Denote by $\Sigma(k)$ the set of $k$-dimensional cones in $\Sigma$ and $D_{\rho}$ the torus-invariant prime divisor corresponding to a ray $\rho \in \Sigma(1)$.
The divisor $D$ is called \defi{basepoint free} if $\mathcal{O}_{X_{\Sigma}}(D)$ is generated by global sections. 

There are several ways to determine the basepoint freeness of $D$
in terms of a polyhedron $P_{D}$, Cartier data $\{\bfm_{\sigma}\}_{\sigma \in \Sigma(n)}$, and the support function $\varphi_D$. 
The polyhedron $P_{D} \subset M_{\R}$ is defined by
\begin{equation}\label{eq_def_of_PD}
P_D = \{ \bfm \in M_{\R} \mid \langle \bfm, \bfu_{\rho} \rangle \geq - a_{\rho} \quad \text{ for all } \rho \in \Sigma(1)\}
\end{equation}
where $\bfu_{\rho}$ is the primitive integral vector generating a ray $\rho$.
When the fan $\Sigma$ is complete, then $P_D$ is bounded, i.e., $P_D$ is a polytope.
Since the divisor $D$ is Cartier, there exists $\bfm_{\sigma} \in M$ for each $\sigma \in \Sigma$ such that
\[
\langle \bfm_{\sigma}, \bfu_{\rho} \rangle = - a_{\rho} \quad \text{ for all } \rho \in \sigma.
\]
We call $\{\bfm_{\sigma}\}_{\sigma \in \Sigma(n)}$ the \defi{Cartier data}.
The \defi{support function} $\varphi_D\colon |\Sigma| \to \R$ is determined by the following properties:
\begin{itemize}
	\item $\varphi_D$ is linear on each cone $\sigma \in \Sigma$, and 
	\item $\varphi_D(\bfu_{\rho}) = -a_{\rho}$ for all $\rho \in \Sigma(1)$. 
\end{itemize}
More explicitly $\varphi_D$ is written by 
\[
\varphi_D(\bfu) = \langle \bfm_{\sigma}, \bfu \rangle \quad \text{ for all } \bfu \in \sigma
\]
for each $\sigma \in \Sigma(n)$.
We call a subset $\P \subset \{ \bfu_{\rho} \mid \rho \in \Sigma(1)\}$ a \defi{primitive collection} if $\Cone(\P) \notin \Sigma$  but $\Cone(\P \setminus \{\bfu_\rho\}) \in \Sigma$ for every $\bfu_\rho \in \P$.
We denote by $\PC(\Sigma)$ the set of primitive collections of $\Sigma$. 
From the definition, we observe the following.
\begin{lemma}\label{lem_smooth_fan}
	Let $\Sigma$ be a smooth fan and $S \subset \{ \bfu_{\rho} \mid \rho \in \Sigma(1)\}$. Then $\Cone(S) \in \Sigma$ if and only if 
	\[
	\P \not\subset S \quad \text{for any}~ \P \in \PC(\Sigma).
	\]
\end{lemma} 
\begin{proof}
	Suppose that $\Cone(S) \not \in \Sigma$. Then we may find a primitive collection $\P \subset S$ in an inductive way.
	This proves the ``if'' part. 	
	Conversely, assume that $S$ contains some $\P \in \mathrm{PC}(\Sigma)$.
	If $\Cone(S) \in \Sigma$, then $\Cone(S)$ is simplicial by the smoothness of $\Sigma$ and hence 
	$\Cone(\P)$ is also in the fan $\Sigma$. But this contradicts to the assumption $\P \in \PC(\Sigma)$.
	This completes the proof. 
\end{proof}

The following theorem presents equivalent conditions for the basepoint freeness of $D$ using primitive collections.
\begin{theorem}[{\cite[Proposition~6.1.1, Theorems~6.3.12 and~6.4.9]{CLS11toric}}]
	\label{thm_BPF}
	Let $X_{\Sigma}$ be a projective simplicial toric variety and let $D$ be a Cartier divisor. The following are equivalent:
	\begin{enumerate}
		\item $D$ is basepoint free.
		\item $\bfm_{\sigma} \in P_D$ for all $\sigma \in \Sigma(n)$.
		\item The support function $\varphi_D$ satisfies
		\[
		\varphi_D\left(\sum_{x \in \P} x \right) \geq \sum_{x \in \P} \varphi_D(x)
		\]
		for all $\P \in \textup{PC}(\Sigma)$.
	\end{enumerate}
	
\end{theorem}

\begin{example}\label{example_H2_PD}
	Let $X_{\Sigma}$ be the Hirzebruch surface $\mathcal{H}_{2} \coloneqq \C P(\mathcal{O}_{\C P^1} \oplus \mathcal{O}_{\C P^1}(2))$ associated with the complete fan $\Sigma$ 
	generated by four rays whose generators are 
	\[
	[ \bfu_1 \ \bfu_2 \ \bfu_3 \ \bfu_4 ] = 
	\begin{bmatrix}
	-1 & 0 & 1 & 0 \\ 2 & -1 & 0 & 1
	\end{bmatrix}.
	\]
	See~Figure~\ref{fig_fan_H2} for the fan $\Sigma$. 
		\begin{figure}[b]
		\begin{tikzpicture}[scale = 0.7]
		\filldraw[pattern= north west lines,  nearly transparent, draw=none] (0,0)--(0,2.5)--(2.5,2.5)--(2.5,0)--cycle ;
		\filldraw[pattern=  horizontal lines, nearly transparent, draw=none] (0,0)--(2.5,0)--(2.5,-2.5)--(0,-2.5)--cycle;
		\filldraw[pattern=  vertical lines, nearly transparent, draw=none] (0,0)--(-1.25,2.5)--(0,2.5)--cycle;
		\filldraw[pattern=  north east lines, nearly transparent, draw=none] (0,0)--(0,-2.5)--(-2.5,-2.5)--(-2.5,2.5)--(-1.25,2.5)--cycle;
		
		\node at (1.5,1) {$\sigma_1$};
		\node at (1.5,-1) { $\sigma_2$};
		\node at (-1,0) { $\sigma_3$};
		\node at (-0.3,1.8) { ${\sigma_4}$};
		
		\filldraw[fill=black] (1,0) circle (0.5mm) node[below] {\small$\bfu_3$};
		\filldraw[fill=black] (0,-1) circle (0.5mm) node[right] {\small $\bfu_2$};
		\filldraw[fill=black] (-1,2) circle (0.5mm) node [left] {\small $\bfu_1$};
		\filldraw[fill=black] (0,1) circle (0.5mm) node [right] {\small $\bfu_4$};

		\draw[->,thick] (0,0)--(2.5,0);
		\draw[->,thick] (0,0)--(0,2.5);
		\draw[->,thick] (0,0)--(-1.25,2.5);
		\draw[->,thick] (0,0)--(0,-2.5);
		
		\end{tikzpicture}
		\caption{The fan $\Sigma$ with $X_{\Sigma} = \mathcal{H}_2$.}
		\label{fig_fan_H2}
	\end{figure}
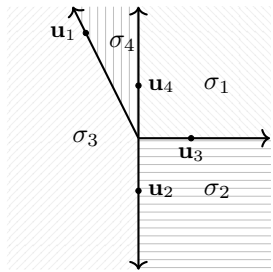
	Let $D_i$ be the divisor corresponding to $\bfu_i$ and consider 
	\[
	D = D_2 \quad \text{ and } \quad D' = D_2 -D_3 .
	\]
	Then the corresponding polytopes and the Cartier data for $D$ and $D'$ are given in Figures~\ref{fig_polytope_and_Cdata_D1} and~\ref{fig_polytope_and_Cdata_D2}, respectively.
	These figures and Theorem~\ref{thm_BPF} imply that $D$ is basepoint free while $D'$ is not.
	
	One can get the same conclusion using (3) in Theorem~\ref{thm_BPF} as follows. Note that the primitive collection for $\Sigma$ is given by
	$\textup{PC}(\Sigma) = \{\{\bfu_1, \bfu_3\}, \{\bfu_2,\bfu_4\}\}$. Since the support function is linear on each cone $\sigma \in \Sigma$, one can check the following:
	\begin{gather*}
	\varphi_D(\bfu_1 + \bfu_3) = \varphi_D(2 \bfu_4) = 2 \cdot 0
	\geq \varphi_D(\bfu_1) + \varphi_D(\bfu_3) = 0 + 0 = 0,\\
	\varphi_D(\bfu_2 + \bfu_4) = \varphi_D((0,0)) = 0 \geq \varphi_D(\bfu_2) + \varphi_D(\bfu_4) = -1 + 0 = -1.
	\end{gather*}
	This computation shows that $D$ is basepoint free again by Theorem~\ref{thm_BPF}.
	In a similar manner, we can check that $D'$ is not basepoint free as
	\begin{gather*}
	\varphi_{D'} (\bfu_1 + \bfu_3) = \varphi_{D'}(2 \bfu_4) = 2 \cdot 0 
	\ngeq \varphi_{D'}(\bfu_1) + \varphi_{D'}(\bfu_3) = 0 +1 = 1, \\
	\varphi_{D'}(\bfu_2 + \bfu_4) = \varphi_{D'}((0,0)) = 0 
	\geq \varphi_{D'}(\bfu_2) + \varphi_{D'} (\bfu_4) =-1 + 0= -1.
	\end{gather*}
\end{example}

The normal fan $\Sigma_{P_D}$ of the polytope $P_D$ for a basepoint free divisor $D$ has the following property:
\begin{proposition}[{\cite[Proposition~6.2.5]{CLS11toric}}]\label{prop_BPF_subdivision}
	Assume that $|\Sigma|$ is complete of full dimension $n$.
	Let $D = \sum_{\rho} a_{\rho} D_{\rho}$ be a basepoint free Cartier divisor on $X_{\Sigma}$ with the polytope $P_D$. 
If $v \in P_D$ is a vertex, then the corresponding cone $\sigma_v$ in the normal fan $\Sigma_{P_D}$ is the union
		\[
		\sigma_v = \bigcup_{\substack{\sigma \in \Sigma(n) \\ \bfm_{\sigma} = v }} \sigma.
		\]
In particular, the fan $\Sigma$ is a refinement of $\Sigma_{P_D}$.
\end{proposition}

	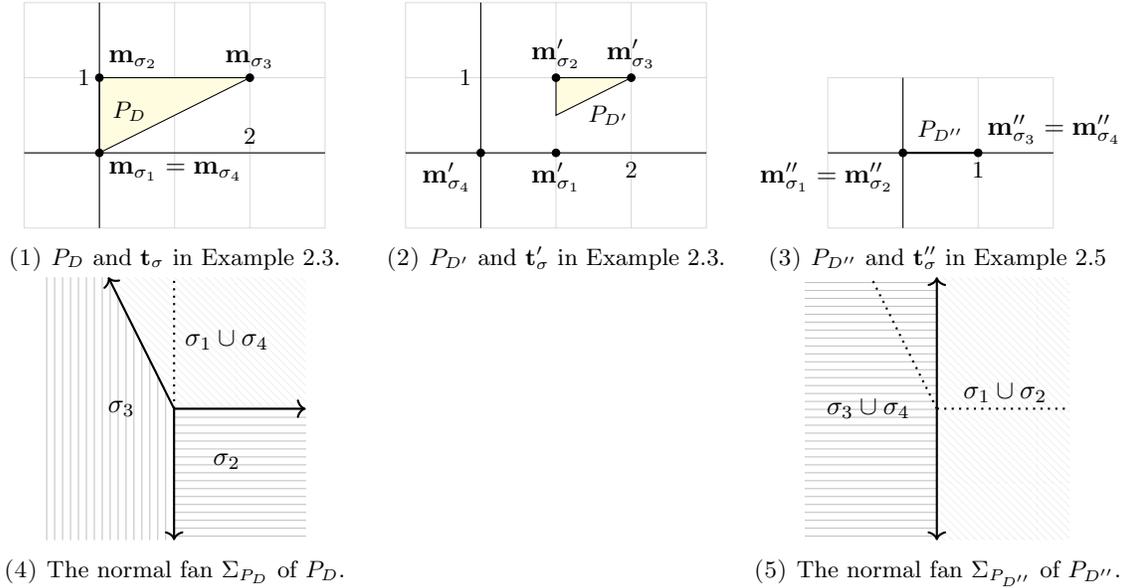
\begin{figure}
			\begin{subfigure}[b]{0.3\textwidth}
				\centering
				\begin{tikzpicture}[scale = 1]
				\draw[step = 1, gray!30!white, very thin] (-1,-1) grid (3,2);
				
				\filldraw[fill=yellow!15, draw=black] (0,0)--(2,1)--(0,1)--cycle;
				
				\node at (0.4,0.55) {\small $P_D$};
				\node[left] at (0,1) {\small $1$};
				\node[above] at (2,0) {\small $2$};

				\draw (-1,0)--(3,0);
				\draw (0,-1)--(0,2);
				
				\filldraw[fill=black] (0,0) circle (0.5mm) node[below right] { $\bfm_{\sigma_1} = \bfm_{\sigma_4}$};
				\filldraw[fill=black] (0,1) circle (0.5mm) node[above right] { $\bfm_{\sigma_2}$};
				\filldraw[fill=black] (2,1) circle (0.5mm) node [above] { $\bfm_{\sigma_3}$};
				
				\end{tikzpicture}
				\caption{$P_D$ and $\bf t_{\sigma}$ in Example~\ref{example_H2_PD}.}
				\label{fig_polytope_and_Cdata_D1}
			\end{subfigure}
			\begin{subfigure}[b]{0.3\textwidth}
				\centering
				\begin{tikzpicture}[scale = 1]
				\draw[step = 1, gray!30!white, very thin] (-1,-1) grid (3,2);
				
				\filldraw[fill=yellow!15, draw=black] (1,0.5)--(2,1)--(1,1)--cycle;
				
				\node at (1.7,0.5) {\small $P_{D'}$};
				\node[left] at (0,1) {\small $1$};
				\node[below] at (2,0) {\small $2$};

				\draw (-1,0)--(3,0);
				\draw (0,-1)--(0,2);
				
				\filldraw[fill=black] (0,0) circle (0.5mm) node[below left] { $\bfm_{\sigma_4}'$};
				\filldraw[fill=black] (1,0) circle (0.5mm) node[below] { $\bfm_{\sigma_1}'$};
				\filldraw[fill=black] (1,1) circle (0.5mm) node[above] { $\bfm_{\sigma_2}'$};
				\filldraw[fill=black] (2,1) circle (0.5mm) node [above] { $\bfm_{\sigma_3}'$};
				
				\end{tikzpicture}
				\caption{$P_{D'}$ and $\bf t_{\sigma}'$ in Example~\ref{example_H2_PD}.}
				\label{fig_polytope_and_Cdata_D2}
			\end{subfigure}
		\begin{subfigure}[b]{0.3\textwidth}
			\centering
			\begin{tikzpicture}[scale = 1]
			\draw[step = 1, gray!30!white, very thin] (-1,-1) grid (2,1);
			
			
			\draw (-1,0)--(2,0);
			\draw (0,-1)--(0,1);
			
			\draw[thick] (0,0)--(1,0);
			
			\node at (0.5,0.3) {\small $P_{D''}$};
			\node[below] at (1,0) {\small $1$};

			\filldraw[fill=black] (0,0) circle (0.5mm) node[below left] { $\bfm_{\sigma_1}'' = \bfm_{\sigma_2}''$};
			\filldraw[fill=black] (1,0) circle (0.5mm) node[above right] { $\bfm_{\sigma_3}'' = \bfm_{\sigma_4}''$};
			
			\end{tikzpicture}
			\caption{$P_{D''}$ and $\bf t_{\sigma}''$ in Example~\ref{example_H2_cone}}
			\label{fig_polytope_and_Cdata_D3}
		\end{subfigure}
	
		\begin{subfigure}[b]{0.3\textwidth}
			\centering
			\begin{tikzpicture}[scale = 0.7]

			\filldraw[pattern= north west lines, nearly transparent, draw=none] (0,0)--(2.5,0)--(2.5,2.5)--(-1.25,2.5,0)--cycle ;
			\filldraw[pattern= horizontal lines, nearly transparent, draw=none] (0,0)--(2.5,0)--(2.5,-2.5)--(0,-2.5)--cycle;
			\filldraw[pattern= vertical lines, nearly transparent, draw=none] (0,0)--(0,-2.5)--(-2.5,-2.5)--(-2.5,2.5)--(-1.25,2.5)--cycle;
			
			\node at (1,1.3) {$\sigma_1 \cup \sigma_4$};
			\node at (1,-1) { $\sigma_2$};
			\node at (-1,0) { $\sigma_3$};
			

			\draw[->,thick] (0,0)--(2.5,0);
			\draw[dotted, thick] (0,0)--(0,2.5);
			\draw[->,thick] (0,0)--(-1.25,2.5);
			\draw[->,thick] (0,0)--(0,-2.5);
			
			\end{tikzpicture}
			\caption{The normal fan $\Sigma_{P_D}$ of $P_D$.}
			\label{fig_fan_S1}
		\end{subfigure}
	\begin{minipage}[b]{0.3\textwidth}
		~~
	\end{minipage}
		\begin{subfigure}[b]{0.3\textwidth}
			\centering
			\begin{tikzpicture}[scale = 0.7]

			\filldraw[pattern= north west lines, nearly transparent, draw=none] (0,2.5)--(0,-2.5)--(2.5,-2.5)--(2.5,2.5)--cycle;
			%
			\filldraw[pattern= horizontal lines, nearly transparent, draw=none] (0,2.5)--(0,-2.5)--(-2.5,-2.5)--(-2.5,2.5)--cycle;
			%
			
			\node at (1.3,0.3) {$\sigma_1 \cup \sigma_2$};
			\node at (-1.3,0) {$\sigma_3 \cup \sigma_4$};
			%

			\draw[dotted,thick] (0,0)--(2.5,0);
			\draw[->,thick] (0,0)--(0,2.5);
			\draw[dotted,thick] (0,0)--(-1.25,2.5);
			\draw[->,thick] (0,0)--(0,-2.5);
			
			\end{tikzpicture}
			\caption{The normal fan $\Sigma_{P_{D''}}$ of $P_{D''}$.}
			\label{fig_fan_S2}
		\end{subfigure}
	\caption{The polytopes $P_{D}$ and normal fans in Examples~\ref{example_H2_PD} and \ref{example_H2_cone}.}
	\end{figure}

\begin{example}\label{example_H2_cone}
	Let $\Sigma$ be the fan of $\mathcal{H}_2$ as in Example~\ref{example_H2_PD}. 
	We consider divisors 
	\[
	D = D_2 \quad \text{ and }\quad D^{\prime \prime} = D_1.
	\]
	The polytopes $P_D$, $P_{D''}$ and the Cartier data are given in Figures~\ref{fig_polytope_and_Cdata_D1} and~\ref{fig_polytope_and_Cdata_D3}. By Theorem~\ref{thm_BPF}, both Cartier divisors $D$ and $D''$ are basepoint free. 
	The normal fans $\Sigma_{P_D}$ and $\Sigma_{P_{D''}}$ of the polytopes are given in Figures~\ref{fig_fan_S1} and \ref{fig_fan_S2}, respectively.
	One can see that $\Sigma$ is a refinement of both of $\Sigma_{P_D}$ and $\Sigma_{P_{D''}}$.
\end{example}

\begin{definition}
	For a singular toric variety $X_{\Sigma}$, we call ${X}_{\widehat{\Sigma}}$ a \defi{small toric desingularization of $X_{\Sigma}$} if $\widehat{\Sigma}$ is a smooth fan and
	the fan $\widehat{\Sigma}$ is a refinement of $\Sigma$ satisfying that $\widehat{\Sigma}(1) = \Sigma(1)$. 
	Then the toric morphism $\psi \colon {X}_{\widehat{\Sigma}} \to X_{\Sigma}$, called a \emph{small toric resolution}, 
	is a resolution of singularities and the exceptional locus of $\psi$ has codimension grater than one. 
\end{definition}
As a direct corollary of Proposition~\ref{prop_BPF_subdivision}, we have the following:
\begin{corollary}\label{cor_BPF_and_small_resolution}
	Let ${\Sigma}$ be a smooth complete polytopal fan in $N_{\R}$, and let $D$ be a basepoint free Cartier divisor on $X_{\Sigma}$. If  $\Sigma_{P_D}$ is singular and ${\Sigma}(1) = \Sigma_{P_D}(1)$, then $X_{\Sigma}$ is a small toric desingularization of $X_{\Sigma_{P_D}}$.
\end{corollary}

When a singular toric variety admits a small toric resolution, the corresponding polytopes are integral:
\begin{proposition}\label{prop_small_resolution_implies_integral}
	Let $X_{\Sigma}$ be a singular projective toric variety of dimension $n$. If $X_{\Sigma}$ admits a small toric desingularization $X_{\widehat{\Sigma}}$, then the polytope $P_D$ is integral
	for every Cartier divisor $D$ on $X_{\Sigma}$. That is, each vertex of $P_D$ is contained in the lattice $M \cong \Z^{n}$.
\end{proposition}
\begin{proof}
	Let $D = \sum_{k = 1}^m a_{k} D_{k}$, where $m = |\Sigma(1)|$. 
	Choose a vertex $v$ of $P_D$. Then the coordinate of $v$ is the solution of the linear equations:
	\begin{equation}\label{eq_system_v_in_PD}
	\{ \langle \bfm, \bfu_k  \rangle = -a_k \mid k \in J \},
	\end{equation}
	where $J$ is a subset of $[m]\coloneqq\{1, 2, \cdots, m\}$ satisfying that $\sigma_v = \Cone(\bfu_k \mid k \in J)$.
	Here $|J| \geq n$ since $v$ may not be a simple vertex, i.e., it could correspond to a non-simplicial maximal cone of the fan $\Sigma_{P_D}$.	
	Since the fan $\Sigma$ admits a small toric desingularization $\widehat{\Sigma}$, 
	every non-simplicial maximal cone admits a smooth subdivision by the definition.
	Hence the system in~\eqref{eq_system_v_in_PD} can be reduced to the system $\{ \langle \bfm, \bfu_k  \rangle = -a_k \mid k \in J'\}$  such that $|J'| = n$ and the set $\{ \bfu_k \mid k \in J'\}$  forms a $\Z$-basis of the lattice $N$. Therefore the solution of the system~\eqref{eq_system_v_in_PD} is contained in $M \cong \Z^{n}$, and the result follows.
\end{proof}

Now we introduce  Bott manifolds which are smooth projective toric varieties.
	\begin{definition}\label{def_Bott_tower}
		A \defi{Bott tower} $B_{\bullet}$ of height $n$ is a tower of fiber bundles
		\[
		B_n \stackrel{p_{n}}{\longrightarrow} B_{n-1} 
		\stackrel{p_{n-1}}{\longrightarrow} \cdots 
		\stackrel{p_2}{\longrightarrow} B_1 \longrightarrow pt,
		\]
		of smooth projective toric varieties, where $B_1 = \C P^1$ and $B_j = \C P( \mathcal{O}_{B_{j-1}} \oplus \xi_{j-1})$ for $2 \leq j \leq n$. Here, $\xi_{j-1}$ is a complex line bundle
		over $B_{j-1}$.
		We refer to $B_j$ as a \defi{$j$-stage Bott manifold} (or just a \defi{Bott manifold}).
	\end{definition}
	
	For instance, the complex projective line	$\C P^1$, the Hirzebruch surface $\mathcal{H}_k = \C P(\mathcal{O}_{\C P^1} \oplus \mathcal{O}_{\C P^1}(k))$, 
	and the product $\C P^1 \times \cdots \times \C P^1$ are Bott manifolds.
	The Picard group of $B_{j-1}$ is isomorphic to the free abelian group of rank $j-1$ and
	there is a canonical way of constructing an isomorphism from $\Z^{j-1}$ to $\mathrm{Pic}(B_{j-1})$ as follows. (We refer the reader to~\cite[\S 2]{GrKa94} for more details.)
	Let $\eta_{j,j-1}$ be the dual of the tautological line bundle over $B_{j-1}$ and define $\eta_{j,i}$ for $1 \leq i \leq j-2$ to be the pullback bundle 
	$\eta_{j,i} \coloneqq p_{j-1}^{\ast} \circ \cdots \circ p_{i+1}^{\ast}(\eta_{i+1,i})$. Then the map
	\[
	\Z^{j-1} \to \textup{Pic}(B_{j-1}), \quad (a_1,\dots,a_{j-1}) \mapsto (\eta_{j,1})^{\otimes a_1} \otimes \cdots \otimes (\eta_{j,j-1})^{\otimes a_{j-1}}
	\]
	is an isomorphism. Hence for each line bundle $\xi_{j-1}$, there exist integers $a_{j,1},\dots,a_{j,j-1}$ such that
	\begin{equation}\label{eq_first_Chern_classes}
	\xi_{j-1} \cong (\eta_{j,1})^{\otimes a_{j,1}} \otimes \cdots \otimes (\eta_{j,j-1})^{\otimes a_{j,j-1}}.
	\end{equation}

	It is known from~\cite[\S2.3]{GrKa94} that an $n$-stage Bott manifold is  determined by the set of 
	integers $\{a_{j,i} \mid 1 \leq i < j \leq n\}$. 
	Moreover, the fan of a Bott manifold  can be described
	as follows:
	\begin{theorem}[{cf. \cite[\S7.3]{CLS11toric} and \cite[Theorem 7.8.6]{BP2015ToricTop2015}}]\label{thm_Bott_manifold}
		Suppose that $B_{\bullet}$ is a Bott tower determined by  $\{a_{j,i} \mid 1 \leq i<j \leq n\}$. Then $B_n$ is a smooth projective toric variety of the fan $\Sigma \subset \R^n$ 
		with $2n$ ray vectors
		\begin{equation}\label{eq_aj_bj_Bn}
		\mathbf{v}_j = -\mathbf{e}_j + a_{j+1,j} \mathbf{e}_{j+1} + \cdots + a_{n,j} \mathbf{e}_n, \quad \mathbf{w}_j = \mathbf{e}_j
		\quad \text{ for } 1 \leq j \leq n,
		\end{equation}
		and $2^n$ maximal cones generated by the sets 
		$\{\mathbf{v}_j \mid j \in S \} \cup \{ \mathbf{w}_j \mid j \notin S\}$
		for all subsets $S \subset [n]$. Here, $\{\mathbf{e}_1,\dots,\mathbf{e}_n\}$ is the standard basis of $N_{\R} \cong \R^n$. 
	\end{theorem}
	
	\begin{example}
	The fan of the Hirzebruch surface $\mathcal{H}_k = \C P^1 (\mathcal{O}_{\C P^1} \oplus \mathcal{O}_{\C P^1}(k))$ has four ray vectors
	\[
	[\bfv_1 \ \bfv_2 \ \bfw_1 \ \bfw_2 ]
	= \begin{bmatrix}
	-1 & 0 & 1 & 0 \\
	k & -1 & 0 & 1
	\end{bmatrix},
	\]
	and has four maximal cones parameterized by $S \subset [2] = \{1,2\}$:
	\[
	\begin{array}{|c|cccc|}
	\hline
	S & \emptyset & \{1\} & \{2\} & \{1,2\} \\
	\hline 
	\Cone & \Cone(\bfw_1, \bfw_2) & \Cone(\bfv_1, \bfw_2) &  \Cone(\bfv_2, \bfw_1) & \Cone(\bfv_1, \bfv_2)\\
	\hline
	\end{array}
	\]
	See Example~\ref{example_H2_PD} and Figure~\ref{fig_fan_H2} for the fan of $\mathcal{H}_2$. 
	\end{example}
	By the description of the maximal cones of the fan of a Bott manifold in Theorem \ref{thm_Bott_manifold}, 
	we are able to list all primitive collections of $B_n$.
	\begin{proposition}\label{prop_primitive_collections_Bott}
		Let $\Sigma$ be the fan of an $n$-stage Bott manifold with ray vectors $\{\mathbf{v}_j, \mathbf{w}_j \mid 1 \leq j \leq n \}$
		as in Theorem \ref{thm_Bott_manifold}.
		Then there are exactly $n$ primitive collections of $\Sigma$ given by
		\[
		\textup{PC}(\Sigma) = \{\{ \mathbf{v}_j, \mathbf{w}_j \} \mid 1 \leq j \leq n\}.
		\]
	\end{proposition}
	
	Suppose that $U_1$ and $U_2$  are lower triangular matrices in $\GL_n(\Z)$
	such that the diagonal entries of $U_1$ are all $-1$ and that of $U_2$ are all $1$.  
	Then the pair $(U_1,U_2)$ defines an $n$-stage Bott manifold where the column vectors of the matrix $U_2^{-1}U_1$ correspond to ${\bf v}_j$'s in \eqref{eq_aj_bj_Bn}
	together with  $\{\mathbf{e}_j \mid 1 \leq j \leq n\}$.
	Hence, without loss of generality, we may consider such matrices $U_1$ and $U_2$ when we define a Bott manifold.
	\begin{example}\label{example_Bott_manifold_3stage}
		Suppose that we have the following vectors:
		\[
		\begin{bmatrix}
		\mathbf{v}_1 & \mathbf{v}_2 & \mathbf{v}_3
		& \mathbf{w}_1 & \mathbf{w}_2 & \mathbf{w}_3
		\end{bmatrix}
		=
		\begin{bmatrix}
		-1 & 0 & 0 & 1 & 0 & 0 \\ 0 & -1 & 0 & 1 & 1 & 0 \\
		-1 & 0 & -1 & 0 & 0 & 1
		\end{bmatrix}.	
		\]
		Since both matrices $\begin{bmatrix}
		\mathbf{v}_1 & \mathbf{v}_2 & \mathbf{v}_3
		\end{bmatrix}$ 
		and 
		$\begin{bmatrix}
		\mathbf{w}_1 & \mathbf{w}_2 & \mathbf{w}_3
		\end{bmatrix}$ are lower triangular matrices in $\GL_3(\Z)$, where 
		the former one has diagonal entries $-1$ and the latter one has diagonal entries $1$, they define a $3$-stage Bott manifold.
	\end{example}

	We
	recall a star subdivision of fans and blow-ups of a toric variety. 
	See~\cite[V.6]{Ewald96Combinatorial} and \cite[\S 3.3]{CLS11toric}.
	\begin{definition}[{\cite[Definition 3.3.17]{CLS11toric}}]\label{def_star}
		Let $\Sigma$ be a fan in $N_{\R} \cong \mathbb{R}^n$ and assume $\tau \in \Sigma$
		has the property that all cones of $\Sigma$ containing $\tau$ are smooth.
		Let $\mathbf{u}_{\tau} = \sum_{\rho \in \tau(1)} \mathbf{u}_{\rho}$ and for each cone 
		$\sigma \in \Sigma$ containing $\tau$, set
		\[
		\Sigma^{\ast}_{\sigma}(\tau) = \{
		\Cone(A) \mid A \subset \{ {\bf u}_{\tau}\} \cup \sigma(1), \, \, \tau(1) \nsubseteq A\}.
		\]
		Then the \defi{star subdivision} of $\Sigma$ relative to $\tau$ is
		the fan
		\[
		\Sigma^{\ast}(\tau) = \{\sigma \in \Sigma \mid \tau \nsubseteq \sigma\}
		\cup \bigcup_{\tau \subseteq \sigma} \Sigma^{\ast}_{\sigma}(\tau). 
		\]
	\end{definition}
\begin{example}\label{example_star_subdivision}
		Let $\Sigma$ be the fan of the $3$-stage Bott manifold defined in Example~\ref{example_Bott_manifold_3stage}. 
Let $\tau = \Cone(\bfv_1, \bfw_2)$ and $\bfu_{\tau} = \bfv_1+ \bfw_2$. Then there are three cones containing $\tau$: 
\[
\tau,\quad \sigma_1 \coloneqq \Cone(\mathbf{v}_1, \bfw_2, \bfv_3),\quad
 \sigma_2 \coloneqq \Cone(\bfv_1, \bfw_2, \bfw_3).
\] 
One can easily check that
\[
\begin{split}
\Sigma_{\tau}^{\ast}(\tau) = \{ \Cone(A) &\mid A \in \{ \{\bfu_{\tau}\}, \{\bfv_1\}, \{\bfw_2\}, \{\bfu_{\tau}, \bfv_1\}, \{\bfu_{\tau}, \bfw_2\}\}, \\
\Sigma_{\sigma_1}^{\ast}(\tau)= \{ \Cone(A) &\mid A \in \{ \{ \bfu_{\tau}\}, \{\bfv_1\}, \{\bfw_2\}, \{ \bfv_3\},
\{\bfu_{\tau},\bfv_1\}, \{\bfu_{\tau},\bfw_2\}, \{\bfu_{\tau},\bfv_3\}, \{\bfv_1,\bfv_3\}, \{ \bfw_2,\bfv_3\}, \\ 
&\qquad \{\bfu_{\tau},\bfv_1,\bfv_3\}, \{\bfu_{\tau},\bfw_2,\bfv_3\}\}\}, \\
\Sigma_{\sigma_2}^{\ast}(\tau) = \{ \Cone(A) &\mid A \in \{ \{ \bfu_{\tau}\}, \{\bfv_1\}, \{\bfw_2\}, \{ \bfw_3\},
\{\bfu_{\tau},\bfv_1\}, \{\bfu_{\tau},\bfw_2\}, \{\bfu_{\tau},\bfw_3\}, \{\bfv_1,\bfw_3\}, \{ \bfw_2,\bfw_3\}, \\ 
&\qquad \{\bfu_{\tau},\bfv_1,\bfw_3\}, \{\bfu_{\tau},\bfw_2,\bfw_3\}\}\}.
\end{split}
\]
Therefore, the star subdivision of $\Sigma$ relative to $\tau$ is the fan computed by 
\[
\begin{split}
\Sigma^{\ast}(\tau) &= \{\sigma \in \Sigma \mid \tau \nsubseteq \sigma\}
\cup \Sigma_{\tau}^{\ast}(\tau) \cup \Sigma_{\sigma_1}^{\ast}(\tau) \cup \Sigma_{\sigma_2}^{\ast}(\tau) \\
&= \{\sigma \in \Sigma \mid \tau \nsubseteq \sigma\} \cup 
\{ \Cone(A) \mid A \subset \{ \bfu_{\sigma},\bfv_1,\bfv_3\}\} \cup \{ \Cone(A) \mid A \subset \{ \bfu_{\tau},\bfw_2, \bfv_3\}\} \\
&\qquad \cup \{ \Cone(A) \mid A \subset \{ \bfu_{\tau},\bfv_1,\bfw_3\}\} \cup \{ \Cone(A) \mid  A \subset\{ \bfu_{\tau},\bfw_2,\bfw_3\}\}.
\end{split}
\]
\end{example}

	The fan $\Sigma^{\ast}(\tau)$ is a refinement of $\Sigma$ and induces a toric morphism
	$\psi \colon X_{\Sigma^{\ast}(\tau)} \to X_{\Sigma}$. 
	Under the map~$\psi$, $X_{\Sigma^{\ast}(\tau)}$ becomes the blowup of $X_{\Sigma}$
	along the orbit closure $V(\tau)$.
	Moreover, if the fan $\Sigma$ is polytopal, i.e., $\Sigma$ is the normal fan of a certain polytope $P$,
	then so is the fan $\Sigma^*(\tau)$. 
	If $\Sigma$ is a smooth fan in addition, then the star subdivision $\Sigma^{\ast}(\tau)$ is also smooth (see~\cite[Theorem V.6.2]{Ewald96Combinatorial}).
	
	We finish up this section with the following proposition which will be used in Section \ref{section_proof_of_main}.
	\begin{proposition}[{\cite[Theorem~4.3]{Sato_Towards_2000}}]\label{prop_PC_star_subdivision}
		Let $\Sigma$ be a finite complete simplicial fan in $N_{\R}$ and $\tau \in \Sigma$. Then the primitive collections  
		of the star subdivision $\Sigma^{\ast}(\tau)$ are
		\begin{itemize}
			\item $G(\tau)\coloneqq \{\mathbf{u}_{\rho} \mid \rho \in \tau(1)\}$,
			\item $\P \in \textup{PC}(\Sigma)$ such that $G(\tau) \not\subset \P$, and
			\item the minimal elements in the set 
			\[
			\{(\P \setminus G(\tau)) \cup \{\mathbf{u}_{\tau}\} \mid \P \in \textup{PC}(\Sigma), \,\, \P \cap G(\tau) \neq \emptyset\}.
			\]
		\end{itemize}
	\end{proposition}

	
	\begin{example}
		We continue Example~\ref{example_star_subdivision}.
By Proposition~\ref{prop_PC_star_subdivision}, we have that
		\[
		\PC(\Sigma^{\ast}(\tau)) = \{\{\bfv_1,\bfw_2\}, \{\bfv_1, \bfw_1\}, \{\bfv_2, \bfw_2\}, \{\bfv_3, \bfw_3\}, \{\bfv_2, \bfv_1 + \bfw_2 \}, \{\bfw_1, \bfv_1 + \bfw_2 \}\}.
		\]	
	\end{example}

\section{Gleizer--Postnikov's description of string polytopes}
\label{secDescriptionOfStringPolytopes}


In this section, we briefly review the combinatorial theory of string polytopes of type A as well as several notations introduced in the earlier work \cite{CKLP}. 
For $G = \SL_{n+1}(\C)$, the Weyl group of $G$ can be naturally identified with the symmetric group $\frak{S}_{n+1}$  generated by the simple transpositions $s_1, \dots, s_n$ in $\frak{S}_{n+1}$ (where $s_i = (i, i+1)$). 
The weight lattice $\Lambda$ of $G$ can be identified with $\Z^n$ and we denote by $\varpi_1,\dots,\varpi_n$ the fundamental weights. 
We call a weight $\lambda = \lambda_1 \varpi_1 + \cdots + \lambda_n \varpi_n$ {\em dominant} if $\lambda_i \geq 0$ for all~$i = 1,\dots,n$, 
and {\em regular dominant} if $\lambda_i > 0$ for all $i = 1,\dots,n$.
Denote by $\Lambda_{+}$ the set of dominant integral weights. 

Let $w_0^{(n+1)}$ be the longest element of $\frak{S}_{n+1}$ and $\Rn{n+1}$ the set of reduced words representing $w_0^{(n+1)}$, i.e.,
\[
\Rn{n+1} = \{ \mathbf{i} = (i_1,\dots,i_{\barn}) \in [n]^{\barn} \mid s_{i_1} s_{i_2} \cdots s_{i_{\barn}} = w_0^{(n+1)} \},
\]
where $\barn$ is the length of the longest element $w_0^{(n+1)}$ which can be computed as
\begin{equation}\label{eq_def_of_nbar}
\barn = \frac{n (n+1)}{2}.
\end{equation}

For a reduced word ${\bf i} \in \Rn{n+1}$ and a dominant integral weight $\lambda = \lambda_1 \varpi_1 + \cdots + \lambda_n \varpi_n $, the {\em string polytope} $\Delta_{\bf i}(\lambda)$
is a convex polytope defined in the Euclidean space $\R^{\barn}$ such that lattice points in $\Delta_{\bf i}(\lambda)$ parameterize the dual canonical basis elements of the irreducible $G$-module $V_{\lambda}$ with highest weight $\lambda$.
In Appendix~\ref{appendix_automorphisms_and_string_polytopes}, we present the original definition of string polytopes following from Littelmann's paper~\cite{Li}.
 
The string polytope $\Delta_{\bf i}(\lambda)$ can be obtained as the intersection of two convex rational polyhedral cones, the {\em string cone} 
$C_{\bf i}$ and the {\em $\lambda$-cone} $C_{\bf i}^\lambda$.
There are several ways of describing the string cone $C_{\bf i}$, see \cite{Li}, \cite{BeZe01}, and \cite{GlPo00} for instance.
Throughout this paper, we follow Gleizer--Postnikov's description \cite{GlPo00} of the string cone $C_{\bf i}$ and Rusinko's description \cite{Rusinko08} of the $\lambda$-cone $C_{\bf i}^\lambda$
where both descriptions use the so-called {\em wiring diagrams}. We also refer to \cite[Sections~2 and~3]{CKLP} for more details.

\begin{definition} 
For a reduced word ${\bf i} = (i_1, \dots, i_\barn) \in \Rn{n+1}$, the {\em wiring diagram} $G({\bf i})$ is an arrangement
 of~$(n+1)$-vertical piecewise straight lines such that 
\begin{itemize}
	\item each pair of wires must intersect exactly once, and 
	\item the $j$th crossing of wires (from the top) is located in the $i_j$th column (from the left) of $G({\bf i})$ for each $j=1, 2, \dots, \barn$.
\end{itemize}
We call crossings {\em nodes} and label them as $t_1, t_2, \ldots, t_\barn$ from the top to the bottom. 
\end{definition}
In Figure~\ref{figure_wiring_diagram_GC}, we present wiring diagrams for reduced words $(1,2,1,3,2,1)$ and $(1,3,2,1,3,2)$ in~$\Rn{4}$.
	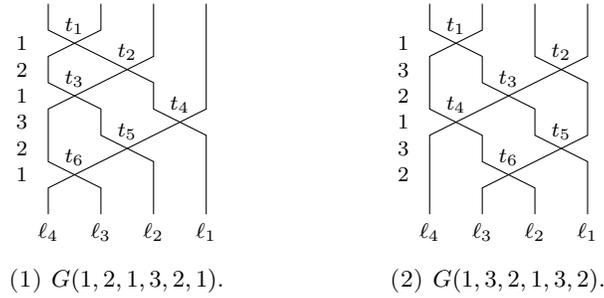
\begin{figure}[H]
\begin{subfigure}{0.3\textwidth}
	\centering
\begin{tikzpicture}[scale = 0.7]
		\tikzset{every node/.style = {font = \footnotesize}}
\tikzset{red line/.style = {line width=0.5ex, red, semitransparent}}
\tikzset{blue line/.style = {line width=0.5ex, blue, semitransparent}}

\draw (0,-0.5)--(0,0)--(3,1.5)--(3,3.5);
\draw (1,-0.5)--(1,0)--(0,0.5)--(0,1.5)--(2,2.5)--(2,3.5);
\draw (2,-0.5)--(2,0.5)--(1,1)--(1,1.5)--(0,2)--(0,2.5)--(1,3)--(1,3.5);
\draw (3,-0.5)--(3,1)--(2,1.5)--(2,2)--(0,3)--(0,3.5);

	\node[below] at (0,-0.5) {$\ell_4$};
\node[below] at (1,-0.5) {$\ell_3$};
\node[below] at (2,-0.5) {$\ell_2$};
\node[below] at (3,-0.5) {$\ell_1$};

\node[above] at (0.5,2.75) {$t_1$};
\node[above] at (1.5, 2.25) {$t_2$};
\node[above] at (0.5,1.75) {$t_3$};
\node[above] at (2.5, 1.25) {$t_4$};
\node[above] at (1.5,0.75) {$t_5$};
\node[above] at (0.5,0.25) {$t_6$};

 	\node at (-0.5, 2.75) {$1$};
 \node at (-0.5, 2.25) {$2$};
 \node at (-0.5, 1.75) {$1$};
 \node at (-0.5, 1.25) {$3$};
 \node at (-0.5, 0.75) {$2$};
 \node at (-0.5,0.25) {$1$};
\end{tikzpicture}
\caption{$G(1,2,1,3,2,1)$.}
\end{subfigure}
\begin{subfigure}{0.3\textwidth}
	\centering
	\begin{tikzpicture}[scale = 0.7]
		\tikzset{every node/.style = {font = \footnotesize}}
	\tikzset{red line/.style = {line width=0.5ex, red, semitransparent}}
	\tikzset{blue line/.style = {line width=0.5ex, blue, semitransparent}}
	
	
	\draw (0, -0.5)--(0,0)--(0,1)--(3,2.5)--(3,3.5);
	\draw (1,-0.5)--(1,0)--(3,1)--(3,2)--(2,2.5)--(2,3.5);
	\draw (2,-0.5)--(2,0)--(1,0.5)--(1,1)--(0,1.5)--(0,2.5)--(1,3)--(1,3.5);
	\draw (3,-0.5)--(3,0)--(3,0.5)--(2,1)--(2,1.5)--(1,2)--(1,2.5)--(0,3)--(0,3.5);

	\node[below] at (0,-0.5) {$\ell_4$};
	\node[below] at (1,-0.5) {$\ell_3$};
	\node[below] at (2,-0.5) {$\ell_2$};
	\node[below] at (3,-0.5) {$\ell_1$};
	
	\node[above] at (0.5, 2.75) {$t_1$};
	\node[above] at (2.5, 2.25) {$t_2$};
	\node[above] at (1.5, 1.75) {$t_3$};
	\node[above] at (0.5, 1.25) {$t_4$};
	\node[above] at (2.5, 0.75) {$t_5$};
	\node[above] at (1.5, 0.25) {$t_6$};
	
	\node at (-0.5, 2.75) {$1$};
	\node at (-0.5, 2.25) {$3$};
	\node at (-0.5, 1.75) {$2$};
	\node at (-0.5, 1.25) {$1$};
	\node at (-0.5, 0.75) {$3$};
	\node at (-0.5,0.25) {$2$};
	
	\end{tikzpicture}
\caption{$G(1,3,2,1,3,2)$.}
\end{subfigure}
		\caption{\label{figure_wiring_diagram_GC} Wiring diagrams for $(1,2,1,3,2,1)$ and $(1,3,2,1,3,2)$.}	
	\end{figure}

\noindent
We label the wires $\ell_1, \ell_2, \ldots , \ell_{n+1}$ and call $\ell_k$ the \defi{$k$th wire}. 
Also the upper end and lower end of each wire $\ell_k$ are labeled by $U_k$ and $L_k$, respectively.

\begin{definition}[{\cite[Section 5.1]{GlPo00}}]\label{definition_rigorous_path}
	For a given ${\bf i} \in \Rn{n+1}$ and $k \in [n]$, define $G({\bf i}, k)$ to be the oriented wiring diagram $G({\bf i})$, where the orientations on each wire is given such that 
	\begin{itemize}
		\item the first $k$ wires $\ell_1, \ldots, \ell_k$ are oriented upward, and 
		\item the other wires $\ell_{k+1}, \ldots, \ell_{n+1}$ are oriented downward. 
	\end{itemize}
	\begin{enumerate}
		\item A {\em rigorous path} is an oriented path on $G({\bf i}, k)$ for some $k \in [n]$ satisfying
		\begin{itemize}
				\item it starts at $L_k$ and ends at $L_{k+1}$, 
				\item it respects the orientation of $G({\bf i}, k)$,
				\item it passes through each node at most once, and
				\item it does \emph{not} include {\em forbidden fragments} given in Figure~\ref{figure_avoiding}:
			\end{itemize}

	\begin{figure}[h]
		\scalebox{0.92}{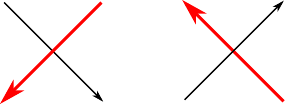}
			\vspace{-0.1cm}
		\caption{\label{figure_avoiding} Forbidden fragments.}
	\end{figure}
	
	\noindent
	We denote the set of rigorous paths by $\mathcal{GP}({\bf i})$.  \vs{0.2cm}
	
	\item A node $t$ is called a {\em peak} of a rigorous path $\p \in \mcal{GP}({\bf i})$ if $t$ is a local maximal node of the path $\p$ with respect to the height of the diagram $G({\bf i})$.
	
	\item Among peaks of a rigorous path $\p \in \GP(\mathbf i)$, we call the global maximal node of the path $\p$ the \defi{maximal peak}.
	
	\item Each node $t_j$ assigns the {\em chamber} defined to be the region $\cham{j}$ enclosed by wires such that 
	\begin{itemize}
		\item $t_j$ is the unique peak of the boundary of $\cham{j}$, and 
		\item any wire does not intersect the interior of $\cham{j}$.
	\end{itemize}	
	\vs{0.2cm}
	
	\item For each chamber $\cham{j}$, define a new variable 
	\[
		m_j \coloneqq \sum_{i=1}^\barn a_i t_i, \quad \quad a_i = \begin{cases}
			1 & \text{if $t_i \in \cham{j}$ is in the same column as $t_j$}, \\
			-1 & \text{if $t_i \in \cham{j}$ is in one column to the right or left of $t_j$}, \\
			0 & \text{otherwise.}
		\end{cases}
	\]
	We call $m_j$'s {\em chamber variables}. See also \cite[Definition 4.1]{CKLP}.
	\end{enumerate}
\end{definition}
	
	\noindent
	We note that a forbidden fragment can appear only when $\ell_i$ crosses over $\ell_j$ such that 
		\begin{itemize}
			\item $i > j$ where the orientation of both wires is downward, or 
			\item $i < j$ where the orientation of both wires is upward.  
		\end{itemize}

	Any rigorous path $\p$ can be expressed by
	\begin{equation}\label{equation_wire_expression}
		\ell_{r_1} \rightarrow \cdots \rightarrow \ell_{r_{p+1}} \quad \quad \quad r_1 = k, \quad r_{p+1} = k+1
	\end{equation}
	which records wires in order appearing in the travel along the path.
	The expression \eqref{equation_wire_expression} is called a {\em wire-expression}\footnote{There is another type of expression, called a \textit{node-expression} of a rigorous path. 
	See~\cite[(2.2) in Section~2.1]{CKLP}.} of $\p$. We denote by
	\[
		\mathrm{node}(\p) = \{ \ell_{r_i} \cap \ell_{r_{i+1}} \mid i=1,\dots, p \}
	\]
	the set of nodes on $\p$ appearing at the intersections of consecutive wires in \eqref{equation_wire_expression}.
	
	\begin{example}
	In Figure~\ref{figure_wd_oriented}, one can find some oriented wiring diagrams and rigorous paths for $\mathbf i = (1,2,1,3,2,1)$ and $\mathbf{i}' = (1,3,2,1,3,2)$. Also, one can find chambers for the word $(1,3,2,1,3,2)$ in Figure~\ref{chamber_132132}.
		\begin{figure}[H]
		\begin{subfigure}{0.35\textwidth}
			\centering
			\begin{tikzpicture}[scale = 0.7]
			\tikzset{every node/.style = {font = \footnotesize}}
			\tikzset{red line/.style = {line width=0.5ex, red, semitransparent}}
			\tikzset{blue line/.style = {line width=0.5ex, blue, semitransparent}}
			
			\draw[<-] (0,-0.5)--(0,0)--(3,1.5)--(3,3.5);
			\draw [<-](1,-0.5)--(1,0)--(0,0.5)--(0,1.5)--(2,2.5)--(2,3.5);
			\draw[<-] (2,-0.5)--(2,0.5)--(1,1)--(1,1.5)--(0,2)--(0,2.5)--(1,3)--(1,3.5);
			\draw[blue, thick, ->] (3,-0.5)--(3,1)--(2,1.5)--(2,2)--(0,3)--(0,3.5);
			
			\node[below] at (0,-0.5) {$\ell_4$};
			\node[below] at (1,-0.5) {$\ell_3$};
			\node[below] at (2,-0.5) {$\ell_2$};
			\node[below] at (3,-0.5) {$\ell_1$};
			
			\node[above] at (0.5,2.75) {$t_1$};
			\node[above] at (1.5, 2.25) {$t_2$};
			\node[above] at (0.5,1.75) {$t_3$};
			\node[above] at (2.5, 1.25) {$t_4$};
			\node[above] at (1.5,0.75) {$t_5$};
			\node[above] at (0.5,0.25) {$t_6$};
			
			\draw[red line, ->] (3,-0.5)--(3,1)--(2,1.5)--(2,2)--(1.5,2.25)--(0.5,1.75)--(1,1.5)--(1,1)--(2,0.5)--(2,-0.5);
			
			\end{tikzpicture}
			\begin{tikzpicture}[scale = 0.7]
			\tikzset{every node/.style = {font = \footnotesize}}
			\tikzset{red line/.style = {line width=0.5ex, red, semitransparent}}
			\tikzset{blue line/.style = {line width=0.5ex, blue, semitransparent}}
			
			\draw[<-] (0,-0.5)--(0,0)--(3,1.5)--(3,3.5);
			\draw[<-] (1,-0.5)--(1,0)--(0,0.5)--(0,1.5)--(2,2.5)--(2,3.5);
			\draw[blue, thick, ->] (2,-0.5)--(2,0.5)--(1,1)--(1,1.5)--(0,2)--(0,2.5)--(1,3)--(1,3.5);
			\draw[blue, thick, ->] (3,-0.5)--(3,1)--(2,1.5)--(2,2)--(0,3)--(0,3.5);
			
			\node[below] at (0,-0.5) {$\ell_4$};
			\node[below] at (1,-0.5) {$\ell_3$};
			\node[below] at (2,-0.5) {$\ell_2$};
			\node[below] at (3,-0.5) {$\ell_1$};
			
			\node[above] at (0.5,2.75) {$t_1$};
			\node[above] at (1.5, 2.25) {$t_2$};
			\node[above] at (0.5,1.75) {$t_3$};
			\node[above] at (2.5, 1.25) {$t_4$};
			\node[above] at (1.5,0.75) {$t_5$};
			\node[above] at (0.5,0.25) {$t_6$};
			
			\draw[red line, ->] (2,-0.5)--(2,0.5)--(1.5,0.75)--(0.5,0.25)--(1,0)--(1,-0.5);
			
			\end{tikzpicture}
			\caption{Oriented wiring diagrams $G(\mathbf i,1)$ and $G(\mathbf i,2)$; rigorous paths $\ell_1 \to \ell_3 \to \ell_2$ and $\ell_2 \to \ell_4 \to \ell_3$.}
		\end{subfigure}\hspace{1em}%
		\begin{subfigure}{0.35\textwidth}
			\centering
			\begin{tikzpicture}[scale = 0.7]
			\tikzset{every node/.style = {font = \footnotesize}}
			\tikzset{red line/.style = {line width=0.5ex, red, semitransparent}}
			\tikzset{blue line/.style = {line width=0.5ex, blue, semitransparent}}
			
			
			\draw[<-] (0, -0.5)--(0,0)--(0,1)--(3,2.5)--(3,3.5);
			\draw[<-] (1,-0.5)--(1,0)--(3,1)--(3,2)--(2,2.5)--(2,3.5);
			\draw[<-] (2,-0.5)--(2,0)--(1,0.5)--(1,1)--(0,1.5)--(0,2.5)--(1,3)--(1,3.5);
			\draw[->, thick, blue] (3,-0.5)--(3,0)--(3,0.5)--(2,1)--(2,1.5)--(1,2)--(1,2.5)--(0,3)--(0,3.5);

			\node[below] at (0,-0.5) {$\ell_4$};
			\node[below] at (1,-0.5) {$\ell_3$};
			\node[below] at (2,-0.5) {$\ell_2$};
			\node[below] at (3,-0.5) {$\ell_1$};
			
			\node[above] at (0.5, 2.75) {$t_1$};
			\node[above] at (2.5, 2.25) {$t_2$};
			\node[above] at (1.5, 1.75) {$t_3$};
			\node[above] at (0.5, 1.25) {$t_4$};
			\node[above] at (2.5, 0.75) {$t_5$};
			\node[above] at (1.5, 0.25) {$t_6$};

			\draw[red line, ->] (3,-0.5)--(3,0.5)--(2,1)--(2,1.5)--(1.5,1.75)--(0.5,1.25)--(1,1)--(1,0.5)--(2,0)--(2,-0.5);
			
			\end{tikzpicture}
			\begin{tikzpicture}[scale = 0.7]
			\tikzset{every node/.style = {font = \footnotesize}}
			\tikzset{red line/.style = {line width=0.5ex, red, semitransparent}}
			\tikzset{blue line/.style = {line width=0.5ex, blue, semitransparent}}
			
			
			\draw[<-] (0, -0.5)--(0,0)--(0,1)--(3,2.5)--(3,3.5);
			\draw[thick, blue, ->] (1,-0.5)--(1,0)--(3,1)--(3,2)--(2,2.5)--(2,3.5);
			\draw[thick, blue, ->] (2,-0.5)--(2,0)--(1,0.5)--(1,1)--(0,1.5)--(0,2.5)--(1,3)--(1,3.5);
			\draw[thick, blue, ->] (3,-0.5)--(3,0)--(3,0.5)--(2,1)--(2,1.5)--(1,2)--(1,2.5)--(0,3)--(0,3.5);

			\node[below] at (0,-0.5) {$\ell_4$};
			\node[below] at (1,-0.5) {$\ell_3$};
			\node[below] at (2,-0.5) {$\ell_2$};
			\node[below] at (3,-0.5) {$\ell_1$};
			
			\node[above] at (0.5, 2.75) {$t_1$};
			\node[above] at (2.5, 2.25) {$t_2$};
			\node[above] at (1.5, 1.75) {$t_3$};
			\node[above] at (0.5, 1.25) {$t_4$};
			\node[above] at (2.5, 0.75) {$t_5$};
			\node[above] at (1.5, 0.25) {$t_6$};
			
			\draw[red line, ->] (1,-0.5)--(1,0)--(2.5,0.75)--(2,1)--(2,1.5)--(1.5,1.75)--(0,1)--(0,-0.5);
			\end{tikzpicture}	
			\caption{Oriented wiring diagrams $G(\mathbf i', 1)$ and $G(\mathbf i', 3)$; rigorous paths $\ell_1 \to \ell_4 \to \ell_2$ and $\ell_3 \to \ell_1 \to \ell_4$.}\label{fig_132132}
		\end{subfigure}\hspace{1em}
		\begin{subfigure}{0.19\textwidth}
				\begin{tikzpicture}[scale = 0.7]
			\tikzset{every node/.style = {font = \footnotesize}}
			\tikzset{red line/.style = {line width=0.5ex, red, semitransparent}}
			\tikzset{blue line/.style = {line width=0.5ex, blue, semitransparent}}
			
			
			\draw (0, -0.5)--(0,0)--(0,1)--(3,2.5)--(3,3.5);
			\draw (1,-0.5)--(1,0)--(3,1)--(3,2)--(2,2.5)--(2,3.5);
			\draw (2,-0.5)--(2,0)--(1,0.5)--(1,1)--(0,1.5)--(0,2.5)--(1,3)--(1,3.5);
			\draw (3,-0.5)--(3,0)--(3,0.5)--(2,1)--(2,1.5)--(1,2)--(1,2.5)--(0,3)--(0,3.5);

			\node[below] at (0,-0.5) {$\ell_4$};
			\node[below] at (1,-0.5) {$\ell_3$};
			\node[below] at (2,-0.5) {$\ell_2$};
			\node[below] at (3,-0.5) {$\ell_1$};
			
			\node[above] at (0.5, 2.75) {$t_1$};
			\node[above] at (2.5, 2.25) {$t_2$};
			\node[above] at (1.5, 1.75) {$t_3$};
			\node[above] at (0.5, 1.25) {$t_4$};
			\node[above] at (2.5, 0.75) {$t_5$};
			\node[above] at (1.5, 0.25) {$t_6$};
			
			\node at (0.5, 2.25) {\textcolor{blue}{\textbf{\normalsize $\cham{1}$}}};
				\node at (2.5, 1.75) {\textcolor{blue}{\textbf{\normalsize $\cham{2}$}}};
	\node at (1.5, 1.25) {\textcolor{blue}{\textbf{\normalsize $\cham{3}$}}};
	\node at (0.5, 0.75) {\textcolor{blue}{\textbf{\normalsize $\cham{4}$}}};			
	\node at (2.5, 0.25) {\textcolor{blue}{\textbf{\normalsize $\cham{5}$}}};	
	\node at (1.5, -0.25) {\textcolor{blue}{\textbf{\normalsize $\cham{6}$}}};
			\end{tikzpicture}
\caption{Chambers $\cham{j}$ for~$\mathbf i'$.}
\label{chamber_132132}
		\end{subfigure}
		
		\caption{\label{figure_wd_oriented} Oriented wiring diagrams for $\mathbf i = (1,2,1,3,2,1)$ and $\mathbf i ' = (1,3,2,1,3,2)$, and chambers.}
	\end{figure}
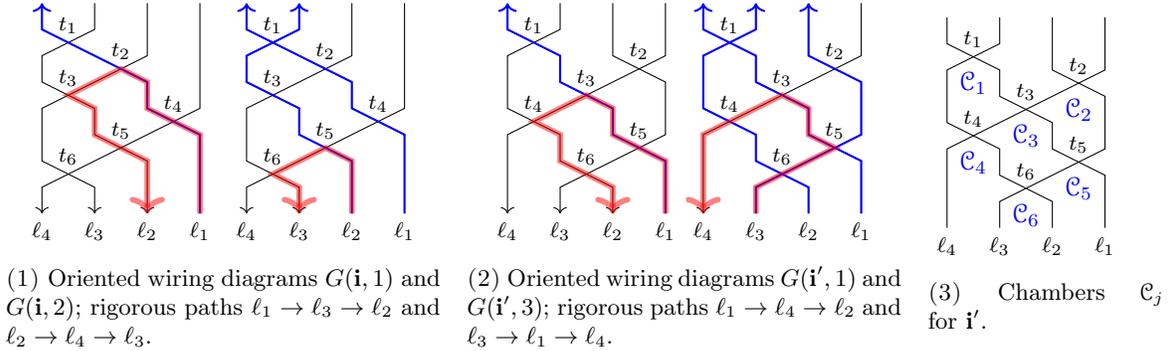	

\end{example} 

	Now we are ready to define $C_{\bf i}$, $C_{\bf i}^\lambda$, and $\Delta_{\bf i}(\lambda)$ each of which is a convex object in $\R^\barn$. 
	We will use the coordinate system $(t_1, \dots, t_\barn)$ by abuse of notation.

	\begin{definition}[{\cite{GlPo00, Rusinko08}}]\label{definition_stringcone_lambdacone_stringpolytope}
	Let ${\bf i} = (i_1, \dots, i_\barn) \in \Rn{n+1}$ and $\lambda = \lambda_1 \varpi_1 + \cdots + \lambda_n \varpi_n \in \Lambda_{+}$. 
	
	\begin{enumerate}
		\item Let $\p$ be a rigorous path in $G({\bf i}, k)$ for some $k \in [n]$. The {\em string inequality for $\p$} is defined by 
		\[
			\sum_{j=1}^{\barn} a_j t_j \geq 0, \quad \quad \text{ where }a_j \coloneqq \begin{cases}
				1 & \text{if $\p$ travels from $\ell_r$ to $\ell_s$ at $t_j$ and $r < s$}, \\ 
				-1 & \text{if $\p$ travels from $\ell_r$ to $\ell_s$ at $t_j$ and $r > s$}, \\ 
				 0 & \text{otherwise}.
			\end{cases}
		\]
		The {\em string cone} $C_{\bf i}$ is the set of points in $\R^\barn$ satisfying all string inequalities. \vs{0.2cm}
		
		\item 	For each node $t_j$ in $G({\bf i})$, the {\em $\lambda$-inequality for $t_j$} is defined by 
		\[
			t_j \leq \lambda_{i_j} + \sum_{k > j} b_k t_k, \quad \quad \text{ where } b_k \coloneqq \begin{cases}
				1 & \text{if the node $t_k$ is in one column to the right or left of $t_j$}, \\
				-2 & \text{if the node $t_k$ is in  the same column as $t_j$}, \\
				0 & \text{otherwise}.
			\end{cases}
		\]
		The {\em $\lambda$-cone} $C_{\bf i}^\lambda$ is the set of points in $\R^\barn$ satisfying all $\lambda$-inequalities. \vs{0.2cm}
	\end{enumerate}	
	\end{definition}
	
	\begin{remark}[{\cite[Section 4.1]{CKLP}\label{remark_chamber_variable}}]
			In terms of chamber variables, the description of the string polytope becomes much simpler.
			Under the change of coordinates $(t_1, \dots, t_\barn) \rightarrow (m_1, \dots, m_\barn)$, we may describe the string cone $C_{\bf i}$ as follows. 
			\[
				C_{\bf i} = \left\{ (m_1, \dots, m_\barn) \in \R^\barn ~\Bigg|~  \sum_{\cham{j} \subset \text{region enclosed by $\p$}} m_j \geq 0, \quad \p \in \mcal{GP}({\bf i}) \right\}.
			\]
			Similarly, the $\lambda$-cone can be described by
			\[
				C_{\bf i}^\lambda = \left\{ (m_1, \dots, m_\barn) \in \R^\barn ~\Bigg|~ - \sum_{k \geq j, \, i_k = i_j} m_k + \lambda_{i_j} \geq 0, \quad j = 1, 2, \dots, \barn \right\}.
			\]
	Note that in~\cite{BF}, they also used a similar description of string polytopes (see, for instance, \cite[Figure~6]{BF}).
	\end{remark}
	
	\begin{definition}\label{def_string_polytopes}
			Let ${\bf i} = (i_1, \dots, i_\barn) \in \Rn{n+1}$ and $\lambda = \lambda_1 \varpi_1 + \cdots + \lambda_n \varpi_n \in \Lambda_{+}$. 
		The \defi{string polytope~$\Delta_{\mathbf i}(\lambda)$} is defined as the intersection of the string cone and the $\lambda$-cone. In terms of the chamber variables~$m_j$'s, 
		the string polytope can be written by
		\[
			\Delta_{\mathbf i}(\lambda) \coloneqq C_{\bf i} \cap C_{\bf i}^\lambda = \bigcap_{\p \in \GP(\mathbf i)} \{\mathbf m \in \R^\barn \mid \langle \mathbf{w}_{\p}, \mathbf m \rangle \geq 0  \} 
			\cap \bigcap_{j=1}^\barn \{ \mathbf m \in 	\R^\barn \mid \langle \mathbf{v}_j , \mathbf m \rangle + \lambda_{i_j} \geq 0 \},
		\]
		where ${\bf w}_{\p}$ and ${\bf v}_j$ denote the coefficient vectors of the string inequality for $\p$ and the $\lambda$-inequality for $m_j$
		as in Remark~\ref{remark_chamber_variable}, respectively. Indeed,
		\begin{equation}\label{eq_wp_and_vj}
		\mathbf{w}_{\p} = \sum_{\cham{j} \subset \text{region enclosed by $\p$}} \mathbf{e}_j, \qquad 
		\mathbf{v}_j = - \sum_{k \geq j, \, i_k = i_j} \mathbf{e}_k,
		\end{equation}
		where $\{\mathbf{e}_1,\dots, \mathbf{e}_{\barn}\}$ is the standard basis of $\mathbb{R}^{\barn}$. 
	\end{definition}

\begin{example}\label{example_132132}
	Let $\mathbf i = (1,3,2,1,3,2)\in \Rn{4}$. Then there are seven rigorous paths, and each path $\p$ defines the following vector $\bfw_{\p}$: (See Figure~\ref{chamber_132132}.)
	\[
	\begin{split}
	&\bfw_{\ell_1 \to \ell_2} = (1,0,1,0,1,0), \quad \bfw_{\ell_1 \to \ell_4 \to \ell_2} = (0,0,1,0,1,0), 
	\quad \bfw_{\ell_1 \to \ell_3 \to \ell_2} = (0,0,0,0,1,0), \\
	&\bfw_{\ell_2 \to \ell_3} = (0,0,0,0,0,1),  \\
	& \bfw_{\ell_3 \to \ell_4} = (0,1,1,1,0,0), \quad \bfw_{\ell_3 \to \ell_1 \to \ell_4} = (0,0,1,1,0,0), \quad
	\bfw_{\ell_3 \to \ell_2 \to \ell_4} = (0,0,0,1,0,0).
	\end{split}
	\]
	On the other hand, we have the vectors $\bfv_j$ for $1 \leq j \leq 6$:
	\[
	\begin{array}{lll}
		\bfv_1 = (-1,0,0,-1,0,0), & \bfv_2 = (0,-1,0,0,-1,0),& \bfv_3 = (0,0,-1,0,0,-1),\\
	\bfv_4 = (0,0,0,-1,0,0), & \bfv_5 = (0,0,0,0,-1,0),& \bfv_6 = (0,0,0,0,0,-1).	
	\end{array}
	\]
	Therefore, for $\lambda = \lambda_1 \varpi_1 + \lambda_2 \varpi_2 + \lambda_3 \varpi_3$, the string polytope $\Delta_{\mathbf i}(\lambda)$ is expressed as follows.
	\[
	\Delta_{\mathbf i}(\lambda) = \left\{
	(m_1,\dots,m_6) \in \R^6  \left\vert
	\begin{array}{l}
	m_1 + m_3 + m_5 \geq 0, \quad m_3 + m_5 \geq 0, \quad m_5 \geq 0, \quad m_6 \geq 0, \\
	m_2 + m_3 + m_4 \geq 0, \quad m_3 + m_4 \geq 0, \quad m_4 \geq 0, \\
	-m_1 - m_4 + \lambda_1 \geq 0, \quad - m_2 - m_5 + \lambda_3 \geq 0, \quad -m_3 - m_6 + \lambda_2 \geq 0, \\
	-m_4 + \lambda_1 \geq 0, \quad -m_5 + \lambda_3 \geq 0, \quad -m_6 + \lambda_2 \geq 0
	\end{array} \right.
	\right\}.
	\]
\end{example}

In the rest of the section, we observe some combinatorial properties of rigorous paths which will be used later.
\begin{proposition}[{\cite[Proposition~4.6]{CKLP}}]\label{prop_non_redundancy_inequalities}
	Let $\lambda$ be a regular dominant weight and $\mathbf i \in \Rn{n+1}$. 
	Then the expression in Definition~\ref{def_string_polytopes} is non-reduandant in the string polytope $\Delta_{\mathbf i}(\lambda)$. Indeed, when we consider the normal fan $\Sigma_{\Delta_{\mathbf i}(\lambda)}$ of the string polytope, the set of ray generators of the fan $\Sigma_{\Delta_{\mathbf i}(\lambda)}$ is the same as 
	\[
	\{ \bfw_{\p} \mid \p \in \GP(\mathbf i)\} \cup \{ \bfv_j \mid j =1,\dots,\barn\}.
	\]
\end{proposition}

Let $R_i$ be the closed region enclosed by the path $\ell_i \rightarrow \ell_{i+1}$ and $R_i^{\circ}$ its interior for each $i \in [n]$. 
See Figure~\ref{fig_DD_R} for the regions $R_i$ of the word $
(4,3,4,2,3,4,1,2,3,4,5,4,6,5,4,3,2,1,4,3,2) \in \Rn{7}$.
\begin{figure}[H]
	\centering

		\begin{tikzpicture}[scale = 0.5]
	\tikzset{every node/.style = {font = \footnotesize}}
	\tikzset{red line/.style = {line width=0.5ex, red, semitransparent}}
	\tikzset{blue line/.style = {line width=0.5ex, blue, semitransparent}}
	
	\draw(0,-0.5)--(0,0) -- (0,1.5)--(6, 4.5)--(6,11);
	\draw(1,-0.5)--(1,0)--(4,1.5)--(4,3)--(3,3.5)--(3,4.5)--(5,5.5)--(5,11);
	\draw(2,-0.5)--(2,0)--(1,0.5)--(1,1.5)--(0,2)--(0,7)--(1,7.5)--(1,8.5)--(2,9)--(2,9.5)--(4,10.5)--(4,11);
	\draw (3,-0.5)--(3,0)--(3,0.5)--(2,1)--(2,2)--(1,2.5)--(1,6.5)--(2,7)--(2,8)--(3,8.5)--(3,9)--(4,9.5)--(4,10)--(3,10.5)--(3,11);
	\draw(4,-0.5)--(4,0)--(4,1)--(3,1.5)--(3,2.5)--(2,3)--(2,6)--(3,6.5)--(3,7.5)--(4,8)--(4,9)--(2,10)--(2,11);
	\draw (5,-0.5)--(5,3.5)--(4,4)--(4,4.5)--(3,5)--(3,5.5)--(4,6)--(4,7.5)--(1,9)--(1,11);
	\draw (6,-0.5)--(6,4)--(5,4.5)--(5,5)--(0,7.5)--(0,11);
	
	\node[below] at (0,-0.5) {$\ell_7$};
	\node[below] at (1,-0.5) {$\ell_6$};
	\node[below] at (2,-0.5) {$\ell_5$};
	\node[below] at (3,-0.5) {$\ell_4$};
	\node[below] at (4,-0.5) {$\ell_3$};
	\node[below] at (5,-0.5) {$\ell_2$};
	\node[below] at (6,-0.5) {$\ell_1$};

%
%
%
%
%

	\node[above] at (1.5,0.25) {\footnotesize $t_{21}$};
	\node[above] at (2.5,0.75) {\footnotesize $t_{20}$};
	\node[above] at (3.5,1.25) {\footnotesize $t_{19}$};
	\node[above] at (0.5,1.75) {\footnotesize $t_{18}$};
	\node[above] at (1.5,2.25) {\footnotesize $t_{17}$};
	\node[above] at (2.5,2.75) {\footnotesize $t_{16}$};
	\node[above] at (3.5,3.25) {\footnotesize $t_{15}$};
	\node[above] at (4.5,3.75) {\footnotesize $t_{14}$};
	\node[above] at (5.5,4.25) {\footnotesize $t_{13}$};
	\node[label={[label distance=-1.5mm]:{\footnotesize $t_{12}$}}] at (3.5,4.75) {};
	\node[above] at (4.5,5.25) {\footnotesize $t_{11}$};
	\node[above] at (3.5,5.75) {\footnotesize $t_{10}$};
	\node[above] at (2.5,6.25) {\footnotesize $t_{9}$};
	\node[above] at (1.5,6.75) {\footnotesize $t_{8}$};
	\node[above] at (0.5,7.25) {\footnotesize $t_{7}$};
	\node[above] at (3.5,7.75) {\footnotesize $t_{6}$};
	\node[above] at (2.5,8.25) {\footnotesize $t_{5}$};
	\node[above] at (1.5,8.75) {\footnotesize $t_{4}$};
	\node[above] at (3.5,9.25) {\footnotesize $t_{3}$};
	\node[above] at (2.5,9.75) {\footnotesize $t_{2}$};
	\node[above] at (3.5,10.25) {\footnotesize $t_{1}$};
	
	
	\end{tikzpicture}\hspace{1em}
	\begin{tikzpicture}[scale = 0.5]
	\tikzset{every node/.style = {font = \footnotesize}}
	\tikzset{red line/.style = {line width=0.5ex, red, semitransparent}}
	\tikzset{blue line/.style = {line width=0.5ex, blue, semitransparent}}
	
	\draw(0,-0.5)--(0,0) -- (0,1.5)--(6, 4.5)--(6,11);
	\draw(1,-0.5)--(1,0)--(4,1.5)--(4,3)--(3,3.5)--(3,4.5)--(5,5.5)--(5,11);
	\draw(2,-0.5)--(2,0)--(1,0.5)--(1,1.5)--(0,2)--(0,7)--(1,7.5)--(1,8.5)--(2,9)--(2,9.5)--(4,10.5)--(4,11);
	\draw (3,-0.5)--(3,0)--(3,0.5)--(2,1)--(2,2)--(1,2.5)--(1,6.5)--(2,7)--(2,8)--(3,8.5)--(3,9)--(4,9.5)--(4,10)--(3,10.5)--(3,11);
	\draw(4,-0.5)--(4,0)--(4,1)--(3,1.5)--(3,2.5)--(2,3)--(2,6)--(3,6.5)--(3,7.5)--(4,8)--(4,9)--(2,10)--(2,11);
	\draw (5,-0.5)--(5,3.5)--(4,4)--(4,4.5)--(3,5)--(3,5.5)--(4,6)--(4,7.5)--(1,9)--(1,11);
	\draw (6,-0.5)--(6,4)--(5,4.5)--(5,5)--(0,7.5)--(0,11);
	

	\fill[RedOrange!20, semitransparent] (6,-0.5)--(6,4)--(5,4.5)--(5,5)--(3.5,5.75)--(3,5.5)--(3,5)--(4,4.5)--(4,4)--(5,3.5)--(5,-0.5)--cycle;			
	\node at (5.5,-0.7) {$R_1$};
	
	\fill[pattern color = Magenta, pattern= north west lines, semitransparent] (5,-0.5)--(5,3.5)--(4,4)--(4,4.5)--(3,5)--(3,5.5)--(4,6)--(4,7.5)--(3.5,7.75)--(3,7.5)--(3,6.5)--(2,6)--(2,3)--(3,2.5)--(3,1.5)--(4,1)--(4,-0.5)--cycle;
	\node at (4.5,-0.7) {$R_2$};
	
	\fill[blue!20, semitransparent] (4,-0.5)--(4,1)--(3,1.5)--(3,2.5)--(2,3)--(2,6)--(3,6.5)--(3,7.5)--(4,8)--(4,9)--(3.5,9.25)--(3,9)--(3,8.5)--(2,8)--(2,7)--(1,6.5)--(1,2.5)--(2,2)--(2,1)--(3,0.5)--(3,-0.5)--cycle;
	\node at (3.5,-0.7) {$R_3$};
	
	\fill[ForestGreen!30, semitransparent]
	(3,-0.5)--(3,0.5)--(2,1)--(2,2)--(1,2.5)--(1,6.5)--(2,7)--(2,8)--(3,8.5)--(3,9)--(4,9.5)--(4,10)--(4,10)--(3.5,10.25)--(2,9.5)--(2,9)--(1,8.5)--(1,7.5)--(0,7)--(0,2)--(1,1.5)--(1,0.5)--(2,0)--(2,-0.5);
	\node at (2.5,-0.7) {$R_4$};
	
	\fill[yellow!40, semitransparent] (2,-0.5)--(2,0)--(1.5,0.25)--(1,0)--(1,-0.5);
	\node at (1.5,-0.7) {$R_5$};
	
	\filldraw[draw = brown, thick, pattern = horizontal lines, pattern color = Sepia, semitransparent] (1,-0.5)--(1,0)--(4,1.5)--(4,3)--(3.5,3.25)--(3,3)--(0,1.5)--(0,-0.5);
	\node at (0.5,-0.7) {$R_6$};
	
	\node[above, red] at (1.5,0.25) {\footnotesize {$t_{21}$}};
\node[above, red] at (2.5,0.75) {\footnotesize {$t_{20}$}};
\node[above, red] at (3.5,1.25) {\footnotesize {$t_{19}$}};
\node[above, red] at (0.5,1.75) {\footnotesize {$t_{18}$}};
\node[above, red] at (1.5,2.25) {\footnotesize $t_{17}$};
\node[above, red] at (2.5,2.75) {\footnotesize $t_{16}$};
\node[above, red] at (3.5,3.25) {\footnotesize $t_{15}$};
\node[above, red] at (4.5,3.75) {\footnotesize $t_{14}$};
\node[above, red] at (5.5,4.25) {\footnotesize $t_{13}$};
\node[label={[label distance=-1.5mm, red]:{\footnotesize $t_{12}$}}] at (3.5,4.75) {};
\node[above, red] at (4.5,5.25) {\footnotesize $t_{11}$};
\node[above, red] at (3.5,5.75) {\footnotesize $t_{10}$};
\node[above, red] at (2.5,6.25) {\footnotesize $t_{9}$};
\node[above, red] at (1.5,6.75) {\footnotesize $t_{8}$};
\node[above, red] at (0.5,7.25) {\footnotesize $t_{7}$};
\node[above, red] at (3.5,7.75) {\footnotesize $t_{6}$};
\node[above, red] at (2.5,8.25) {\footnotesize $t_{5}$};
\node[above, red] at (1.5,8.75) {\footnotesize $t_{4}$};
\node[above, red] at (3.5,9.25) {\footnotesize $t_{3}$};
\node[above, red] at (2.5,9.75) {\footnotesize $t_{2}$};
\node[above, red] at (3.5,10.25) {\footnotesize $t_{1}$};
	
	\end{tikzpicture}
	\caption{The wiring diagram and $R_i$ of $(4,3,4,2,3,4,1,2,3,4,5,4,6,5,4,3,2,1,4,3,2)$.}
	\label{fig_DD_R}
\end{figure}
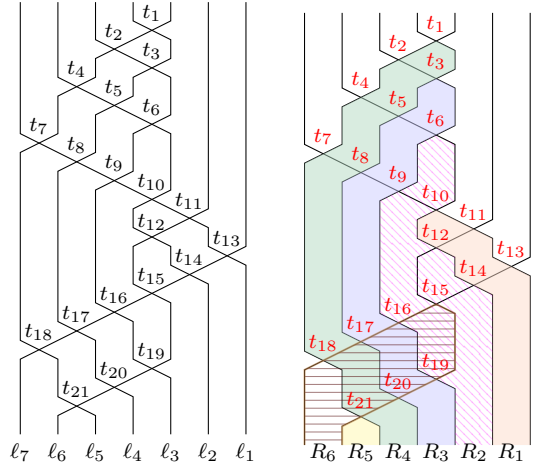
\noindent The regions $R_i$ can be expressed as  unions of chambers. For the word in Figure~\ref{fig_DD_R}, we have the following relation:
\[
\begin{array}{ll}
R_1 = \cham{10} \cup \cham{11} \cup \cham{13}, 
& R_2 = \cham{6} \cup \cham{9} \cup \cham{12} \cup \cham{14} \cup \cham{15}, \\
R_3 = \cham{3} \cup \cham{5} \cup \cham{8} \cup \cham{16} \cup \cham{19}, 
& R_4 = \cham{1} \cup \cham{2} \cup \cham{4} \cup \cham{7} \cup \cham{17} \cup \cham{20}, \\
R_5 = \cham{21}, 
& R_6 = \cham{15} \cup \cham{16} \cup \cham{17} \cup \cham{18}.
\end{array}
\]

By the definition of the rigorous paths and the region $R_i$, one can immediately obtain the following lemma.
\begin{lemma}\label{lemma_P_contatined_in_Ri}
	Let $\mathbf i \in \Rn{n+1}$ and $\p = (\ell_i \to \ell_{r_1} \to \cdots \to \ell_{r_s} \to \ell_{i+1}) \in \GP(\mathbf i)$. Then $\p \subset R_i$.
\end{lemma}

\section{Extensions and contractions on reduced words}
\label{secExtensionsAndContractions}

In this section, we introduce two operations, called a {\em contraction} and an {\em extension}, that produce a new reduced word in $\Rn{n}$ and $\Rn{n+2}$, respectively, for 
each reduced word ${\bf i} \in \Rn{n+1}$. See also~\cite[Section 3.3]{CKLP} for more details.

Recall that for a given ${\bf i} \in \Rn{n+1}$, one can produce a new reduced word in $\Rn{n+1}$ by a {\em braid move}. 
There are two types of braid moves:
	\begin{itemize}
		\item (2-move) exchanging $(i,j)$ with $(j,i)$ for $|i-j| > 1$, i.e., $s_i s_j = s_j s_i$.
		\item (3-move) exchanging $(i,i+1,i)$ with $(i+1,i,i+1)$, i.e., $s_i s_{i+1} s_i  = s_{i+1} s_i s_{i+1}$.
	\end{itemize}
According to Tits' Theorem \cite{Ti}, every pair of reduced words in $\Rn{n+1}$ is connected by a sequence of braid moves. 

Define an equivalence relation `$\sim$' on $\Rn{n+1}$ such that 
\[
	{\bf i} \sim {\bf i}' \quad \Leftrightarrow \quad \text{${\bf i}$ and ${\bf i}'$ are related by a sequence of 2-moves}.
\]
From the definition of string cones, it immediately follows that if two reduced words $\mathbf i$ and $\mathbf i'$ differ by a sequence $2$-moves, then the corresponding string cones differ by the change of coordinates. 
Hence, using the equivalence relation $\sim$, we can state the following.
\begin{lemma}[{\cite[Lemma~3.1]{CKLP}}]\label{lemma_2_move_unimodular_same_string_polytopes}
	If $\mathbf i \sim \mathbf i'$, then the string polytopes $\Delta_{\mathbf i}(\lambda)$ and $\Delta_{\mathbf i'}(\lambda)$ are the same up to coordinate changes for any dominant integral weight $\lambda \in \Lambda_{+}$.
\end{lemma}

The following proposition observed in \cite{CKLP} suggests two canonical representatives for each equivalence class in $\Rn{n+1} / \sim$.
See also \cite[Example 3.3]{CKLP}.

\begin{proposition}[{\cite[Proposition 3.2]{CKLP}}]\label{prop_index}
		For any ${\bf i} = (i_1, \dots, i_\barn) \in \Rn{n+1}$, we may rearrange ${\bf i}$ using 2-moves so that
		\[
			{\bf i} \sim (\underbrace{i_1', \dots, i_u'}_{=: ~{\bf i}_\D^-}, \underbrace{n, n-1, \dots, 2, 1}_{=: ~\D_n}, \underbrace{i_{u+n+1}', \dots, i_\barn'}_{=: ~{\bf i}^+_\D}) 		
		\]
		for some integer $u\geq 0$ and $i_j' \in [n]$. Similarly, there exist an integer $v \geq 0$ and $i_j''$'s in $[n]$ such that 
		\[
			{\bf i} \sim (\underbrace{i_1'', \dots, i_v''}_{=: ~{\bf i}_\A^-}, \underbrace{1, 2, \dots, n-1, n}_{=: ~\A_n}, \underbrace{i_{v+n+1}'', \dots, i_\barn''}_{=: ~{\bf i}^+_\A}),
		\]
		where `$\D$' and `$\A$' stand for `descending' and `ascending', respectively.
		Here, the descending chain $\D_n$ and ascending chain $\A_n$ are given by
		\[
		\D_n \coloneqq (n,n-1,\dots,2,1) \quad \text{ and } \quad \A_n \coloneqq (1,2,\dots,n-1,n).
		\]
\end{proposition}

	Using Proposition \ref{prop_index}, we can define two {\em indices} of ${\bf i}$ as follows. 

\begin{definition}[{\cite[Definition 3.4]{CKLP}}]\label{def_index}
	For each ${\bf i} \in \Rn{n+1}$ with 
	\[
		{\bf i} \quad \sim \quad {\bf i}^-_\D \D_n {\bf i}^+_\D \quad \sim \quad {\bf i}^-_\A \A_n {\bf i}^+_\A,
	\]
	define 
	\[
	\mathrm{ind}_\D({\bf i}) \coloneqq |{\bf i}^+_\D| \quad \text{ and } \quad \mathrm{ind}_\A({\bf i}) \coloneqq |{\bf i}^+_\A| 
	\]
	and call them the \defi{$\D$-index} of $\mathbf i$ and the \defi{$\A$-index} of $\mathbf i$, respectively.
\end{definition}

\begin{remark}\label{rm_index}
	Note that $\mathrm{ind}_\D({\bf i})$ and $\mathrm{ind}_\A({\bf i})$ count the number of nodes in $G({\bf i})$ below $\ell_{n+1}$ and $\ell_1$, respectively,
	as explained in the proof of \cite[Proposition 3.5]{CKLP}.
\end{remark}

We introduce some notations as follows. For a word $\mathbf i = (i_1,\dots,i_k)$, we denote by
\[
\mathbf i +1 \coloneqq (i_1 +1,\dots,i_k+1) \quad \text{ and } \quad
\mathbf i -1 \coloneqq (i_1 - 1, \dots, i_k-1).
\]
We also denote by $[a,b] \coloneqq \{a, a+1,\dots,b\}$ for $a,b \in \Z$. Note that there does not exist $1$ in $\bf i_\D^+$ because $\mathbf i$ is reduced. Similarly, there is no $1$ in 
$\mathbf{i}_{\A}^-$.

\begin{definition}[{\cite[Definition 3.6]{CKLP}}]\label{def_contraction}
	For any ${\bf i} \in \Rn{n+1}$ and $s \in \{0, 1, \dots, \bar{n} \}$, assume that 
	\[
		{\bf i} = (\underbrace{i_1, \dots, i_{\bar{n} - s}}_{=: {\bf i}^-(s)}, \underbrace{i_{\bar{n} - s + 1}, \dots, i_{\bar{n}}}_{=: {\bf i}^+(s)})
		\quad \sim \quad {\bf i}^-_\D \D_n {\bf i}^+_\D \quad \sim \quad {\bf i}^-_\A \A_n {\bf i}^+_\A,
	\]
	where both words ${\bf i}^-_\D \D_n {\bf i}^+_\D$ and ${\bf i}^-_\A \A_n {\bf i}^+_\A$ are {\em minimal} in the sense that 2-moves are used as little as possible to obtain
	them from ${\bf i}$, respectively.
	
	\begin{enumerate}
		\item The {\em $\D$-contraction} of ${\bf i}$, denoted by $C_\D({\bf i})$, is the reduced word ${\bf i}_\D^- ({\bf i}_\D^+ - 1) \in \Rn{n}$. \vs{0.1cm}
		\item The {\em $\A$-contraction} of ${\bf i}$, denoted by $C_\A({\bf i})$, is the reduced word $({\bf i}_\A^- - 1) {\bf i}_\A^+ \in \Rn{n}$. \vs{0.1cm}		
		\item For $s \in [0,\barn]$, the {\em $\D$-extension} of ${\bf i}$ at $s$, denoted by $E_\D(s)({\bf i})$, is the reduced word
		\[
			{\bf i}^-(s) ~\D_{n+1} ~({\bf i}^+(s) + 1) \in \Rn{n+2}.
		\]	
		\item For $s \in [0,\barn]$, the {\em $\A$-extension} of ${\bf i}$ at $s$, denoted by $E_\A(s)({\bf i})$, is the reduced word
		\[
			({\bf i}^-(s) +1)~\A_{n+1} ~{\bf i}^+(s) \in \Rn{n+2}.		
		\]
	\end{enumerate}	
\end{definition}

\begin{example}
	\begin{enumerate}
\item 	Suppose that $\mathbf i = (1,2,1,3,2,1)$. Then we have that $E_{\D}(3)(\mathbf i) = (1,2,1,4,3,2,1,4,3,2)$. 
\item Let $\mathbf i = (2,1,3,2,1,3)$. Then $\mathbf i \sim (2,3,1,2,3,1)$. Therefore, $\ind_{\A}(\mathbf i) = 1$ and $C_{\A}(\mathbf i) = (1,2,1)$. 
\item For any $\mathbf i \in \Rn{n+1}$, we have that 
\[
E_{\D}(0)(\mathbf i) = \mathbf i \ \D_{n+1}, \quad
E_{\D}(\barn)(\mathbf i) = \D_{n+1} (\mathbf i + 1).
\] 
\item By the definition, for any $s  \in \{0,1,\dots,\barn\}$, we have that $C_{\A}(E_{\A}(s)(\mathbf i)) = \mathbf i$ and $C_{\D}(E_{\D}(s)(\mathbf i)) = \mathbf i$.
	\end{enumerate}
\end{example}

\begin{remark}\label{rmk_extension_surjective}
	The extension map is surjective up to $2$-moves (see~\cite[Section~3.3]{CKLP}). Indeed, the following composition is surjective:
	\[
\begin{array}{ccccl}\vs{0.1cm}
\Rn{n+1} \times [0, \barn]& \stackrel{E_\bullet} \longrightarrow & \Rn{n+2} & \stackrel{\pi} \longrightarrow & \Rn{n+2}/\sim \\ \vs{0.1cm}
({\bf i}, s) & \mapsto & E_\bullet(s)({\bf i}) & \mapsto & \left[ E_\bullet(s)({\bf i}) \right].
\end{array}
\]
\end{remark}

Note that, starting with the empty set, we may apply extensions $n$ times repeatedly so that  one can get a reduced word (depending on the choice of $\A$ or $\D$ in each step) in $\Rn{n+1}$.
Similarly for each ${\bf i} \in \Rn{n+1}$ and the choice of $\A$ or $\D$ in each step, one can apply contractions repeatedly and finally get the empty set.
\begin{definition}\label{def_index_vector}
	\begin{enumerate}
\item For any sequence $\ad = (\ad_1,\dots,\ad_n) \in \{\A,\D\}^n$ and an integer vector $I = (I_1,\dots,I_n)$ with $I_i \in [0, \overline{i-1}]$, 
we obtain a reduced word 
\[
\is \coloneqq (E_{\ad_n}(I_n) \circ \cdots \circ  E_{\ad_2}(I_2) \circ E_{\ad_1}(I_1))(\emptyset) \in \Rn{n+1}.
\]
Here, $ \overline{i-1}= i(i-1)/2$ is given in \eqref{eq_def_of_nbar}.
\item For a given sequence $\ad = (\ad_1,\dots,\ad_n) \in \{\A,\D\}^n$, we define an integer vector $\mathrm{ind}_\ad({\bf i}) = (I_1,\dots,I_n) \in \Z^n_{\geq 0}$
to be 
\[
I_n \coloneqq \mathrm{ind}_{\ad_n}({\bf i}), \quad I_{n-1} \coloneqq \mathrm{ind}_{\ad_{n-1}}(C_{\ad_n}({\bf i})), \quad \dots,
\quad I_1 \coloneqq \ind_{\ad_1}(C_{\ad_{ 2}} \circ  \cdots \circ C_{\ad_n}({\bf i})).
\]
We call the vector $\mathrm{ind}_\ad({\bf i})$ the {\em $\ad$-index} of ${\bf i}$. 
\end{enumerate}
\end{definition}

\begin{example}
	Let $\ad = (\D,\D,\D)$ and $I = (0,0,2)$. Then, we have that
	\[
	\is = (E_\D(2) \circ E_\D(0) \circ E_\D(0))(\emptyset)
	= E_\D(2) (E_\D(0) (1)) = E_\D(2)(1,2,1) = (1,3,2,1,3,2).
	\]
\end{example}
From the definition we have the following property:
\begin{lemma}\label{lemma_index_of_i_sigma_I}
For a sequence $\ad \in \{\A,\D\}^n$ and $I = (I_1,\dots,I_n)$ with $I_i \in [0, \overline{i-1}]$, if $\mathbf i \sim \mathbf{i}_{\ad}(I)$, we have
\[
\ind_{\ad}(\mathbf i) = I.
\]
\end{lemma}

	Note that for two equivalent words $\mathbf i$ and  $\mathbf i'$, their extensions do not need to be equivalent. For instance, let us consider
	\[
	(1,2,1,3,2,1) \sim (1,2,3,1,2,1).
	\]
	On the other hand, we have 
	\[
	E_{\D}(3)(1,2,1,3,2,1) = (1,2,1,4,3,2,1,4,3,2)
	\not\sim E_\D(3)(1,2,3,1,2,1) = (1,2,3,4,3,2,1,2,3,2).
	\]
	It is not hard to check that 
	\[
	\ind_{(\D,\D,\D,\D)}(1,2,1,4,3,2,1,4,3,2) = 
	\ind_{(\D,\D,\D,\D)}(1,2,3,4,3,2,1,2,3,2) = (0,0,0,3).
	\]
	Therefore, for a given sequence $\ad \in \{\A,\D\}^n$ and $\mathbf i, \mathbf i' \in \Rn{n+1}$, $\ind_{\ad}(\mathbf i ) = \ind_{\ad}(\mathbf i')$ does not imply that two words are the same up to $2$-moves.

Now, we investigate how the set of rigorous paths enlarges by extensions or contractions. 
\begin{lemma}[{\cite[Lemma 5.1]{CKLP}}]\label{lem_inclusion}
	For a given ${\bf i} \in \Rn{n+1}$, $s \in [0,\bar{n}]$, and $\bullet = \D$ or $\A$, there is a canonical inclusion
	\[
		\Psi_\bullet ({\bf i}, s) : \mcal{GP}({\bf i}) \hookrightarrow \mcal{GP}(E_\bullet(s)({\bf i}))
	\]
	Moreover, 
	\[
		\mathrm{Im} \Psi_\bullet ({\bf i}, s) = \{ \p \in \mcal{GP}(E_\bullet(s)({\bf i})) ~|~ \text{$\mathrm{node}(\p)$ does not contain a node lying on $\ell_\bullet$} \},
	\]
	where $\ell_\D \coloneqq \ell_{n+2}$ and $\ell_\A \coloneqq \ell_1$.
\end{lemma}

	We call a rigorous path $\p \in \mcal{GP}(E_\bullet(s)({\bf i}))$ {\em $\bullet$-new} if it is not in $\mathrm{Im} \Psi_\bullet ({\bf i}, s)$.
	More generally, we can define $\bullet$-new paths for general reduced words as follows. Note that if ${\bf i} \sim {\bf i}'$, then there is a natural
	identification between $\mcal{GP}({\bf i})$ and $\mcal{GP}({\bf i}')$. 
	Since the extension is surjective up to $2$-moves (see~Remark~\ref{rmk_extension_surjective}), for any $\mathbf i \in \Rn{n+1}$, there exist $\mathbf i' \in \Rn{n}$, $\bullet \in \{\A,\D\}$, and $s \in [ 0, \overline{n-1}]$ such that $E_{\bullet}(s)(\mathbf{i'}) \sim \mathbf{i}$. 	
	Thus we say that $\p \in \mcal{GP}({\bf i})$ {\em $\bullet$-new} if 
	$\p$ is $\bullet$-new in $\mcal{GP}(E_{\bullet}(s)(\mathbf{i'}) )$.
Equivalently, $\p \in \mcal{GP}({\bf i})$ is $\bullet$-new 
	if $\mathrm{node}(\p)$ contains a node lying on $\ell_\bullet$. See \cite[Section 5]{CKLP} for more details.
	
	For any reduced word ${\bf i} \in \Rn{n+1}$, the authors provide in ~\cite[Propositions~5.6 and~5.7]{CKLP} an explicit way of finding $n$
	$\bullet$-new paths (called {\em canonical}) in $\mcal{GP}({\bf i})$ for each $\bullet = \A$ or $\D$.

	\begin{proposition}[{\cite[Proposition~5.6]{CKLP}}]\label{prop_D_canonical}
		Let ${\bf i} \in \Rn{n+1}$ and $k \in [n]$. Let $t_{j_k}$ be the node at which $\ell_k$ and $\ell_{n+1}$ intersect. 
		Then there exists a rigorous path $\p_\D({\bf i}, k) \in \mcal{GP}({\bf i})$ such that 
		\begin{itemize}
			\item it has a unique peak $t_{j_k}$, 
			\item it travels from $\ell_k$ to $\ell_{n+1}$ at $t_{j_k}$, 			
			\item it is below $\ell_{n+1}$, 
			\item with respect to the wire-expression of $\p_\D({\bf i}, k)\colon$ 
			\[
				\ell_{r_p} \rightarrow \cdots \rightarrow \ell_{r_1} \rightarrow \ell_k \rightarrow \ell_{n+1} \rightarrow \ell_{u_q} \rightarrow \cdots \rightarrow \ell_{u_1} 
				(\coloneqq \ell_{r_p+1}),
			\]
			the sequences $r_1, \dots, r_p$ and $u_1, \dots, u_q$ are increasing, and 
			\item $\p_\D(\mathbf i, k) \subset R_{a_k}$, where $a_k = \max \{ a \mid t_{j_k} \in R_a\}$. 
		\end{itemize}		
	We call the path $\p_\D(\mathbf i,k)$ a  \textit{canonical $\D$-new path}. 
	\end{proposition}
\begin{example}\label{example_canonical_D}
	Let $\mathbf i = (4,3,4,2,3,4,1,2,3,4,5,4,6,5,4,3,2,1,4,3,2) \in \Rn{7}$. See Figure~\ref{fig_DD_R} for the wiring diagram of $\mathbf i$.
	Then, \[
	\mathbf i = E_\D(3)(4,3,4,2,3,4,1,2,3,4,5,4,3,2,1).
	\] 
	For $k = 3$, there are two $\D$-new rigorous paths which satisfy the first four conditions in Proposition~\ref{prop_D_canonical}:
	\[
	\p_1 \coloneqq (\ell_3 \to \ell_7 \to \ell_4) \subset R_3,  \quad 
	\p_2 \coloneqq (\ell_6 \to \ell_3 \to \ell_7) \subset R_6.
	\]
	Since  $6 >3$, the path $\p_2$ is the canonical $\D$-new path $\p_\D(\mathbf i,3)$, and $\p_1$ is a $\D$-new path but \textit{not} canonical. 
	By  similar observations, one can find the following canonical $\D$-new paths.
	\[
	\begin{array}{lll}
	\p_\D(\mathbf i, 1) = (\ell_1 \to \ell_7 \to \ell_2), &
	\p_\D(\mathbf i,2) = (\ell_2 \to \ell_7 \to \ell_6 \to \ell_3),&
	\p_\D(\mathbf i,3) = (\ell_6 \to \ell_3 \to \ell_7),\\
	\p_\D(\mathbf i,4) = (\ell_6 \to \ell_4 \to \ell_7),&
	\p_\D(\mathbf i, 5) = (\ell_6 \to \ell_5 \to \ell_7),&
	\p_\D(\mathbf i,6) = (\ell_6 \to \ell_7).
	\end{array}
	\]
\end{example}

\section{Combinatorics of string polytopes of index $(0,\dots,0,k)$}
\label{section_comb_string_polytopes}

In this section, we compare  $\mcal{GP}(C_\D({\bf i}))$ with $\mcal{GP}({\bf i})$ which are sets of rigorous paths. 
Throughout this section, we assume that 
$\mathbf{i} \sim \is \in \Rn{n+1}$ with $I = (0,\dots,0,k)$ for some $k \leq n-1$ and $\ad = (\ad_1,\dots,\ad_n) \in \{\A, \D\}^n$.
Without loss of generality, we may assume that 
\[
\ad_n = \D
\]
by Proposition~\ref{same string polytopes under Dynkin diagram automorphism}.
Moreover, by Lemma~\ref{lemma_2_move_unimodular_same_string_polytopes}, we may assume that
\[
\mathbf i = \mathbf i_{\ad}(I) (= \mathbf i_{(\ad_1,\dots,\ad_n)}(0,\dots,0,k)).
\]
Then the $(\ad_1,\dots,\ad_{n-1})$-index of the contraction $C_{\D}(\mathbf i) \in \Rn{n}$ becomes the zero vector, i.e.,
\[
\ind_{(\ad_1,\dots,\ad_{n-1})}(C_\D(\mathbf i)) = (0,\dots,0) \in \Z^{n-1}.
\]
Therefore, 
$C_{\D}(\mathbf i)$ defines the string polytope which is unimodularly equivalent to the Gelfand--Cetlin polytope by \cite[Theorem~6.7]{CKLP}. In particular, 
we have
\[
	\lvert\GP(C_{\D}(\mathbf i)) \rvert = \frac{n(n-1)}{2} = \barn-n.
\]

Now we consider the canonical inclusion map defined in Lemma \ref{lem_inclusion}:
\[
	\Psi  \coloneqq \Psi_{\D}(C_{\D}(\mathbf i),k) \colon \GP(C_{\D}(\mathbf i)) \hookrightarrow \GP(\mathbf i)
\]
which sends a rigorous path $\p = (\ell_i \to \ell_{r_1} \to \cdots \to \ell_{r_p} \to \ell_{i+1})$ in $\GP(C_{\D}(\mathbf i))$ to the path 
\begin{equation}\label{eq_injective_map_path}
\Psi(\p) =
(\ell_i \to \ell_{r_1} \to \cdots \to \ell_{r_p} \to \ell_{i+1}) \in \GP(\mathbf i).
\end{equation}

\begin{theorem}\label{thm_DD_AD_combine}
	Let $\ad = (\ad_1,\dots,\ad_n) \in \{\A,\D\}^n$ and $I = (0,\dots,0,k)$ with $k \leq n-1$.
	Assume that $n \geq 2$ and $\ad_n = \D$. Then,
	\begin{equation}
	\lvert \GP(\is) \rvert = \begin{cases}
	\barn + k -1 & \text{ if } \ad_{n-1} = \D \text{ and } k = n-1, \\
	\barn+k & \text{ if } \ad_{n-1} = \D  \text{ and } 0< k < n-1,\\
	\barn & \text{ if } k = 0; \text{or }\ad_{n-1} = \A \text{ and } k = n-1,\\
	\barn+1 & \text{ if } \ad_{n-1} = \A \text{ and } 0 < k < n-1.
	\end{cases}
	\end{equation}
	Indeed, $|\mcal{GP}(\is)|$ depends only on $k$, $\ad_{n-1}$ and $\ad_n$.
\end{theorem}

The proof of Theorem~\ref{thm_DD_AD_combine} will be split into the cases mentioned in the theorem.
To prove Theorem \ref{thm_DD_AD_combine}, we need the following lemma. 
\begin{lemma}\label{lemma_i+1_once}
	Let $\mathbf i \in \Rn{n+1}$ and $\p = (\ell_i \to \ell_{r_1} \to \cdots \to \ell_{r_p} \to \ell_{i+1} ) \in \GP(\mathbf i)$. Then $ i+1 \notin \{r_1,\dots,r_p\}$.
	In particular, with respect to the upward orientation of $\ell_i$, 
	if $\ell_i$ meets $\ell_{r_1}$ just before  intersecting $\ell_{i+1}$, then $\p$ should be $\ell_i \to \ell_{r_1} \to \ell_{i+1}$.
\end{lemma}
\begin{proof}
	Clearly, $r_p \neq i +1 $. 
	Assume on the contrary that $r_j = i + 1$ for some $j \in \{1,\dots,p-1\}$. Then there are four possible configurations of the sub-path $\ell_{r_j} \rightarrow \ell_{r_{j+1}}$ as follows
	(where each {\em red} broken arrow describes a part of the path $\p$).
	\begin{figure}[H]
	\begin{subfigure}{0.15\textwidth}
		\centering
		\begin{tikzpicture}
		\tikzset{red line/.style = {line width=0.5ex, red, semitransparent}}
		
		\filldraw[black!10, draw=white] (0,0)--(1,0)--(1,1)--cycle;
		
		\draw[thick, ->]  (1,1) -- (0,0) node[above, at start] {$\ell_{i+1}$};
		\draw[thick, ->] (0,1) -- (1,0) node[above, at start] {$\ell_{r_{j+1}}$};
		
		\draw[red line, ->] (1,1)--(0.5,0.5)--(1,0);
		\end{tikzpicture}
		\caption{Case $1$.}
	\end{subfigure}
	\begin{subfigure}{0.15\textwidth}
		\centering
		\begin{tikzpicture}
		\tikzset{red line/.style = {line width=0.5ex, red, semitransparent}}
		
		\filldraw[black!10, draw=white] (0,0)--(1,0)--(1,1)--cycle;
		
		\draw[thick, ->]  (1,1) -- (0,0) node[above, at start] {$\ell_{i+1}$};
		\draw[thick, <-] (0,1) -- (1,0) node[above, at start] {$\ell_{r_{j+1}}$};
		
		\draw[red line, ->] (1,1)--(0.5,0.5)--(0,1);
		\end{tikzpicture}
		\caption{Case $2$.}
	\end{subfigure}
	\begin{subfigure}{0.15\textwidth}
		\centering
		\begin{tikzpicture}
		\tikzset{red line/.style = {line width=0.5ex, red, semitransparent}}
		
		\filldraw[black!10, draw=white] (0,1)--(1,0)--(1,1)--cycle;
		
		\draw[thick, <-]  (1,1) -- (0,0) node[above, at start] {$\ell_{r_{j+1}}$};
		\draw[thick, ->] (0,1) -- (1,0) node[above, at start] {$\ell_{i+1}$};
		
		\draw[red line, <-] (1,1)--(0.5,0.5)--(0,1);
		\end{tikzpicture}
		\caption{Case $3$.}
	\end{subfigure}
	\begin{subfigure}{0.15\textwidth}
		\centering
		\begin{tikzpicture}
		\tikzset{red line/.style = {line width=0.5ex, red, semitransparent}}
		
		\filldraw[black!10, draw=white] (0,1)--(1,0)--(1,1)--cycle;
		
		\draw[thick, ->]  (1,1) -- (0,0) node[above, at start] {$\ell_{r_{j+1}}$};
		\draw[thick, ->] (0,1) -- (1,0) node[above, at start] {$\ell_{i+1}$};
		
		\draw[red line, <-] (0,0)--(0.5,0.5)--(0,1);
		\end{tikzpicture}
		\caption{Case $4$.}
	\end{subfigure}
	\caption{Four possible configurations of $\ell_{i+1} \to \ell_{r_{j+1}}$.}\label{figure_crossing}
\end{figure}
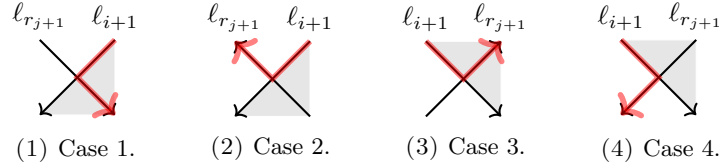
	\noindent
	By Lemma~\ref{lemma_P_contatined_in_Ri}, the path $\p$ should be contained in the region $R_i$.
	Since the region $R_i$ is enclosed by paths $\ell_{i+1}$ and $\ell_i$, the path $\p$ should travel the shaded parts in Figure~\ref{figure_crossing}.
	But, for Cases $2$ and $4$, the red paths are not contained in $R_i$, and hence those cases are excluded. 
	
	For Case $1$, since the wire $\ell_{r_{j+1}}$ goes downward, $r_{j+1} > i+1$. Thus, the wire $\ell_{r_{j+1}}$ is on the left hand side of the wire $\ell_{i+1}$ on the bottom. But it is impossible since each pair of wires meets only once. Similarly, Case $3$ is also impossible since $r_{j+1} < i+1$ and so the orientation of $\ell_{r_{j+1}}$ should be downward. 
Therefore, the result follows. 
\end{proof}

\subsection{Case $1$: $(\ad_{n-1}, \ad_n) = (\D,\D)$.}
\begin{proposition}\label{prop_DD}
	Let $\ad = (\ad_1,\dots,\ad_n) \in \{\A,\D\}^n$ and $\mathbf i  \in \Rn{n+1}$.
	Suppose that $\ad_{n-1} = \ad_n = \D$. 
	If $\mathbf{i} \sim  \is$ with $I = (0,\dots,0,k)$ for some $k \leq n-1$, then 
	\[
	\lvert \GP(\mathbf i) \rvert =
	\begin{cases} 
	\barn + k - 1&\text{ if }k = n-1,\\
	\barn +k &\text{ if } k <n-1.
	\end{cases}
	\] 
\end{proposition}

	\begin{figure}[H]
	\centering
	\begin{tikzpicture}[xscale=0.9,yscale=0.6]
	\tikzset{every node/.style = {font = \footnotesize}}
	\tikzset{red line/.style =  {line width=0.5ex, red, dotted}}
	\tikzset{blue line/.style = {line width=0.5ex, blue, semitransparent}}
	
	\filldraw[blue!10] (7,0)--(7,6.5)--(6,7)--(6,8)--(5,8.5)--(5,11.5)--(4,11.5)--(4,6)--(5,5.5)--(5,3.5)--(6,3)--(6,0);
	
	\node at (4.5,12) {$R_{n-k-1}$};
	
	\draw (0,0)--(0,3.5)--(9,8)--(9,11); 
	\draw (1,0)--(1,1)--(6,3.5)--(6,6)--(5,6.5)--(5,8)--(8,9.5)--(8,11);
	\draw (2,0)--(2,1)--(1,1.5)--(1,3.5)--(0,4)--(0,11);
	\draw[dashed, black!40] (3,0)--(3,1.5)--(2,2)--(2,4)--(1,4.5)--(1,11);
	\draw (4,0)--(4,2)--(3,2.5)--(3,4.5)--(2,5)--(2,11);
	\draw[dashed, black!40] (5,0)--(5,2.5)--(4,3)--(4,5)--(3,5.5)--(3,11);
	\draw (6,0)--(6,3)--(5,3.5)--(5,5.5)--(4,6)--(4,11);
	\draw (7,0)--(7,6.5)--(6,7)--(6,8)--(5,8.5)--(5,11);
	\draw[dashed, black!40] (8,0)--(8,7)--(7,7.5)--(7,8.5)--(6,9)--(6,11);
	\draw (9,0)--(9,7.5)--(8,8)--(8,9)--(7,9.5)--(7,11);

	\draw[dotted, draw = gray] (1.5,1.25)--(10,1.25) node[at end, right] {$\barn$};
	\draw[dotted, draw = gray] (3.5,2.25)--(10,2.25) node[at end, right] {$\barn-i+1$};
	\draw[dotted, draw = gray] (5.5,3.25)--(10,3.25) node[at end, right] {$\barn-k+1$};
	\draw[dotted, draw = gray] (0.5,3.75)--(10,3.75) node[at end, right] {$\barn-k$};
	\draw[dotted, draw = gray] (2.5,4.75)--(10,4.75) node[at end, right] {$\barn-k-i+1$};
	\draw[dotted, draw = gray] (4.5,5.75)--(10,5.75) node[at end, right] {$\barn-2k+1$};
	\draw[dotted, draw = gray] (5.5,6.25)--(10,6.25) node[at end, right] {$\barn-2k$};
	\draw[dotted, draw = gray] (6.5,6.75)--(10,6.75) node[at end, right] {$\barn-2k-1$};
	\draw[dotted, draw = gray] (8.5,7.75)--(10,7.75) node[at end, right] {$\barn-(n+k)+1$};
	\draw[dotted, draw = gray] (5.5,8.25)--(10,8.25) node[at end, right] {$\barn-(n+k)$};
	\draw[dotted, draw = gray] (7.5,9.25)--(10,9.25) node[at end, right] {$\barn-2n+2$};

	\node at (11,11) {labels of nodes};

	\node at (0,11.5) {$\vdots$};
	\node at (2,11.5) {$\vdots$};
	\node at (4,11.5) {$\vdots$};
	\node at (5,11.5) {$\vdots$};
	\node at (7,11.5) {$\vdots$};
	\node at (8,11.5) {$\vdots$};
	\node at (9,11.5) {$\vdots$};
	
	\node[below] at (0,0) {$\ell_{n+1}$};
	\node[below] at (1,0) {$\ell_n$};
	\node[below] at (2,0) {$\ell_{n-1}$};
	\node[below] at (3,0) {$\cdots$};
	\node[below] at (4,0) {$\ell_{n-i}$};
	\node[below] at (5,0) {$\cdots$};
	\node[below] at (6,0) {$\ell_{n-k}$};
	\node[below] at (7,0)  {$\ell_{n-k-1}$};
	\node[below] at (8,0) {$\cdots$};
	\node[below] at (9,0) {$\ell_1$};

	\draw[red line,->] (6.95,0)--(6.95,6.5)--(5.95,7)--(5.95,8)--(5.5,8.2)--(5.05,8)--(5.05,6.5)--(6.05,6)--(6.05,3.5)--(5.6,3.25)--(6.05,3)--(6.05,0);
	\node[red] at (6.3,1) {$\widetilde{\p}_0$};
	
	\draw[line width = 0.5ex, green!50!black,->, semitransparent] (7.05,0)--(7.05,6.55)--(6.05,7.05)--(6.05,8.05)--(5.5,8.3)--(4.95,8.05)--(4.95,6.45)--(5.45,6.25)--(4.45,5.75)--(4.95,5.5)--(4.95,3.5)--(5.95,3)--(5.95,0);
	
	\draw[blue line, ->] (4,0)--(4,2)--(3,2.5)--(3,4.5)--(2.5,4.75)--(1.5,4.25)--(2,4)--(2,2)--(3,1.5)--(3,0);
 
 	\node[blue] at (3.7,1) {$\widetilde{\p}_i$};
	
	\node[above] at (1.5,1.25) {$t_{\barn}$};

	\node[above] at (3.5,2.25) {$t_{\barn-i+1}$};
	\node[above] at (5.5, 3.25) {$t_{\barn-k+1}$};
	\node[above] at (0.5, 3.75) {$t_{\barn-k}$};	
	\node[above] at (2.5,4.75) {$t_{\barn-k-i+1}$};
	\node[above] at (4.5,5.75) {$t_{\barn-2k+1}$};
	\node[above] at (5.5,6.25) {$t_{\barn-2k}$};
	\node[above] at (6.5,6.75) {$t_{\barn-2k-1}$};
	\node[below] at (8.5,7.75) {$t_{\barn-(n+k)+1}$};
	\node[above] at (5.5,8.25) {$t_{\barn-(n+k)}$};
	\node[ above] at (7.5,9.25) {$t_{\barn-2n+2}$};
	
	\end{tikzpicture}
	\caption{Rigorous paths $\ell_{n-k-1} \to \ell_n \to \ell_{n+1} \to \ell_{n-k}$ (green), $\widetilde{\p}_0 = (\ell_{n-k-1} \to \ell_n \to \ell_{n-k})$ (dotted red) and $\widetilde{\p}_i = (\ell_{n-i} \to \ell_{n+1} \to \ell_{n-i+1})$ (blue) when $(\ad_{n-1},\ad_n) = (\D,\D)$.}
	\label{fig_p0_and_pi_DD}
\end{figure}

\begin{proof}
	By Proposition~\ref{prop_D_canonical}, there are $\barn$ canonical $\D$-new paths. To prove the proposition, it is enough to count the non-canonical new paths.
By Lemma~\ref{lemma_i+1_once}, the only possible region $R_i$ which contains two different rigorous paths having the same maximal peak is 
$R_{n-k-1}$. This is because $R_{n-k-1}$ is the only region, which contains a node in its interior. One can see in   Figure~\ref{fig_p0_and_pi_DD} that the wires $\ell_{n+1}$ and $\ell_n$ intersect at the node $t_{\barn-2k}$ in the interior of $R_{n-k-1}$.
This produces a $\D$-new path of Case~II-2 in~\cite[Proposition~5.10]{CKLP} (see the green path in Figure~\ref{fig_p0_and_pi_DD}):
	\begin{equation}\label{eq_new_path_DD_case}
\ell_{n-k-1} \to \ell_n \to \ell_{n+1} \to \ell_{n-k},
\end{equation}
which is not canonical since its peak, $t_{\barn-(n+k)}$, does not lie on the wire $\ell_{n+1}$ (this violates the first condition in Proposition \ref{prop_D_canonical}).
Note that there is another rigorous path $\widetilde{\p}_0 \coloneqq (\ell_{n-k-1} \to \ell_n \to \ell_{n-k})$ having the same maximal peak $t_{\barn - (n+k)}$. See the dotted red path in Figure~\ref{fig_p0_and_pi_DD}.

On the other hand, one can see that for $ 1 \leq j_1 < j_2 \leq n$, 
\[
R_{j_1}^{\circ} \cap R_{j_2}^{\circ} \neq \emptyset \iff n-k-1 \leq j_1 \leq n-2 \text{ and } j_2 = n
\]
(cf. Figure~\ref{fig_DD_R}). These intersections produce $k-1$  $\D$-new paths of Case~I-2 in~\cite[Proposition~5.10]{CKLP}:
	\[
\widetilde{\p}_i \coloneqq (\ell_{n-i} \to \ell_{n+1} \to \ell_{n-i+1}) \quad \text{ for } 2 \leq i \leq k 
\] 
which are not canonical. See the blue path in Figure~\ref{fig_p0_and_pi_DD}. 
Therefore, there are exactly $k$ non-canonical $\D$-new paths when $n-k-1 > 0$, i.e., $k < n-1$. 

If $k=n-1$, then no red dotted line appears in Figure \ref{fig_p0_and_pi_DD} and so there are exactly $k-1$ non-canonical $\D$-new paths (blue paths). 
This completes the proof.
\end{proof}

\begin{example}\label{example_DD}
Let $\ad = (\D,\A,\A,\A,\D,\D)$ and $I = (0,0,0,0,0,3)$. Let
	\[
	\mathbf i \coloneqq \is = (4,3,4,2,3,4,1,2,3,4,5,4,6,5,4,3,2,1,4,3,2) \in \Rn{7}.
	\] 
	The regions $R_i$ are presented in Figure~\ref{fig_DD_R}. There are  three $\D$-new paths which are not canonical (see red paths in Figure~\ref{fig_example_DD}):
	\[
	\ell_2 \to \ell_6 \to \ell_7 \to \ell_3,\quad \ell_3 \to \ell_7 \to \ell_4, \quad \ell_4 \to \ell_7 \to \ell_5. 
	\]
	Note that the canonical $\D$-new paths $\p_\D(\mathbf i, k)$ for $k = 3, 4$ are $\ell_6 \to \ell_3 \to \ell_7$ and $\ell_6 \to \ell_4 \to \ell_7$, respectively (see Example~\ref{example_canonical_D}, and blue paths in Figures~\ref{fig_DD_1} and \ref{fig_DD_2}).
	\begin{figure}[H]
			\begin{subfigure}[t]{0.23\textwidth}
		\centering
		\begin{tikzpicture}[scale = 0.5]
		\tikzset{every node/.style = {font = \footnotesize}}
		\tikzset{red line/.style = {line width=0.5ex, red, semitransparent}}
		\tikzset{blue line/.style = {line width=0.5ex, blue, semitransparent}}
		
		\draw(0,-0.5)--(0,0) -- (0,1.5)--(6, 4.5)--(6,11);
		\draw(1,-0.5)--(1,0)--(4,1.5)--(4,3)--(3,3.5)--(3,4.5)--(5,5.5)--(5,11);
		\draw(2,-0.5)--(2,0)--(1,0.5)--(1,1.5)--(0,2)--(0,7)--(1,7.5)--(1,8.5)--(2,9)--(2,9.5)--(4,10.5)--(4,11);
		\draw (3,-0.5)--(3,0)--(3,0.5)--(2,1)--(2,2)--(1,2.5)--(1,6.5)--(2,7)--(2,8)--(3,8.5)--(3,9)--(4,9.5)--(4,10)--(3,10.5)--(3,11);
		\draw(4,-0.5)--(4,0)--(4,1)--(3,1.5)--(3,2.5)--(2,3)--(2,6)--(3,6.5)--(3,7.5)--(4,8)--(4,9)--(2,10)--(2,11);
		\draw (5,-0.5)--(5,3.5)--(4,4)--(4,4.5)--(3,5)--(3,5.5)--(4,6)--(4,7.5)--(1,9)--(1,11);
		\draw (6,-0.5)--(6,4)--(5,4.5)--(5,5)--(0,7.5)--(0,11);
		
		\node[below] at (0,-0.5) {$\ell_7$};
		\node[below] at (1,-0.5) {$\ell_6$};
		\node[below] at (2,-0.5) {$\ell_5$};
		\node[below] at (3,-0.5) {$\ell_4$};
		\node[below] at (4,-0.5) {$\ell_3$};
		\node[below] at (5,-0.5) {$\ell_2$};
		\node[below] at (6,-0.5) {$\ell_1$};
		
		\draw[red line, ->] (5,-0.5)--(5,3.5)--(4,4)--(4,4.5)--(3.5,4.75)--(3,4.5)--(3,3.5)--(3.5,3.25)--(2.5,2.75)--(3,2.5)--(3,1.5)--(4,1)--(4,-0.5);
	
		\node[above] at (1.5,0.25) {\footnotesize $t_{21}$};
	\node[above] at (2.5,0.75) {\footnotesize $t_{20}$};
	\node[above] at (3.5,1.25) {\footnotesize $t_{19}$};
	\node[above] at (0.5,1.75) {\footnotesize $t_{18}$};
	\node[above] at (1.5,2.25) {\footnotesize $t_{17}$};
	\node[above] at (2.5,2.75) {\footnotesize $t_{16}$};
	\node[above] at (3.5,3.25) {\footnotesize $t_{15}$};
	\node[above] at (4.5,3.75) {\footnotesize $t_{14}$};
	\node[above] at (5.5,4.25) {\footnotesize $t_{13}$};
	\node[label={[label distance=-1.5mm]:{\footnotesize $t_{12}$}}] at (3.5,4.75) {};
	\node[above] at (4.5,5.25) {\footnotesize $t_{11}$};
	\node[above] at (3.5,5.75) {\footnotesize $t_{10}$};
	\node[above] at (2.5,6.25) {\footnotesize $t_{9}$};
	\node[above] at (1.5,6.75) {\footnotesize $t_{8}$};
	\node[above] at (0.5,7.25) {\footnotesize $t_{7}$};
	\node[above] at (3.5,7.75) {\footnotesize $t_{6}$};
	\node[above] at (2.5,8.25) {\footnotesize $t_{5}$};
	\node[above] at (1.5,8.75) {\footnotesize $t_{4}$};
	\node[above] at (3.5,9.25) {\footnotesize $t_{3}$};
	\node[above] at (2.5,9.75) {\footnotesize $t_{2}$};
	\node[above] at (3.5,10.25) {\footnotesize $t_{1}$};
		
		\end{tikzpicture}
		\caption{$\ell_2 \to \ell_6 \to \ell_7 \to \ell_3$.}
		\label{fig_DD_3}
	\end{subfigure}
		\begin{subfigure}[t]{0.35\textwidth}
			\centering
			\begin{tikzpicture}[scale = 0.5]
			\tikzset{every node/.style = {font = \footnotesize}}
			\tikzset{red line/.style = {line width=0.5ex, red, dotted}}
			\tikzset{blue line/.style = {blue, very thick}}
			
			\draw(0,-0.5)--(0,0) -- (0,1.5)--(6, 4.5)--(6,11);
			\draw(1,-0.5)--(1,0)--(4,1.5)--(4,3)--(3,3.5)--(3,4.5)--(5,5.5)--(5,11);
			\draw(2,-0.5)--(2,0)--(1,0.5)--(1,1.5)--(0,2)--(0,7)--(1,7.5)--(1,8.5)--(2,9)--(2,9.5)--(4,10.5)--(4,11);
			\draw (3,-0.5)--(3,0)--(3,0.5)--(2,1)--(2,2)--(1,2.5)--(1,6.5)--(2,7)--(2,8)--(3,8.5)--(3,9)--(4,9.5)--(4,10)--(3,10.5)--(3,11);
			\draw(4,-0.5)--(4,0)--(4,1)--(3,1.5)--(3,2.5)--(2,3)--(2,6)--(3,6.5)--(3,7.5)--(4,8)--(4,9)--(2,10)--(2,11);
			\draw (5,-0.5)--(5,3.5)--(4,4)--(4,4.5)--(3,5)--(3,5.5)--(4,6)--(4,7.5)--(1,9)--(1,11);
			\draw (6,-0.5)--(6,4)--(5,4.5)--(5,5)--(0,7.5)--(0,11);
			
			\node[below] at (0,-0.5) {$\ell_7$};
			\node[below] at (1,-0.5) {$\ell_6$};
			\node[below] at (2,-0.5) {$\ell_5$};
			\node[below] at (3,-0.5) {$\ell_4$};
			\node[below] at (4,-0.5) {$\ell_3$};
			\node[below] at (5,-0.5) {$\ell_2$};
			\node[below] at (6,-0.5) {$\ell_1$};
			
			\draw[blue line, ->] (1,-0.5)--(1,0)--(3.5,1.25)--(3,1.5)--(3,2.5)--(2.5,2.75)--(0,1.5)--(0,-0.5);
			\draw[red line,->] (4,-0.4)--(4,1)--(3,1.5)--(3,2.5)--(2.5,2.75)--(1.5,2.25)--(2,2)--(2,1)--(3,0.5)--(3,-0.5);
			\node[above] at (1.5,0.25) {\footnotesize $t_{21}$};
		\node[above] at (2.5,0.75) {\footnotesize $t_{20}$};
		\node[above] at (3.5,1.25) {\footnotesize $t_{19}$};
		\node[above] at (0.5,1.75) {\footnotesize $t_{18}$};
		\node[above] at (1.5,2.25) {\footnotesize $t_{17}$};
		\node[above] at (2.5,2.75) {\footnotesize $t_{16}$};
		\node[above] at (3.5,3.25) {\footnotesize $t_{15}$};
		\node[above] at (4.5,3.75) {\footnotesize $t_{14}$};
		\node[above] at (5.5,4.25) {\footnotesize $t_{13}$};
		\node[label={[label distance=-1.5mm]:{\footnotesize $t_{12}$}}] at (3.5,4.75) {};
		\node[above] at (4.5,5.25) {\footnotesize $t_{11}$};
		\node[above] at (3.5,5.75) {\footnotesize $t_{10}$};
		\node[above] at (2.5,6.25) {\footnotesize $t_{9}$};
		\node[above] at (1.5,6.75) {\footnotesize $t_{8}$};
		\node[above] at (0.5,7.25) {\footnotesize $t_{7}$};
		\node[above] at (3.5,7.75) {\footnotesize $t_{6}$};
		\node[above] at (2.5,8.25) {\footnotesize $t_{5}$};
		\node[above] at (1.5,8.75) {\footnotesize $t_{4}$};
		\node[above] at (3.5,9.25) {\footnotesize $t_{3}$};
		\node[above] at (2.5,9.75) {\footnotesize $t_{2}$};
		\node[above] at (3.5,10.25) {\footnotesize $t_{1}$};	
			\end{tikzpicture}
			\caption{$\ell_3 \to \ell_7 \to \ell_4$ (dotted red) and \\
				$\ell_6 \to \ell_3 \to \ell_7$ (blue).}
			\label{fig_DD_2}
		\end{subfigure}
		\begin{subfigure}[t]{0.35\textwidth}
			\centering
			\begin{tikzpicture}[scale = 0.5]
			\tikzset{every node/.style = {font = \footnotesize}}
			\tikzset{red line/.style = {line width=0.5ex, red, dotted}}
			\tikzset{blue line/.style = {blue, very thick}}
			
			\draw(0,-0.5)--(0,0) -- (0,1.5)--(6, 4.5)--(6,11);
			\draw(1,-0.5)--(1,0)--(4,1.5)--(4,3)--(3,3.5)--(3,4.5)--(5,5.5)--(5,11);
			\draw(2,-0.5)--(2,0)--(1,0.5)--(1,1.5)--(0,2)--(0,7)--(1,7.5)--(1,8.5)--(2,9)--(2,9.5)--(4,10.5)--(4,11);
			\draw (3,-0.5)--(3,0)--(3,0.5)--(2,1)--(2,2)--(1,2.5)--(1,6.5)--(2,7)--(2,8)--(3,8.5)--(3,9)--(4,9.5)--(4,10)--(3,10.5)--(3,11);
			\draw(4,-0.5)--(4,0)--(4,1)--(3,1.5)--(3,2.5)--(2,3)--(2,6)--(3,6.5)--(3,7.5)--(4,8)--(4,9)--(2,10)--(2,11);
			\draw (5,-0.5)--(5,3.5)--(4,4)--(4,4.5)--(3,5)--(3,5.5)--(4,6)--(4,7.5)--(1,9)--(1,11);
			\draw (6,-0.5)--(6,4)--(5,4.5)--(5,5)--(0,7.5)--(0,11);
			
			\node[below] at (0,-0.5) {$\ell_7$};
			\node[below] at (1,-0.5) {$\ell_6$};
			\node[below] at (2,-0.5) {$\ell_5$};
			\node[below] at (3,-0.5) {$\ell_4$};
			\node[below] at (4,-0.5) {$\ell_3$};
			\node[below] at (5,-0.5) {$\ell_2$};
			\node[below] at (6,-0.5) {$\ell_1$};
			
			\draw[blue line, ->] (1,-0.5)--(1,0)--(2.5,0.75)--(2,1)--(2,2)--(1.5,2.25)--(0,1.5)--(0,-0.5);
			
			\draw[red line, ->] (3,-0.5)--(3,0.5)--(2,1)--(2,2)--(1.5,2.25)--(0.5,1.75)--(1,1.5)--(1,0.5)--(2,0)--(2,-0.5);
			\node[above] at (1.5,0.25) {\footnotesize $t_{21}$};
		\node[above] at (2.5,0.75) {\footnotesize $t_{20}$};
		\node[above] at (3.5,1.25) {\footnotesize $t_{19}$};
		\node[above] at (0.5,1.75) {\footnotesize $t_{18}$};
		\node[above] at (1.5,2.25) {\footnotesize $t_{17}$};
		\node[above] at (2.5,2.75) {\footnotesize $t_{16}$};
		\node[above] at (3.5,3.25) {\footnotesize $t_{15}$};
		\node[above] at (4.5,3.75) {\footnotesize $t_{14}$};
		\node[above] at (5.5,4.25) {\footnotesize $t_{13}$};
		\node[label={[label distance=-1.5mm]:{\footnotesize $t_{12}$}}] at (3.5,4.75) {};
		\node[above] at (4.5,5.25) {\footnotesize $t_{11}$};
		\node[above] at (3.5,5.75) {\footnotesize $t_{10}$};
		\node[above] at (2.5,6.25) {\footnotesize $t_{9}$};
		\node[above] at (1.5,6.75) {\footnotesize $t_{8}$};
		\node[above] at (0.5,7.25) {\footnotesize $t_{7}$};
		\node[above] at (3.5,7.75) {\footnotesize $t_{6}$};
		\node[above] at (2.5,8.25) {\footnotesize $t_{5}$};
		\node[above] at (1.5,8.75) {\footnotesize $t_{4}$};
		\node[above] at (3.5,9.25) {\footnotesize $t_{3}$};
		\node[above] at (2.5,9.75) {\footnotesize $t_{2}$};
		\node[above] at (3.5,10.25) {\footnotesize $t_{1}$};	
			\end{tikzpicture}
			\caption{ $\ell_4 \to \ell_7 \to \ell_5$ (dotted red) and \\
				$\ell_6 \to \ell_4 \to \ell_7$ (blue).}
			\label{fig_DD_1}
		\end{subfigure}

		\caption{$\D$-new paths in Example~\ref{example_DD}.}
		\label{fig_example_DD}
	\end{figure}
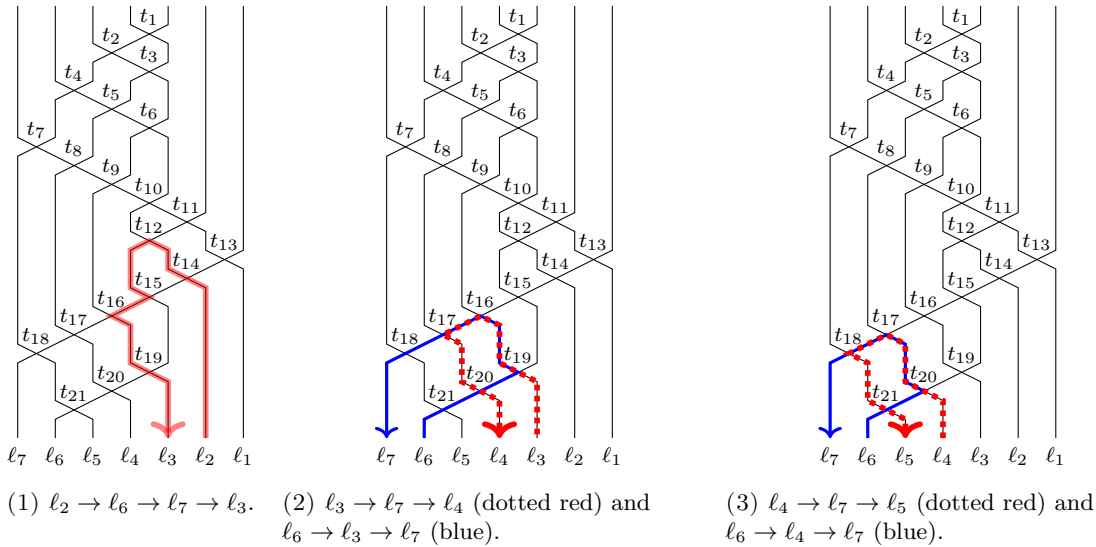	
\end{example}

\subsection{Case $2$: $(\ad_{n-1}, \ad_n) = (\A,\D)$.}
\begin{lemma}\label{lem_AD_DA}
	Let $k$ be a positive integer satisfying $k \leq n-1$. 
	Consider sequences $\ad = (\ad_1,\dots,\ad_{n-2},\A,\D)$ and
	$\ad' = (\ad_1,\dots,\ad_{n-2}, \D, \A)$ in $\{\A,\D\}^n$. 
	Then, we have the equivalence
	\[
	\isI{\ad}{\underbrace{0,\dots,0}_{n-1},k} \sim \isI{\ad'}{\underbrace{0,\dots,0}_{n-1},n-k-1}.
	\]
\end{lemma}
\begin{proof}
	Let $I = (0,\dots,0,k)$. We set $\mathbf i \coloneqq \is$.
	Since $\ind_\D(\mathbf i) = k$ and $\ind_\A(C_\D(\mathbf i)) = 0$, the last part of the sequence $\mathbf i$ has the following form:
	\[
	\mathbf i = (\mathbf{i}',  \boxed{1}, \boxed{2},\dots, \boxed{n-k-1}, \underbrace{n,n-1, \dots,n-k+1, \boxed{n-k}, n-k-1, \dots, 2, 1}_{\D_n}, \boxed{n-k + 1}, \dots,\boxed{n}).
	\]
	Here, 
	\(
	\mathbf i' = \mathbf{i}_{(\ad_1,\dots,\ad_{n-2})}(0,\dots,0) + 1 = C_\D(C_\A(\mathbf i))+1.
	\)
	Then, the boxed numbers in the above equation form the sequence $\A_n$. Thus we have that
	\begin{equation}\label{eq_equivalence_on_i_sigma_and_i'_sigma'}
	\mathbf i \sim (\mathbf{i}', n,n-1,\dots,n-k+1, \underbrace{\boxed{1}, \boxed{2},\dots,\boxed{n-k-1},
	\boxed{n-k},\boxed{n-k+1},\dots,\boxed{n}}_{\A_n}, n-k-1,\dots,2,1).
	\end{equation}
	On the other hand, we have that
	\[
	\begin{split}
		\mathbf i_{\ad'}(0,\dots,0,n-k-1)
		&= (E_\A(n-k-1) \circ E_\D(0))(\mathbf i_{(\ad_1,\dots,\ad_{n-2})}(0,\dots,0)) \\
		&= (E_\A(n-k-1) \circ E_\D(0))(\mathbf i'-1) \\
		&= (E_\A(n-k-1))(\mathbf i'-1,n-1,n-2,\dots,1) \\
		&= (\mathbf i',n,n-1,\dots,n-k+1,1,2,\dots,n,n-k-1,\dots,2,1).		
	\end{split}
	\]
Hence the equivalence~\eqref{eq_equivalence_on_i_sigma_and_i'_sigma'} proves the lemma.
\end{proof}

\begin{example}
	Let $\mathbf i = (2,3,2,1,2,3) \in \Rn{4}$. Then we have that
	\[
	\mathbf i = \isI{(\A,\A,\D)}{0,0,2} = \isI{(\A,\D,\A)}{0,0,0}.
	\]
\end{example}

\begin{proposition}\label{prop_AD}
	Let $\ad = (\ad_1,\dots,\ad_n) \in \{\A,\D\}^n$ and $I = (0,\dots,0,k)$ with $0 < k < n-1$.
	Suppose that $\ad_{n-1} = \A$ and $\ad_n = \D$. If $\mathbf{i} \sim \is$, then $\lvert \GP(\mathbf i) \rvert  = \barn +1$. 
\end{proposition}
\begin{proof}
By Lemma~\ref{lemma_i+1_once}, the only possible enclosed region $R_i$ which contains two different rigorous paths which have the same maximal peak is $R_{k+1}$ (see Figure~\ref{fig_p0_AD}). Note that in the interior of $R_{k+1}$, wires $\ell_{n+1}$ and $\ell_{1}$ meet at $t_{\barn-n+1}$. 
	This produces  a $\D$-new path of Case~I-1 in~\cite[Proposition~5.10]{CKLP}:
	\begin{equation}\label{eq_new_path_AD_case}
	\ell_{k+1} \to \ell_{n+1} \to \ell_1 \to \ell_{k+2}.
	\end{equation}
	which is not canonical as its maximal peak $t_{\barn-(n+k)}$ does not lie on the wire $\ell_{n+1}$ (this violates the first condition in Proposition \ref{prop_D_canonical}).
	Note that there is another rigorous path $\widetilde{\p}_0 \coloneqq(\ell_{k+1} \to \ell_1 \to \ell_{k+2})$ having the same maximal peak 
	$t_{\barn-(n+k)}$.
	On the other hand, in this case, we have $R_{j_1}^{\circ} \cap R_{j_2}^{\circ} = \emptyset$ for all $1 \leq j_1 < j_2 \leq n$. 
	Consequently, there is exactly one non-canonical $\D$-new path, 
	and therefore the result follows.
\end{proof}

\begin{figure}[h]
	\centering
	\begin{tikzpicture}[xscale=0.9,yscale=0.6]
	\tikzset{every node/.style = {font = \footnotesize}}
	\tikzset{red line/.style = {line width=0.5ex, red, dotted}}
	\tikzset{blue line/.style = {line width=0.5ex, blue, semitransparent}}
	
	\filldraw[blue!10] (4,0)--(4,2)--(5,2.5)--(5,4.5)--(4,5)--(4,8)--(3,8)--(3,6.5)--(2,6)--(2,4)--(3,3.5)--(3,0);
	
	\node at (3.5,7.5) {$R_{k+1}$};
	
	\draw (0,0)--(0,2.5)--(7,6)--(7,8);
	\draw (1,0)--(1,2.5)--(0,3)--(0,7)--(1,7.5)--(1,8);
	\draw[dashed, black!40] (2,0)--(2,3)--(1,3.5)--(1,6.5)--(2,7)--(2,8);
	\draw (3,0)--(3,3.5)--(2,4)--(2,6)--(3,6.5)--(3,8);
	\draw (4,0)--(4,2)--(5,2.5)--(5,4.5)--(4,5)--(4,8);
	\draw[dashed, black!40] (5,0)--(5,1.5)--(6,2)--(6,5)--(5,5.5)--(5,8);
	\draw (6,0)--(6,1)--(7,1.5)--(7,5.5)--(6,6)--(6,8);
	\draw (7,0)--(7,1)--(4,2.5)--(4,4)--(3,4.5)--(3,6)--(0,7.5)--(0,8);
	
	\draw (3,8.75)--(3,9)--(4,9.5)--(4,9.75);
	\draw (4,8.75)--(4,9)--(3,9.5)--(3,9.75);

	\node at (0,8.5) {$\vdots$};
	\node at (1,8.5) {$\vdots$};
	\node at (3,8.5) {$\vdots$};
	\node at (4,8.5) {$\vdots$};
	\node at (6,8.5) {$\vdots$};
	\node at (7,8.5) {$\vdots$};
	
	\node at (0,10.25) {$\vdots$};
	\node at (1,10.25) {$\vdots$};
	\node at (3,10.25) {$\vdots$};
	\node at (4,10.25) {$\vdots$};
	\node at (6,10.25) {$\vdots$};
	\node at (7,10.25) {$\vdots$};

	\node[below] at (0,0) {$\ell_{n+1}$};
	\node[below] at (1,0) {$\ell_n$};
	\node[below] at (2,0) {$\cdots$};
	\node[below] at (3,0) {$\ell_{k+2}$};
	\node[below] at (4,0) {$\ell_{k+1}$};
	\node[below] at (5,0) {$\cdots$};
	\node[below] at (6,0) {$\ell_{2}$};
	\node[below] at (7,0)  {$\ell_{1}$};
	
	\node at (9,11) {labels of nodes};
	
		\draw[red line, ->] (3.95,0)--(3.95,2.05)--(4.45,2.25)--(3.95,2.5)--(3.95,3.95)--(2.95,4.45)--(2.95,5.95)--(2.5,6.2)--(2.05,5.95)--(2.05,4.05)--(3.05,3.55)--(3.05,0);
	\node[red] at (3.7,1) {$\widetilde{\p}_0$};
	
	\draw[blue line,->] (4.05,0)--(4.05,1.95)--(5.05,2.5)--(5.05,4.5)--(4.5,4.8)--(3.5,4.3)--(3.05,4.5)--(3.05,6.05)--(2.5,6.3)--(1.95,6.05)--(1.95,4)--(2.95,3.5)--(2.95,0);

	\draw[dotted, draw = gray] (6.5,1.25)--(8,1.25) node[at end, right] {$\barn$};
	\draw[dotted, draw = gray] (4.5,2.25)--(8,2.25) node[at end, right] {$\barn-k+1$};
	\draw[dotted, draw = gray] (0.5,2.75)--(8,2.75) node[at end, right] {$\barn-k$};
	\draw[dotted, draw = gray] (2.5,3.75)--(8,3.75) node[at end, right] {$\barn-n+2$};
	\draw[dotted, draw = gray] (3.5,4.25)--(8,4.25) node[at end, right] {$\barn-n+1$};
	\draw[dotted, draw = gray] (4.5,4.75)--(8,4.75) node[at end, right] {$\barn-n$};
	\draw[dotted, draw = gray] (6.5,5.75)--(8,5.75) node[at end, right] {$\barn-(n+k)+1$};
	\draw[dotted, draw = gray] (2.5,6.25)--(8,6.25) node[at end, right] {$\barn-(n+k)$};
	\draw[dotted, draw = gray] (0.5,7.25)--(8,7.25) node[at end, right] {$\barn-2n+2$};
	\draw[dotted, draw=gray] (3.5,9.25)--(8,9.25) node[at end, right] {$x$};
	
	\node[above] at (6.5,1.25) {$t_{\barn}$};
	\node[above] at (4.5,2.25) {$t_{\barn-k+1}$};
	\node[above] at (0.5,2.75) {$t_{\barn-k}$};
	\node[above] at (2.5,3.75) {$t_{\barn-n+2}$};
	\node[above] at (3.5,4.25) {$t_{\barn-n+1}$};
	\node[above] at (4.5,4.75) {$t_{\barn-n}$};
	\node[above] at (6.5,5.75) {$t_{\barn-(n+k)+1}$};
	\node[above] at (2.5,6.25) {$t_{\barn-(n+k)}$};
	\node[above] at (0.5,7.25) {$t_{\barn-2n+2}$};
	\node[above] at (3.5,9.25) {$t_{x}$};

	\end{tikzpicture}
	\caption{Rigorous paths $	\ell_{k+1} \to \ell_{n+1} \to \ell_1 \to \ell_{k+2}$ (blue) and $\widetilde{\p}_0=(\ell_{k+1} \to \ell_1 \to \ell_{k+2})$ (dotted red) when $(\ad_{n-1},\ad_n) = (\A,\D)$.}
	\label{fig_p0_AD}
\end{figure}
\begin{example}\label{example_AD}
	Let  $\mathbf i \coloneqq \isI{(\D,\A,\A,\A,\D)}{0,0,0,0,2} = (4,3,4,2,3,4,1,2,5,4,3,2,1,4,5)$. 
	In this case, the regions $R_i$ can be expressed by
	\[
	\begin{split}
	&R_1 = \cham{15}, \quad R_2 = \cham{6} \cup \cham{9} \cup \cham{14}, \quad R_3 = \cham{3} \cup \cham{5} \cup \cham{8} \cup \cham{10} \cup \cham{11}, \\
	& R_4 = \cham{1} \cup \cham{2} \cup \cham{4} \cup \cham{7} \cup \cham{12}, \quad R_5 = \cham{13}.
	\end{split}
	\]
	One can easily check that $R_{j_1}^{\circ} \cap R_{j_2}^{\circ} = \emptyset$ for any $1 \leq j_1 < j_2 \leq 5$, and
	there is one $\D$-new path which is not canonical (see the dotted red path in Figure~\ref{fig_example_AD_2}):
	\[
	\ell_3 \to \ell_6 \to \ell_1 \to \ell_4.
	\]
	In this case, we have five canonical $\D$-new paths:
	\[
	\begin{array}{lll}
	\p_\D(\mathbf i,1) = (\ell_3 \to \ell_1 \to \ell_6 \to \ell_4), &
	\p_\D(\mathbf i,2) = (\ell_2 \to \ell_6 \to \ell_3), &
	\p_\D(\mathbf i,3) = (\ell_3 \to \ell_6 \to \ell_4), \\
	\p_\D(\mathbf i,4) = (\ell_4 \to \ell_6 \to \ell_5), &
	\p_\D(\mathbf i,5) = (\ell_5 \to \ell_6).
	\end{array}
	\]
	For example, the canonical $\D$-new path $\p_\D(\mathbf i,3)$ is the blue path in Figure~\ref{fig_example_AD_2}.
\begin{figure}[b]
	\begin{subfigure}{0.4\textwidth}
	\centering
\begin{tikzpicture}[scale = 0.6]
\tikzset{every node/.style = {font = \footnotesize}}
\tikzset{red line/.style = {line width=0.5ex, red, semitransparent}}
\tikzset{blue line/.style = {line width=0.5ex, blue, semitransparent}}

\draw(0,-0.5)--(0,1)--(5,3.5)--(5,8);
\draw (1,-0.5)--(1,1)--(0,1.5)--(0,4)--(1,4.5)--(1,5.5)--(2,6)--(2,6.5)--(4,7.5)--(4,8);
\draw(2,-0.5)--(2,1.5)--(1,2)--(1,3.5)--(2,4)--(2,5)--(3,5.5)--(3,6)--(4,6.5)--(4,7)--(3,7.5)--(3,8);
\draw (3,-0.5)--(3,0.5)--(4,1)--(4,2.5)--(3,3)--(3,4.5)--(4,5)--(4,6)--(2,7)--(2,8);
\draw (4,-0.5)--(4,0)--(5,0.5)--(5,3)--(4,3.5)--(4,4.5)--(1,6)--(1,8);
\draw (5,-0.5)--(5,0)--(3,1)--(3,2)--(2,2.5)--(2,3.5)--(0,4.5)--(0,8);

\node[below] at (0,-1) {$\ell_6$};
\node[below] at (1,-1) {$\ell_5$};
\node[below] at (2,-1) {$\ell_4$};
\node[below] at (3,-1) {$\ell_3$};
\node[below] at (4,-1) {$\ell_2$};
\node[below] at (5,-1) {$\ell_1$};

\fill[RoyalPurple!20, semitransparent, pattern= north east lines,]
(5,-0.5)--(5,0)--(4.5,0.25)--(4,0)--(4,-0.5);
\node at (4.5,-0.7) {$R_1$};

\fill[RedOrange!20, semitransparent]
(4,-0.5)--(4,0)--(5,0.5)--(5,3)--(4,3.5)--(4,4.5)--(3.5,4.75)--(3,4.5)--(3,3)--(4,2.5)--(4,1)--(3,0.5)--(3,-0.5);
\node at (3.5,-0.7) {$R_2$};

\fill[pattern color = Magenta, pattern= north west lines, semitransparent]
(3,-0.5)--(3,0.5)--(4,1)--(4,2.5)--(3,3)--(3,4.5)--(4,5)--(4,6)--(3.5,6.25)--(3,6)--(3,5.5)--(2,5)--(2,4)--(1,3.5)--(1,2)--(2,1.5)--(2,-0.5);
\node at (2.5,-0.7) {$R_3$};

\fill[blue!20, semitransparent] 
(2,-0.5)--(2,1.5)--(1,2)--(1,3.5)--(2,4)--(2,5)--(3,5.5)--(3,6)--(4,6.5)--(4,7)--(3.5,7.25)--(2,6.5)--(2,6)--(1,5.5)--(1,4.5)--(0,4)--(0,1.5)--(1,1)--(1,-0.5);
\node at (1.5,-0.7) {$R_4$};

\fill[ForestGreen!30, semitransparent]
(1,-0.5)--(1,1)--(0.5,1.25)--(0,1)--(0,-0.5);
\node at (0.5,-0.7) {$R_5$};

\node[above] at (4.5,0.25) {\footnotesize $t_{15}$};
\node[above] at (3.5,0.75) {\footnotesize $t_{14}$};
\node[above] at (0.5,1.25) {\footnotesize $t_{13}$};
\node[above] at (1.5,1.75) {\footnotesize $t_{12}$};
\node[above] at (2.5,2.25) {\footnotesize $t_{11}$};
\node[above] at (3.5,2.75) {\footnotesize $t_{10}$};
\node[above] at (4.5,3.25) {\footnotesize $t_{9}$};
\node[above] at (1.5,3.75) {\footnotesize $t_{8}$};
\node[above] at (0.5,4.25) {\footnotesize $t_{7}$};
\node[above] at (3.5,4.75) {\footnotesize $t_{6}$};
\node[above] at (2.5,5.25) {\footnotesize $t_{5}$};
\node[above] at (1.5,5.75) {\footnotesize $t_{4}$};
\node[above] at (3.5,6.25) {\footnotesize $t_{3}$};
\node[above] at (2.5,6.75) {\footnotesize $t_{2}$};
\node[above] at (3.5,7.25) {\footnotesize $t_{1}$};

\end{tikzpicture}
\caption{$R_i$.}\label{fig_AD_Ri}
\end{subfigure}
\begin{subfigure}{0.4\textwidth}
	\centering
\begin{tikzpicture}[scale = 0.6]
\tikzset{every node/.style = {font = \footnotesize}}
\tikzset{red line/.style = {line width=0.5ex, red, semitransparent}}
\tikzset{blue line/.style = {line width=0.5ex, blue, semitransparent}}

\draw(0,-0.5)--(0,1)--(5,3.5)--(5,8);
\draw (1,-0.5)--(1,1)--(0,1.5)--(0,4)--(1,4.5)--(1,5.5)--(2,6)--(2,6.5)--(4,7.5)--(4,8);
\draw(2,-0.5)--(2,1.5)--(1,2)--(1,3.5)--(2,4)--(2,5)--(3,5.5)--(3,6)--(4,6.5)--(4,7)--(3,7.5)--(3,8);
\draw (3,-0.5)--(3,0.5)--(4,1)--(4,2.5)--(3,3)--(3,4.5)--(4,5)--(4,6)--(2,7)--(2,8);
\draw (4,-0.5)--(4,0)--(5,0.5)--(5,3)--(4,3.5)--(4,4.5)--(1,6)--(1,8);
\draw (5,-0.5)--(5,0)--(3,1)--(3,2)--(2,2.5)--(2,3.5)--(0,4.5)--(0,8);

\node[below] at (0,-0.5) {$\ell_6$};
\node[below] at (1,-0.5) {$\ell_5$};
\node[below] at (2,-0.5) {$\ell_4$};
\node[below] at (3,-0.5) {$\ell_3$};
\node[below] at (4,-0.5) {$\ell_2$};
\node[below] at (5,-0.5) {$\ell_1$};

\draw[red, line width = 0.5ex, dotted, ->] (3.05,-0.5)--(3.05,0.5)--(4.05,1)--(4.05,2.5)--(3.5,2.8)--(2.5,2.3)--(2.05,2.5)--(2.05,3.5)--(1.5,3.8)--(0.95,3.5)--(0.95,2)--(1.95,1.5)--(1.95,-0.5);

\draw[blue, very thick,->] (3,-0.5)--(3,0.5)--(4,1)--(4,2.5)--(3.5,2.7)--(1.5,1.75)--(2,1.5)--(2,-0.5);

\node[above] at (4.5,0.25) {\footnotesize $t_{15}$};
\node[above] at (3.5,0.75) {\footnotesize $t_{14}$};
\node[above] at (0.5,1.25) {\footnotesize $t_{13}$};
\node[above] at (1.5,1.75) {\footnotesize $t_{12}$};
\node[above] at (2.5,2.25) {\footnotesize $t_{11}$};
\node[above] at (3.5,2.75) {\footnotesize $t_{10}$};
\node[above] at (4.5,3.25) {\footnotesize $t_{9}$};
\node[above] at (1.5,3.75) {\footnotesize $t_{8}$};
\node[above] at (0.5,4.25) {\footnotesize $t_{7}$};
\node[above] at (3.5,4.75) {\footnotesize $t_{6}$};
\node[above] at (2.5,5.25) {\footnotesize $t_{5}$};
\node[above] at (1.5,5.75) {\footnotesize $t_{4}$};
\node[above] at (3.5,6.25) {\footnotesize $t_{3}$};
\node[above] at (2.5,6.75) {\footnotesize $t_{2}$};
\node[above] at (3.5,7.25) {\footnotesize $t_{1}$};

\end{tikzpicture}
\caption{$\ell_3 \to \ell_6 \to \ell_1 \to \ell_4$ (dotted red) and $\ell_3 \to \ell_6 \to \ell_4$ (blue).}
\label{fig_example_AD_2}
\end{subfigure}
\caption{$R_i$ and $\D$-new paths in Example~\ref{example_AD}.}
\label{fig_example_AD}
\end{figure}
\end{example}

Combining Propositions~\ref{prop_DD} and \ref{prop_AD}, we prove Theorem~\ref{thm_DD_AD_combine}.
\begin{proof}[Proof of Theorem~\ref{thm_DD_AD_combine}]
	In case of $(\ad_{n-1}, \ad_n) = (\D,\D)$, the result follows from Proposition~\ref{prop_DD}. 
	For the case of  $(\ad_{n-1}, \ad_n) = (\A,\D)$, the result follows 
	\begin{itemize}
		\item from~\cite[Theorem~6.7]{CKLP} when $k = 0$, 
		\item from Lemma~\ref{lem_AD_DA} and ~\cite[Theorem~6.7]{CKLP} when $k = n-1$, 
		\item from Proposition~\ref{prop_AD} when $0 < k < n-1$. \qedhere
	\end{itemize}
\end{proof}

On the other hand, from the proofs of Propositions~\ref{prop_DD} and \ref{prop_AD}, we obtain the following consequence.
		Let $\ad = (\ad_1,\dots,\ad_n) \in \{\A,\D\}^n$, $I = (0,\dots,0,k)$ with $k \leq n-1$, and $\mathbf i = \is$.
	Assume that $n \geq 2$ and $\ad_n = \D$.   
	For each $j  \in [\barn]$, the number of rigorous paths having maximal peak at $t_j$  is $1$ or $2$.  Moreover, if two rigorous paths 
	$\p$ and $\p'$ have the same maximal peak $t_j$, then there are two possibilities:
	\begin{enumerate}
		\item either the region enclosed by one of the paths is contained in that of the other, or
		\item the node $t_j$ is on the wire $\ell_{n+1}$ and only one of them is a \textit{canonical} $\D$-new path. (The other one is automatically 
		a non-canonical $\D$-new path.)
	\end{enumerate}
For the purpose of later use, we define the following paths.
\begin{definition}\label{def_path_j}
For each $j \in [\barn]$, define $\p_j$ to be 
	\begin{itemize}
		\item $\p$ if there exists only one rigorous path $\p$ having the maximal peak $t_j$,
		\item the path enclosing the larger region for Case (1),
		\item the canonical $\D$-new path for Case (2).
	\end{itemize}
	For the remaining rigorous paths, we label them as follows (cf. Figure \ref{fig_p0_and_pi_DD} and \ref{fig_p0_AD}).
	\begin{itemize}
		\item In case of $\ad_{n-1} = \D$, define (whenever possible)
		\[
		\begin{split}
			\widetilde{\p}_0 &\coloneqq (\ell_{n-k-1} \to \ell_n \to \ell_{n-k}), \\
			\widetilde{\p}_i & \coloneqq (\ell_{n-i} \to \ell_{n+1} \to \ell_{n-i+1}) \quad \text{ for } 2 \leq i \leq k.
		\end{split}
		\]
		\item In case of $\ad_{n-1} = \A$, define (whenever possible)
		\[
		\widetilde{\p}_0 \coloneqq (\ell_{k+1} \to \ell_1 \to \ell_{k+2}).
		\]
	\end{itemize}
	
\end{definition}

\begin{example}\label{ex_Bis_0} The following examples exhibit $\gamma_i$'s as well as 
$\widetilde{\p}_j$'s defined in Definition \ref{def_path_j}.
	\begin{enumerate}
		\item Let  $\mathbf i \coloneqq  \isI{(\A,\A,\D)}{0,0,1} = (2,1,3,2,1,3) \in \Rn{4}$. Then, 
		for $s \neq 2$, there exists only one rigorous path whose maximal peak is $t_s$. 
		For $s = 2$,  there are two paths 
		\[
		\p \coloneqq (\ell_2 \to \ell_1 \to \ell_3) \quad \text{ and } \quad
		\p' \coloneqq (\ell_2 \to \ell_4 \to \ell_1 \to \ell_3)
		\]
		whose maximal peak is $t_2$. 
		Then the region enclosed by $\p$ is $\cham{2} \cup \cham{4}$ and the region enclosed by $\p'$ is $\cham{2} \cup \cham{3} \cup \cham{4}$ (see Figure~\ref{fig_213213}). Thus we have that 
		\[
		\p_2 = \p', \quad \widetilde{\p}_0 = \p. 
		\]
	\item Let $\mathbf i \coloneqq \isI{(\D,\D,\D)}{0,0,1} = (1,2,3,2,1,2)$. Then, for $s \neq 2$, there exists only one rigorous path whose maximal peak is $t_s$.
		For $s = 2$, there are two paths
		\[
		\p \coloneqq (\ell_1 \to \ell_3 \to \ell_2) \quad \text{ and } \quad
		\p' \coloneqq (\ell_1 \to \ell_3 \to \ell_4 \to \ell_2)
		\]
		whose maximal peak is $t_2$. 
		Since the region enclosed by $\p$ is $\cham{2} \cup \cham{3}$ and that of $\p'$ is $\cham{2}\cup \cham{3} \cup \cham{4}$ 
		(see Figure~\ref{fig_123212}), we obtain
		\[
		\p_2 = \p', \quad \widetilde{\p}_0 = \p. 
		\]
	\item Let $\mathbf i \coloneqq \isI{(\D,\D,\D)}{0,0,2} = (1,3,2,1,3,2) \in \Rn{4}$.
	Then, for $s = 3$, there are two paths
	\[
	\p \coloneqq (\ell_1 \to \ell_4 \to \ell_2) \quad \text{ and }
	\quad \p' \coloneqq (\ell_3 \to \ell_1 \to \ell_4)
	\]
	whose maximal peak is $t_3$ (see Figure~\ref{fig_132132}).
	In this case, the regions enclosed by these paths are not contained in each other. 
	Since $\p' = \p_\D(\mathbf i,1)$ is canonical (while $\p$ is not canonical), it follows that
	\[
	\p_3 = \p', \quad \widetilde{\p}_2 = \p.
	\]
	In this case, $\widetilde{\p}_0$ is not defined. (Indeed, $(k,n)= (2,3)$ and $n-k-1 = 0$ so that $|\GP(\mathbf i)| = \bar{n} + k - 1 = 7$. 
	See Proposition \ref{prop_DD}.)
	\end{enumerate}
\end{example}
\begin{figure}[b]
	\begin{subfigure}{0.4\textwidth}
		\centering
	\begin{tikzpicture}[scale = 0.7]
	\tikzset{every node/.style = {font = \footnotesize}}
	\tikzset{red line/.style = {line width=0.5ex, red, semitransparent}}
	\tikzset{blue line/.style = {line width=0.5ex, blue, semitransparent}}
	
	
	\draw (0,0)--(0,1.5)--(3,3)--(3,4.5);
	\draw (1,0)--(1,1.5)--(0,2)--(0,3)--(2,4)--(2,4.5);
	\draw (2,0)--(2,1)--(3,1.5)--(3,2.5)--(2,3)--(2,3.5)--(1,4)--(1,4.5);
	\draw (3,0)--(3,1)--(2,1.5)--(2,2)--(1,2.5)--(1,3)--(0,3.5)--(0,4.5);
	

	\node[below] at (0,0) {$\ell_4$};
	\node[below] at (1,0) {$\ell_3$};
	\node[below] at (2,0) {$\ell_2$};
	\node[below] at (3,0) {$\ell_1$};
	
	\node at (2.5,0.5) {$\cham{6}$};
	\node at (0.5,1) {$\cham{5}$};
	\node at (1.5,1.5) {$\cham{4}$};
	\node at (2.5,2) {$\cham{3}$};
	\node at (0.5,2.5) {$\cham{2}$};
	\node at (1.5,3) {$\cham{1}$};
	

\draw[red, line width = 0.5ex, dotted, ->]
(2.05,0)--(2.05,1)--(3.05,1.5)--(3.05,2.5)--(2.5,2.8)--(1.5,2.3)--(1.05,2.5)--(1.05,3)--(0.5,3.3)--(-0.05,3)--(-0.05,2)--(0.95,1.5)--(0.95,0);

\draw[blue, very thick,->] 
(2,0)--(2,1)--(2.5,1.25)--(2,1.5)--(2,2)--(1,2.5)--(1,3)--(0.5,3.25)--(0,3)--(0,2)--(1,1.5)--(1,0);
	
	\end{tikzpicture}
	\caption{$\mathbf i = (2,1,3,2,1,3)$. $\ell_2 \to \ell_1 \to \ell_3$ (blue) and $\ell_2 \to \ell_4 \to \ell_1 \to \ell_3$ (dotted red).}\label{fig_213213}
	\end{subfigure}\vspace{1em}
	\begin{subfigure}{0.4\textwidth}
	\centering
	\begin{tikzpicture}[scale = 0.7]
	\tikzset{every node/.style = {font = \footnotesize}}
	\tikzset{red line/.style = {line width=0.5ex, red, semitransparent}}
	\tikzset{blue line/.style = {line width=0.5ex, blue, semitransparent}}
	
	
	\draw (0,0)--(0,1.5)--(3,3)--(3,4.5);
	\draw (1,0)--(1,1)--(2,1.5)--(2,2)--(1,2.5)--(1,3)--(2,3.5)--(2,4.5);
	\draw (2,0)--(2,1)--(0,2)--(0,3.5)--(1,4)--(1,4.5);
	\draw (3,0)--(3,2.5)--(0,4)--(0,4.5);
	
	

	\node[below] at (0,0) {$\ell_4$};
	\node[below] at (1,0) {$\ell_3$};
	\node[below] at (2,0) {$\ell_2$};
	\node[below] at (3,0) {$\ell_1$};
	
	\node at (1.5,0.5) {$\cham{6}$};
	\node at (0.5,1) {$\cham{5}$};
	\node at (1.5,1.7) {$\cham{4}$};
	\node at (2.5,2) {$\cham{3}$};
	\node at (1.5,2.7) {$\cham{2}$};
	\node at (0.5,3) {$\cham{1}$};
	
	\draw[red, line width = 0.5ex, dotted, ->]
	(3.05,0)--(3.05,2.5)--(1.5,3.3)--(0.95,3)--(0.95,2.5)--(1.4,2.25)--(0.4,1.75)--(1.95,1)--(1.95,0);

	\draw[blue, very thick,->] 
	(3,0)--(3,2.5)--(1.5,3.25)--(1,3)--(1,2.5)--(2,2)--(2,1.5)--(1.5,1.25)--(2,1)--(2,0);

	%
	
	\end{tikzpicture}
	\caption{$\mathbf i = (1,2,3,2,1,2)$ $\ell_1 \to \ell_3 \to \ell_4 \to \ell_2$ (dotted red) and $\ell_1 \to \ell_3 \to \ell_2$ (blue).}\label{fig_123212}
\end{subfigure}
\caption{Wiring diagrams and chambers.}
\end{figure}

\section{Small toric resolutions of string polytopes with small indices}\label{section_small_resolution_main}
In this section, we associate a certain Bott manifold $B_{\mathbf i}$, an iterated $\C P^1$-bundle,  to each reduced word $\mathbf i$. 
Moreover, we prove that a small toric resolution of the toric variety $X_{\Delta_{\bf i}(\lambda)}$ of the string polytope $\Delta_{\mathbf i}$ 
can be obtained by blowing up the Bott manifold $B_{\mathbf i}$ under the assumption that  $\mathbf i$ has {\em small indices} 
(see Theorem~\ref{thm_main}).

Recall from~\eqref{eq_def_of_nbar} that the length of $\mathbf i$ is $\barn = n(n+1)/2$, which is the same as the dimension of the string polytope $\Delta_{\mathbf i}(\lambda)$ when $\lambda$ is regular dominant.
Among the elements in $\GP(\mathbf i)$, we have $\barn$  rigorous paths $\{\p_j \mid j \in [\barn]\}$ as in Definition~\ref{def_path_j}
and the paths assign integral vectors 
\[
\bfw_j \coloneqq \bfw_{\p_j} \quad \text{ for } j \in [\barn],
\]
where $\bfw_{\p}$ is the coefficient vector (with respect to chamber variables) of the string inequality for the path $\p$ (see~\eqref{eq_wp_and_vj}). 
Note that the matrix $[\bfws{1} \ \cdots \ \bfws{\barn}]$ is a lower triangular matrix whose diagonal entries are all $1$. 
On the other hand, we can associate $\barn$ integral vectors $\mathbf{v}_1,\dots,\mathbf{v}_\barn$ to  $\lambda$-inequalities
such that the matrix $[\mathbf{v}_1 \ \cdots \ \mathbf{v}_\barn]$ is a lower triangular matrix whose diagonal entries are all $-1$. 
See also ~\eqref{eq_wp_and_vj} for the definition of $\bfv_j$. Using these vectors, one obtains a Bott manifold (see Theorem \ref{thm_Bott_manifold} and Proposition~\ref{prop_primitive_collections_Bott}).
\begin{definition}\label{def_Bi}
	Let $\mathbf i = \is$ with $I = (0,\dots,0,k)$ and $0 \leq k \leq n-1$. 
	Define $\Bis$ to be the Bott manifold 
	determined by the vectors $\{\mathbf{v}_j, \bfws{j}  \mid 1 \leq j\leq \barn\}$ in the sense of Theorem \ref{thm_Bott_manifold}. We denote by $\Sigma_{\bf i}$ the fan of~$\Bis$.
\end{definition}
\begin{remark}\label{rmk_GC_and_Bott}
	For $\mathbf i \in \Rn{n+1}$ satisfying $\mathbf i \sim \isI{\ad}{0,\dots,0}$ for some $\ad \in \{\A,\D\}^n$, the string polytope $\Delta_{\mathbf i}(\lambda)$ is unimodularly equivalent to the Gelfand--Cetlin polytope $\GC(\lambda)$ by~\cite[Theorem~6.7]{CKLP}. We call such a reduced word~$\mathbf i$  \defi{Gelfand--Cetlin type}. Moreover, it has been proved in~\cite[Proposition~3.1.2]{BCKV} that the toric variety $X_{\GC(\lambda)}$ of
the Gelfand--Cetlin polytope $\GC(\lambda)$ admits a small toric resolution for any dominant integral weight $\lambda$. Indeed, for $\mathbf i_0 \coloneqq (1,2,1,3,2,1,\dots,n,n-1,\dots,1)$, considering the affine transformation sending the string polytope $\Delta_{\mathbf i_0}(\lambda)$ to $\GC(\lambda)$ constructed in~\cite[Section~6]{CKLP}, one can see that 
the fan of the Bott manifold $B_{\mathbf i_0}$ is the same as that of the small toric desingularization of $X_{\GC(\lambda)}$ constructed in~\cite[Proposition~3.1.2]{BCKV} for a regular dominant integral weight $\lambda$.
\end{remark}

We introduce the following notion. 
\begin{definition}\label{def_small}
	We say that ${\bf i}$ has {\em small indices} if $\ind_{\ad}(\mathbf i) = (0,\dots,0,k)$ for some $\ad \in \{\A,\D\}^n$ and \\ $k \leq \kappa(\delta_{n-1},\delta_n)$, where
	\[
	\kappa(\delta_{n-1},\delta_n) \coloneqq \begin{cases}
	2 & \text{ if } \delta_{n-1} = \delta_n, \\
	n-1 & \text{ if } \delta_{n-1} \neq \delta_n. 
	\end{cases}
	\]
\end{definition}
We will see later that every reduced word in $\Rn{4}$ has small indices.
We construct a smooth projective toric variety for ${\bf i}$ having small indices.  Suppose that $\ad_n = \D$.
Recall from Theorem~\ref{thm_DD_AD_combine} that
\begin{equation}\label{eq_set_GP_small_indices}
\GP(\mathbf i) 
= \begin{cases}
\{\p_j \mid j \in [\barn]\} & \text{ if } k = 0; \text{ or } \ad_{n-1}= \A \text{ and }k = n-1, \\
\{\p_j \mid j \in [\barn]\} \cup \widetilde{\p}_0& \text{ if } \ad_{n-1} = \A \text{ and } 0 < k < n-1; \text{ or }\ad_{n-1} = \D \text{ and } k=1, \\
\{\p_j \mid j \in [\barn]\} \cup \widetilde{\p}_2 & \text{ if } \ad_{n-1} = \D \text{ and }(k,n) = (2,3), \\
\{\p_j \mid j \in [\barn]\} \cup \widetilde{\p}_0 \cup \widetilde{\p}_2 & \text{ if }\ad_{n-1} = \D \text{ and } k = 2, n>3, 
\end{cases}
\end{equation}
in which the number of rigorous paths is at most $\barn + 2$.

\begin{proposition}\label{prop_w_p_i_sum_of_w_and_v}
	Each vector $\bfw_{\widetilde{\p}_i}$ for $i = 0,2$ can be expressed as a linear combination of $\bfw_j$'s and $\bfv_j$'s as follows. 
	\begin{enumerate}
				\item When $\ad_{n-1} = \D$ and $k = 1,2$:
		\[
		\begin{split}
		\bfw_{\widetilde{\p}_0} &= \bfw_{\barn-(n+k)} + \bfv_{\barn-2k} + \bfw_{\barn-k+1}, \\
		\bfw_{\widetilde{\p}_2} &= \bfw_{\barn-3} + \bfv_{\barn-2} + \bfw_{\barn-1}.
		\end{split}
		\]
		\item When $\ad_{n-1} = \A$ and $0 < k < n-1$:
		\[
		\bfw_{\widetilde{\p}_0} = \bfw_{\barn-(n+k)} + \bfv_{\barn-n} + \bfw_{\barn-k+1}.
		\]
	\end{enumerate}
\end{proposition}
\begin{proof}
	First we consider the case $\ad_{n-1} = \D$. Then using Figure~\ref{fig_p0_and_pi_DD}, one can see that
	\[
	\begin{split}
		\bfw_{\widetilde{\p}_0} &= \mathbf{e}_{\barn-(n+k)} + \mathbf{e}_{\barn-2k-1} \\
&= (\mathbf{e}_{\barn-(n+k)} + \mathbf{e}_{\barn-2k-1} + \mathbf{e}_{\barn-2k})
+ (-\mathbf{e}_{\barn-2k} - \mathbf{e}_{\barn-k+1}) + \mathbf{e}_{\barn-k+1} \\
&= \bfws{\barn-(n+k)} + \bfv_{\barn-2k} + \bfws{\barn-k+1}.
\end{split}
	\]
	For $\widetilde{\p}_2$, the path $\p_2$ exists only when $k = 2$ and we have
	\[
	\bfw_{\widetilde{\p}_2} = \mathbf{e}_{\barn-3} + \mathbf{e}_{\barn-1} = \bfw_{\barn-3} + \bfv_{\barn-2} + \bfw_{\barn-1}.
	\]
	This proves the claim when $\ad_{n-1} = \D$.
	
	Now assume $\ad_{n-1} = \A$. From Figure~\ref{fig_p0_AD}, we can easily see that
	\[
	\begin{split}
		\bfw_{\widetilde{\p}_0} &= \mathbf{e}_{\barn -(n+k)} + \mathbf{e}_{\barn-n+1}\\
		&= (\mathbf{e}_{\barn-(n+k)} + \mathbf{e}_{\barn-n+1} + \mathbf{e}_{\barn-n}) + (-\mathbf{e}_{\barn-n} - \mathbf{e}_{\barn-k+1}) 
		+ \mathbf{e}_{\barn-k+1}\\
		&= \bfws{\barn-(n+k)} + \bfv_{\barn-n} + \bfws{\barn-k+1},
	\end{split}
	\]
	This completes the proof.
\end{proof}

\begin{example}\label{ex_Bis} In the following examples, 
we compute $\widetilde{\p}_i$'s for $i=0$ or $2$ as well as $\bfw_j$'s and $\bfv_j$'s explicitly, 
and also confirm Proposition \ref{prop_w_p_i_sum_of_w_and_v} in each case.
	\begin{enumerate}
		\item Let  $\mathbf i \coloneqq  \isI{(\A,\A,\D)}{0,0,1} = (2,1,3,2,1,3) \in \Rn{4}$. From Example~\ref{ex_Bis_0}(1) and Figure~\ref{fig_213213}, 
		we obtain 
		\[
		[\bfv_1 \ \cdots \ \bfv_6 ~\vert~ \bfws{1} \ \cdots \ \bfws{6} ] = 
		\left[
		\begin{array}{cccccc|cccccc}
		-1 & 0 & 0 & 0 & 0 & 0 & 1 & 0 & 0 & 0 & 0 & 0 \\
		0 & -1 & 0 & 0 & 0 & 0 & 1 & 1 & 0 & 0 & 0 & 0 \\
		0 & 0 & -1 & 0 & 0 & 0 & 1 & 1 & 1 & 0 & 0 & 0 \\
		-1 & 0 & 0 & -1 & 0 & 0 & 1 & 1 & 1 & 1 & 0 & 0 \\
		0 & -1 & 0 & 0 & -1 & 0 & 0 & 0 & 0 & 0 & 1 & 0 \\
		0 & 0 & -1 & 0 & 0 & -1 & 0 & 0 & 0 & 0 & 0 & 1
		\end{array}
		\right].
		\]
		Note that $|\GP(\mathbf{i})| = 7$ and the path $\widetilde{\p}_0 = \ell_2 \to \ell_1 \to \ell_3$ defines a vector $(0,1,0,1,0,0)$. 
		Moreover, we have that
		\begin{equation*}
		(0,1,0,1,0,0) = \bfws{2} + \bfv_3 + \bfws{6}.
		\end{equation*}
		\item Let $\mathbf i \coloneqq \isI{(\D,\D,\D)}{0,0,1} = (1,2,3,2,1,2)$. Following observations in Example~\ref{ex_Bis_0}(2) and 
		Figure~\ref{fig_123212}, we have
		\[
		[\bfv_1 \ \cdots \ \bfv_6 ~\vert~ \bfws{1} \ \cdots \ \bfws{6} ] = 
		\left[
		\begin{array}{cccccc|cccccc}
		-1 & 0 & 0 & 0 & 0 & 0 & 1 & 0 & 0 & 0 & 0 & 0 \\
		0 & -1 & 0 & 0 & 0 & 0 & 1 & 1 & 0 & 0 & 0 & 0 \\
		0 & 0 & -1 & 0 & 0 & 0 & 1  & 1 & 1 & 0 & 0 & 0 \\
		0 & -1 & 0 & -1 & 0 & 0 & 1 & 1 & 0 & 1 & 0 & 0 \\
		-1 & 0 & 0 & 0 & -1 & 0 & 0 & 0 & 0 & 1 & 1 & 0 \\
		0 & -1 & 0 & -1 & 0 & -1 & 0 & 0 & 0 & 0 & 0 & 1
		\end{array}
		\right].
		\]
		In this case, the path $\widetilde{\p}_0 = \ell_1 \to \ell_3 \to \ell_2$ defines a vector $(0,1,1,0,0,0)$ expressed by
		\begin{equation*}
		(0,1,1,0,0,0) = \bfws{2} + \bfv_4 + \bfws{6}.
		\end{equation*}
		\item Let $\mathbf i \coloneqq \isI{(\D,\D,\D)}{0,0,2} = (1,3,2,1,3,2)$. Again by Example~\ref{ex_Bis_0}(3)
		 (also, see Example~\ref{example_132132}) and Figure \ref{fig_132132}, we may check that 
		\[
		[\bfv_1 \ \cdots \ \bfv_6 ~\vert~ \bfws{1} \ \cdots \ \bfws{6} ] = 
		\left[
		\begin{array}{cccccc|cccccc}
		-1 & 0 & 0 & 0 & 0 & 0 & 1 & 0 & 0 & 0 & 0 & 0 \\
		0 & -1 & 0 & 0 & 0 & 0 & 0 & 1 & 0 & 0 & 0 & 0\\
		0 & 0 & -1 & 0 & 0 & 0 & 1 & 1 & 1 & 0 & 0 & 0 \\
		-1 & 0 & 0 & -1 & 0 & 0 & 0 & 1 & 1 & 1 & 0 & 0 \\
		0 & -1 & 0 & 0 & -1 & 0 & 1 & 0 & 0 & 0 & 1 & 0 \\
		0 & 0 & -1 & 0 & 0 & -1 & 0 & 0 & 0 & 0 & 0 & 1
		\end{array}
		\right]. 
		\]
		In this case, we have $|\GP(\mathbf i)| = 7$ and the path $\widetilde{\p}_2 = \ell_1 \to \ell_4 \to \ell_2$ 
		defines a vector $(0,0,1,0,1,0)$. Moreover, we have that
		\[
		(0,0,1,0,1,0) = \bfw_3 + \bfv_4 + \bfw_5.
		\]
	\end{enumerate}
\end{example}

In the rest of this section, we will show that for each ${\bf i}$ having small indices, there exists a smooth complete fan whose ray generators are 
$\{ \bfw_\gamma ~|~ \gamma \in \GP(\mathbf{i}) \} \cup \{ {\bf v}_j ~|~ j \in [\barn] \}$, and this fan can be obtained 
as a refinement of the fan of the Bott manifold $B_{\bf i}$.

For each $\widetilde{\p}_i \in \GP(\mathbf{i})$ in Definition~\ref{def_path_j}, define two subsets $A_i$ and $B_i$ of $[\barn]$ such that 
\[
	\bfw_{\widetilde{\p}_i} = \sum_{a \in A_i} \bfw_a + \sum_{b \in B_i}  \bfv_b.
\]
From Proposition~\ref{prop_w_p_i_sum_of_w_and_v}, we see that $A_i \cap B_i = \emptyset$. In particular, we have 
\begin{equation}\label{eq_AiBi_not_supset_wjvj}
	\{\bfw_a \mid a \in A_i\} \cup \{\bfv_b \mid b \in B_i\} \not\supset \{\bfw_j, \bfv_j \} \quad \text{ for any }j \in [\barn].
\end{equation}
This implies that the cone 
\[
	\Cone(\{\bfw_a \mid a \in A_i\} \cup \{\bfv_b \mid b \in B_i\})
\] 
is contained in the fan $\Sigma_{\bf i}$ of the Bott manifold $B_{\bf i}$.
(See Theorem~\ref{thm_Bott_manifold}.)
For example, if $|\GP(\mathbf i)| = \barn+2$, i.e., $\ad_{n-1} = \D$ with $n > 3$ and $k = 2$, 
the paths $\widetilde{\p}_0$ and $\widetilde{\p}_2$ respectively correspond to the set $\{\bfw_{\barn-(n+2)}, \bfv_{\barn-4}, \bfw_{\barn-1}\}$ and 
the set $\{\bfw_{\barn-3}, \bfv_{\barn-2}, \bfw_{\barn-1}\}$ by Proposition~\ref{prop_w_p_i_sum_of_w_and_v}. Therefore 
\[
	\mathrm{Cone}(\bfw_{\barn-(n+2)}, \bfv_{\barn-4}, \bfw_{\barn-1}) \in \Sigma_{\bf i} \quad \text{and} \quad 
	\mathrm{Cone}(\bfw_{\barn-3}, \bfv_{\barn-2}, \bfw_{\barn-1}) \in \Sigma_{\bf i}
\]
by Lemma~\ref{lem_smooth_fan} and~\eqref{eq_AiBi_not_supset_wjvj}.

Let $\tau \in \Sigma_{\bf i}$ be such that
\begin{equation}\label{eq_def_tau_P}
\tau
\coloneqq \begin{cases}
\Cone(\bfws{\barn-(n+k)}, \bfv_{\barn-n}, \bfws{\barn-k+1}) & \text{ if }\ad_{n-1} = \A, \\ 
\Cone(\bfw_3, \bfv_4, \bfw_5) & \text{ if } \ad_{n-1} = \D \text{ and } (k,n) =(2,3),\\
\Cone(\bfws{\barn-(n+k)}, \bfv_{\barn-2k}, \bfws{\barn-k+1}) & \text{ if } \ad_{n-1} = \D \text{ and } (k,n) \neq (2,3).
\end{cases}
\end{equation}
Since $\tau \in \Sigma_{\bf i}$ for each case, we can think of the star subdivision $\Sigma_{\bf i}^{\ast}(\tau)$ of the fan $\Sigma_{\bf i}$ of $B_{\bf i}$ along the cone $\tau$. 

On the other hand, when $(\ad_{n-1}, \ad_n) = (\D,\D)$ with $k = 2$ and $n > 3$, the primitive collection $\mathrm{PC}(\Sigma_{\bf i}^*(\tau))$ consists of 
\[
\{\bfw_{\barn-(n+2)}, \bfv_{\barn-4}, \bfw_{\barn-1}\},
\{\bfw_{\widetilde{\p}_0}, \bfv_{\barn-(n+2)}\}, \{\bfw_{\widetilde{\p}_0},\bfw_{\barn-4}\},
\{\bfw_{\widetilde{\p}_0},\bfv_{\barn-1}\}, \{ \bfw_j, \bfv_j \}_{j \in [\barn]} 
\]
by Propositions~\ref{prop_primitive_collections_Bott} and~\ref{prop_PC_star_subdivision}. 
Moreover, Lemma \ref{lem_smooth_fan} implies that the cone
\begin{equation}\label{eq_def_tau_Q}
\tau_{{2}} \coloneqq \Cone(\bfws{\barn-3}, \bfv_{\barn-2}, \bfw_{\barn-1}) 
\end{equation}
 is contained in the fan $\Sigma_{\bf i}^{\ast}(\tau)$ since each element of $\mathrm{PC}(\Sigma_{\bf i}^*(\tau))$ is not contained in 
 $\{\bfw_{\barn-3}, \bfv_{\barn-2}, \bfw_{\barn-1}\}$.

\begin{definition}\label{def_fan}
	Suppose that $\mathbf i$ has small indices.
	Let $\Sigma_{\bf i}$ be the fan of the Bott manifold $\Bis$ defined in Definition~\ref{def_Bi}. We define the fan $\wS_{\bf i}$ by
\begin{equation}\label{eq_fan_tildeBiS}
\wS_{\bf i} \coloneqq \begin{cases}
\Sigma_{\bf i} & \text{ if } \lvert \GP(\mathbf{i}) \rvert = \barn, \\
\Sigma^{\ast}_{\bf i}(\tau) & \text{ if } \lvert \GP(\mathbf{i}) \rvert  = \barn+1, \\
(\Sigma^{\ast}_{\bf i}(\tau))^{\ast}(\tau_2) & \text{ if } \lvert \GP(\mathbf{i}) \rvert  = \barn+2. 
\end{cases}
\end{equation}
\end{definition}

By the property of the star subdivision (see the paragraph below Definition \ref{def_star})
and Proposition~\ref{prop_non_redundancy_inequalities}, we get the following.
\begin{proposition}\label{lem_rays_are_same}
	Suppose that $\mathbf i$ has small indices.
Then the fan $\wS_{\bf i}$ is a smooth polytopal fan. Moreover, for any regular dominant integral weight $\lambda$, we have 
\[
\{\bfu_{\rho} \mid \rho \in \wS_{\bf i}(1) \}  = \{\bfw_{\p} \mid \p \in \GP(\mathbf{i})\} \cup \{\bfv_j \mid 1 \leq j \leq \barn\} = \{ \bfu_{\rho} \mid \rho \in \Sigma_{\Delta_{\mathbf i}(\lambda)}(1)\},
\]  
where $\Sigma_{\Delta_{\mathbf i}(\lambda)}$ is the normal fan of the string polytope $\Delta_{\bf i}(\lambda)$.
\end{proposition}

Now we are ready to state our main theorem which asserts that the smooth projective toric variety $X_{\wS_{\bf i}}$ provides a small toric resolution of $X_{\Delta_{\mathbf i}(\lambda)}$ when 
$\mathbf i$ has small indices. (The proof of Theorem~\ref{thm_main} will be given in Section~\ref{section_proof_of_main}.)

\begin{theorem}\label{thm_main}
	Let $\mathbf i$ be a reduced word of the longest element in the Weyl group of $\SL_{n+1}(\C)$ and $\lambda$ a regular 
	dominant integral weight. If $\mathbf i$ has small indices,
	then the toric variety $X_{\wS_{\bf i}}$ is a small toric desingularization of
	the toric variety $X_{\Delta_{\mathbf i}(\lambda)}$. In particular, if ${\bf i}$ has small indices, then the toric variety $X_{\Delta_{\mathbf i}(\lambda)}$ admits a small toric resolution. 
\end{theorem}

We will consider the case when $\mathbf i$ does not have small indices  in Example~\ref{example_not_small}. 
In this case, our smooth projective toric variety $X_{\wS_{\bf i}}$ does not give a small toric resolution of $X_{\Delta_{\bf i}(\lambda)}$.
\begin{remark}
	By Proposition~\ref{prop_small_resolution_implies_integral}, if the toric variety $X_{\Delta_{\mathbf i}(\lambda)}$ admits a small toric resolution
	for a regular dominant weight~$\lambda$, then the string polytope $\Delta_{\mathbf i}(\mu)$ becomes an integral polytope 
	for any dominant integral weight~$\mu$.
	However,  Steinert \cite{Steinert19} recently provides an example of a non-integral string polytope.
	More precisely, for $\mathbf i =(1,3,2,1,3,2,4,3,2,1,5,4,3,2,1)$, the string polytope~$\Delta_{\mathbf i}(\varpi_3)$ is not integral (see~\cite[Example~7.5]
	{Steinert19}). Consequently, the toric variety $X_{\Delta_{\mathbf i}(\lambda)}$ cannot admit a small toric resolution for a regular dominant integral weight 
	$\lambda$. Note that we still do not know whether the toric variety $X_{\Delta_{\bf i}(\lambda)}$ admits a small (non-toric) resolution.
\end{remark}

\begin{remark}
	Harada and Yang \cite{HYstring} considered a relation between string polytopes and certain Bott manifolds. 
	They call it \textit{functorial} resolutions of toric varieties $X_{\Delta_{\mathbf i}(\lambda)}$. 
    But their resolution is not small. Indeed, if we denote by $\Sigma$ the fan of a Bott manifold considered in their paper, $\Sigma(1) \neq \Sigma_{\Delta_{\bf i}(\lambda)}(1)$. In fact, the numbers of elements are different in general since $\Sigma(1) = 2n$ and $\Sigma_{\Delta_{\bf i}(\lambda)}(1) \geq 2n$ in general. 
	Hence the resolution in~\cite{HYstring} is not a small resolution. 	
	Note that $\Sigma_{\Delta_{\bf i}(\lambda)}(1) = 2n$ if and only if $\mathbf i$ is of Gelfand--Cetlin type (see~\cite{CKLP}). However, the Bott manifold defined in~\cite{HYstring} is not the same as $B_{\mathbf i}$ even when $\mathbf i$ is of Gelfand--Cetlin type.
\end{remark}

Now we consider the string polytopes for $n \leq 3$. When $n = 1$ or $2$, then all string polytopes are Gelfand--Cetlin type, and hence they admit small toric resolutions (see~Remark~\ref{rmk_GC_and_Bott}).  When $n = 3$, there are $16$ reduced words of the longest element in $\mathfrak{S}_{4}$ by the hook length formula (see~\cite[Corollary~7.4.8]{BB05Combinatorics})
\begin{equation}\label{eq_hook_length_formula}
|\Rn{n+1}| = \frac{{n+1 \choose 2} ! }{1^n 3^{n-1} 5^{n-2} \cdots (2n-1)}.
\end{equation}
Under the equivalence relation given by $2$-moves and the involution $\iota \colon [3] \to [3]$ of the Dynkin diagram as in Example~\ref{example_Dynkin diagram automorphisms}, one can classify the reduced words into four types as we see below. In other words, there are at most four types of string polytopes up to unimodular equivalence. 
\[
\begin{split}
	&\text{Type 1. } (1,2,1,3,2,1), (1,2,3,1,2,1), (2,1,2,3,2,1), (2,3,2,1,2,3), (3,2,3,1,2,3), (3,2,1,3,2,3). \\
	&\text{Type 2. } (1,2,3,2,1,2), (3,2,1,2,3,2). \\
	&\text{Type 3. } (1,3,2,3,1,2),(1,3,2,1,3,2),(3,1,2,3,1,2),(3,1,2,1,3,2).\\
	&\text{Type 4. } (2,1,3,2,3,1), (2,3,1,2,3,1), (2,1,3,2,1,3), (2,3,1,2,1,3).
\end{split}
\]
For each type, we can compute $\ad$-indices (for a particularly chosen $\ad$ as below) and see that 
all the reduced words in $\Rn{4}$ have small indices:
\[
\begin{split}
& (1,2,1,3,2,1) = \isI{(\D,\D,\D)}{0,0,0}, \quad  (1,2,3,2,1,2) = \isI{(\D,\D,\D)}{0,0,1}, \\
& (1,3,2,1,3,2) = \isI{(\D,\D,\D)}{0,0,2}, \quad (2,1,3,2,1,3) = \isI{(\D,\A,\D)}{0,0,1}.
\end{split}
\]
Therefore we obtain the following. 
\begin{corollary}\label{cor_n3}
	For any ${\bf i} \in \Rn{4}$ and a regular dominant weight $\lambda$, 
	the toric variety $X_{\Delta_{\bf i}(\lambda)}$ of a string polytope $\Delta_{\bf i}(\lambda)$ admits a small toric resolution. 
\end{corollary}
For $n = 4$, there exist reduced words of the longest element which do not have small indices.
In Appendix~\ref{appendix_S5}, we calculate all $\delta$-indices for every reduced word in $\Rn{5}$. See Table \ref{table_S5}.

Now we illustrate some corollaries of Theorem~\ref{thm_main}. 
The following, the integrality of string polytopes, is an immediate consequence of Proposition~\ref{prop_small_resolution_implies_integral}.

\begin{corollary}\label{cor_integral_polytope}
	Suppose that $\mathbf{i} \in \Rn{n+1}$ has small indices. Then the string polytope $\Delta_{\bf i}(\lambda)$ is integral
	for any dominant integral weight $\lambda$. 
\end{corollary}

Next, let $P \subset G$ be a parabolic subgroup and $\lambda_P$ the weight corresponding to the anticanonical bundle of the partial flag variety $G/P$.
(For example, we have $\lambda_B = 2 \varpi_1 + \cdots  + 2 \varpi_n$ for a Borel subgroup $B \subset G$.)
The following corollary guarantees the reflexivity of string polytopes for the weight $\lambda_P$.

\begin{corollary}\label{cor_Fano}
	Let $P \subset G$ be a parabolic subgroup and $\lambda_P$ the weight corresponding to the anticanonical bundle of a partial flag variety $G/P$. 
		Suppose that $\mathbf{i} \in \Rn{n+1}$ has small indices.
		Then the string polytope $\Delta_{\mathbf i}(\lambda_P)$ is reflexive, and therefore 
		the toric variety of the string polytope $\Delta_{\bf i}(\lambda_P)$ is a Gorenstein Fano variety. 
\end{corollary}
\begin{proof}
	Note that the string polytope $\Delta_{\bf i}(\lambda_P)$ is integral by  Corollary~\ref{cor_integral_polytope}. 
	On the other hand, Rusinko~\cite[Corollary~7.1]{Rusinko08} proved that the string polytope $\Delta_{\bf i}(\lambda_P)$ is a reflexive polytope (after translation by a lattice vector) if  it is integral.
Thus $\Delta_{\bf i}(\lambda_P)$ is reflexive and therefore the projective toric variety of 
$\Delta_{\bf i}(\lambda_P)$ is a Gorenstein Fano variety by Theorem~8.3.4 in~\cite{CLS11toric}. 
\end{proof}

Finally we introduce some applications of Theorem \ref{thm_main} to the symplectic geometry of flag varieties.
For a given symplectic manifold $(M,\omega)$ and a Lagrangian submanifold $L \subset M$ with a suitable condition 
(weakly unobstructed for example), one can define a function $\frak{PO}$ called a {\em Landau--Ginzburg superpotential} or a {\em disc potential}. The function
$\frak{PO}$ is defined on some set called a {\em Maurer--Cartan space}. Roughly speaking, the potential function $\frak{PO}$ is a Laurent series and it encodes the number of holomorphic discs (of Maslov index two) bounded by $L$. 

Now let $\lambda$ be a regular dominant integral weight. 
Our goal is to compute the potential function $\frak{PO}$ of some Lagrangian tori in the flag variety $G/B$ 
obtained from a toric degeneration $G/B$ by using the work of Nishinou--Nohara--Ueda \cite{NNU10}. 

By Caldero~\cite{Cal}, there exists a toric degeneration 
\[
\mathfrak{X} = \{X_t \mid t \in \C\}
\text{ such that } X_t \cong G/B  \text{ for } t \in \C \setminus \{0\} \text{  and  }X_0 \cong X_{\Delta_{\mathbf i}(\lambda)}.
\] 
From $\frak{X}$, one can obtain a completely integrable system 
\[
\Phi_{\mathbf i, \lambda} \colon G/B \to \R^\barn
\]
in the sense of Harada--Kaveh such that the image $\Phi_{\mathbf i, \lambda}(G/B)$ is the string polytope 
$\Delta_{\mathbf i}(\lambda)$ (see~\cite[Corollary~3.36]{HaKa15}). In particular, every fiber $L(u) \coloneqq \Phi_{\mathbf i, \lambda}^{-1}(u)$ becomes a Lagrangian torus with respect 
to a Kostant--Kirillov--Souriau symplectic form on $G/B$. The following corollary, obtained directly from 
Theorem~\ref{thm_main}, \cite[Theorem~10.1]{NNU10}, and \cite[Theorem~1]{NNU12}, states that the potential function for each $L(u)$ can be computed 
in terms of $\Delta_{\bf i}(\lambda)$ as follows.

\begin{corollary}\label{cor_disk_potential}
Let $B$ be a Borel subgroup of $G = \mathrm{SL}_{n+1}(\C)$. 
Let $\lambda = \lambda_1 \varpi_1 + \cdots + \lambda_n \varpi_n$ be a regular dominant integral weight.
Suppose that $\mathbf{i} \in \Rn{n+1}$ has small indices.
	Then the disk potential of $L(u)$ is 
	\begin{equation}\label{equ_potentialLG}
	\mathfrak{PO}_{\mathbf i}(x) = \sum_{ \p \in \GP(\mathbf i)} e^{\langle \bfw_\p, x \rangle } T^{\langle \bfw_\p, u \rangle} + \sum_{j=1}^\barn e^{\langle \bfv_j, x \rangle} T^{\langle \bfv_j, u \rangle + \lambda_{i_j}}
	\end{equation}
	for $x \in H^1(L(u),\Lambda_0)$. 	
Here, $\Lambda_0$ is the Novikov ring $\left\{ \sum_{i=0}^{\infty} a_i T^{\mu_i} \mid a_i \in \mathbb{C}, \mu_i \in \R_{\geq 0}, \lim_{i \to \infty} \mu_i= \infty \right\}$. 
	\end{corollary}
	
\begin{remark}
	In Corollary \ref{cor_disk_potential}, note that the potential function for $L(u)$ is defined on some set (called the Maurer--Cartan space) denoted by 
	$\mcal{M}(L(u))$. 
	It is proved by Nishinou--Nohara--Ueda \cite[Theorem~1]{NNU12} that $H^1(L(u), \Lambda_0) \subset \mcal{M}(L(u))$ when 
	the toric variety $X_{\Delta_{\bf i}(\lambda)}$ of $\Delta_{\bf i}(\lambda)$ admits a small toric resolution. 
	The formula of $\frak{PO}_{\bf i}$ in Corollary \ref{cor_disk_potential} is indeed for the restriction of $\frak{PO}_{\bf i}$ to $H^1(L(u), \Lambda_0)$.  
\end{remark}

\begin{example}
	Let $\mathbf i = (1,3,2,1,3,2) \in \Rn{4}$ and $\lambda =   \lambda_1 \varpi_1 + \lambda_2 \varpi_2 + \lambda_3 \varpi_3$. (See Example~\ref{example_132132} for the inequalities for the corresponding string polytope.) Then the potential function for the Lagrangian torus fiber determined by the string polytope $\Delta_{\bf i}(\lambda)$ is given by
	\[
	\begin{split}
	\mathfrak{PO}_{\mathbf i} &= e^{x_1 + x_3 + x_5} T^{u_1 + u_3 + u_5} + e^{x_3 + x_5} T^{u_3 + u_5}
	+ e^{x_5} T^{u_5} + e^{x_6} T^{u_6} \\
	&\quad + e^{x_2 + x_3 + x_4 } T^{u_2 + u_3 + u_4} + e^{x_3 + x_4} T^{u_3 + u_4} 
	+ e^{x_4} T^{u_4} \\
	&\quad \quad  + e^{-x_1 - x_4}T^{-u_1 - u_4 + \lambda_1} + e^{-x_2 - x_5} T^{-u_2 - u_5 + \lambda_3} + e^{-x_3 - x_6 }T^{-u_3 - u_6 + \lambda_2} \\
	& \quad \quad \quad +  e^{-x_4} T^{-u_4 + \lambda_1} 
	+ e^{-x_5} T^{-u_5 + \lambda_3} + e^{-x_6} T^{-u_6 + \lambda_2}.
	\end{split}
	\]
	This potential function can be written as 
	\[
	\begin{split}
	\mathfrak{PO}_{\mathbf i} &= y_1y_3y_5 + y_3y_5 + y_5 + y_6 + y_2y_3y_4 + y_3y_4 + y_4 \\
	&\quad + \frac{q_1}{y_1y_4} + \frac{q_3}{y_2y_5} + \frac{q_2}{y_3y_6} + \frac{q_1}{y_4} + \frac{q_3}{y_5} + \frac{q_2}{y_6},
	\end{split}
	\]
	by setting 
	$q_i = T^{\lambda_i}$ for $1 \leq i \leq 3$ 
	and $y_j = e^{x_j} T^{u_j}$ for $1 \leq j \leq 6$. 
\end{example}

\begin{remark}
	Suppose that reduced words $\mathbf i$ and $\mathbf i'$ in $\Rn{n}$ are related by a $3$-move. Namely, 
	\[
	\mathbf{i} = (\mathbf{i}_1, i, i+1, i, \mathbf{i}_2) \leftrightarrow (\mathbf{i}_1, i+1, i, i+1, \mathbf{i}_2) = \mathbf{i}'.
	\] 
	Berenstein and Zelevinsky proved that there is a piecewise-linear automorphism ${}_{\mathbf i'} T_{\mathbf i} \colon \R^{\barn} \to \R^{\barn}$ preserving the lattice and such that 
	\[
	{}_{\mathbf i'} T_{\mathbf i}(\Delta_{\mathbf i}(\lambda) \cap \Z^{\barn}) = \Delta_{\mathbf i'}(\lambda) \cap \Z^{\barn}
	\] 
	for a dominant integral weight $\lambda$ (see~\cite[Theorem~2.7]{BeZe93}). 
More precisely, suppose that the $3$-move occupies positions $k, k+1,k+2$. 
Then, the map ${}_{\mathbf i'} T_{\mathbf i}$ leaves all the components of a coordinate $t = (t_1,\dots,t_{\barn})$ except $t_k, t_{k+1}, t_{k+2}$ and changes $(t_k, t_{k+1}, t_{k+2})$ to 
	\begin{equation}\label{troprel}
	(\max(t_{k+2},t_{k+1} - t_k), t_k + t_{k+2}, \min(t_k, t_{k+1} -t_{k+2})).
	\end{equation}
For two words ${\mathbf i}$ and ${\mathbf i'}$ which are related by a $3$-move and have small indices, 
two potentials $\mathfrak{PO}_{\mathbf i}$ and $\mathfrak{PO}_{\mathbf i'}$ are related by the coordinate change whose tropical lift is~\eqref{troprel}. 
	
\end{remark}

\section{Proof of Theorem~\ref{thm_main}}
\label{section_proof_of_main}
In this section, we will give a proof of Theorem~\ref{thm_main}. Before giving a proof, we prepare one lemma. 
\begin{proposition}\label{prop_equiv_small_indices}
	For $\mathbf i \in \Rn{n+1}$, $\mathbf i$ has small indices if and only if $\mathbf i \sim \mathbf i_{\ad}(0,\dots,0,\kappa(\ad_{n-1},\ad_n))$ for some $\ad \in \{\A,\D\}^n$.
\end{proposition}
\begin{proof}
	First of all, if $\mathbf i \sim \mathbf i_{\ad}(0,\dots,0,\kappa(\ad_{n-1},\ad_n))$ for some $\ad \in \{\A,\D\}^n$, then $\mathbf i$ has small indices by Lemma~\ref{lemma_index_of_i_sigma_I}. Therefore, it is enough to prove that the converse statement holds.  
	The case of $n = 2$ is trivial since there are only two reduced word $(1,2,1)$ and $(2,1,2)$.
	Thus we assume that $n \geq 3$. 
	Note that if $\ind_{\ad}(\mathbf i) = (0,\dots,0)$, then $\mathbf i \sim \isI{\ad}{0,\dots,0}$. This is because if $\mathbf i \sim \mathbf i'$ then $E_{\bullet}(0)(\mathbf i) \sim E_{\bullet}(0)(\mathbf i')$. Hence we get
	\[
	\mathbf i'' \coloneqq C_{\ad_n}(\mathbf i) \sim \isI{(\ad_1,\dots,\ad_{n-1})}{0,\dots,0} =: \mathbf i'. 
	\]
	Now we observe the last part of the word $\mathbf i'$:
	\begin{enumerate}
		\item $(\ad_{n-2},\ad_{n-1}) = (\D,\D)$: 
		$\mathbf i' = (\ldots, n-2,n-3,\dots,2,1,n-1,\dots,n-2,\dots,3, \boxed{2},\boxed{1})$.
		\item $(\ad_{n-2},\ad_{n-1}) = (\A,\D)$:
		$\mathbf i' = (\ldots,1,2,\dots,n-2,\boxed{n-1},\boxed{n-2},\dots,\boxed{3},\boxed{2},\boxed{1})$.
		\item $(\ad_{n-2},\ad_{n-1}) = (\D,\A)$:
		$\mathbf i' = (\ldots,n-1,n-2,\dots,2,\boxed{1},\boxed{2},\dots,\boxed{n-2},\boxed{n-1})$.
		\item $(\ad_{n-2},\ad_{n-1}) = (\A,\A)$:
		$\mathbf i' = (\ldots,2,3,\dots,n-1,1,2,\dots,n-3,\boxed{n-2},\boxed{n-1})$. 
	\end{enumerate}
	One can see that, in any case, the boxed words of $\mathbf i'$ do not change when applying $2$-move on $\mathbf i'$. More precisely, when $\ad_{n-1} = \ad_n$, that is, for cases (1) and (4), the last two words $(2,1)$ or $(n-2,n-1)$ do not change. When $\ad_{n-1} \neq \ad_n$, that is, for cases (2) and (3), the last $n-1$ words $(n-1,n-2,\dots,1)$ or $(1,2,\dots,n-1)$ do not change. One may notice that the last $\kappa(\ad_{n-1},\ad_n)$ words of $\mathbf i'$ do not change. 
	This implies that, for $k \leq \kappa(\ad_{n-1}, \ad_n)$, by setting $\mathbf i'' = (i''_1,\dots,i''_{\barn-n})$ and $\mathbf i' = (i'_1,\dots,i'_{\barn-n})$, we have
	\[
	(i''_1,\dots,i''_{\barn-n-k}) \sim (i'_1,\dots,i'_{\barn -n-k}) \quad \text{ and }\quad
	(i''_{\barn-n-k+1},\dots,i''_{\barn-n}) = (i'_{\barn-n-k+1},\dots, i'_{\barn-n}).
	\]
	Therefore, we get
	\[
	\mathbf i = E_{\ad_n}(k)(\mathbf i'') \sim E_{\ad_n}(k)(\mathbf i') =
	\isI{\ad}{0,\dots,0,k}
	\]
	since $k \leq \kappa(\ad_{n-1}, \ad_n)$. Hence, we prove the proposition.
\end{proof}

\begin{proof}[Proof of Theorem~\ref{thm_main}]
	Let $\mathbf i \in \Rn{n+1}$ having small indices.
By Proposition~\ref{prop_equiv_small_indices}, we may assume that $\mathbf i \sim \mathbf i_{\ad}(0,\dots,0,\kappa(\ad_{n-1},\ad_n))$ for some $\ad \in \{\A,\D\}^n$.
Since the normal fan of the string polytope $\Sigma_{\Delta_{\bf i}(\lambda)}$ is independent of the choice of a regular dominant weight $\lambda$, 
we may assume that 
\[
\lambda = 2 \sum_{i=1}^n \varpi_i.
\] 

Let $\wS_{\bf i}$ be the fan 
defined in Definition~\ref{def_fan} and let $D_j$ be the torus-invariant prime divisor corresponding to the ray generator $\bfv_j$ for each $j \in [\barn]$.
Let $D$ be a Cartier\footnote{Any Weil divisor of a smooth toric variety is Cartier, see \cite[Proposition 4.2.6]{CLS11toric}.} divisor given by
\[
	D = \sum_{j=1}^\barn 2 D_j.
\]
Then the polyhedron associated to the divisor $D$ is precisely
\[
P_D = \Delta_{\bf i }(\lambda)
\]
by the definition of the polytope $P_D$ (see~\eqref{eq_def_of_PD}) and the string polytope (see Definition~\ref{def_string_polytopes}).
Therefore, to prove Theorem~\ref{thm_main}, it is enough to show that $D$ is a basepoint free divisor on $X_{\wS_{\bf i}}$ by Corollary~\ref{cor_BPF_and_small_resolution}. 

To check the basepoint freeness of the divisor $D$, we apply Theorem~\ref{thm_BPF} to primitive collections in $\PC(\wS_{\bf i})$. 
That is, it is enough to show that the support function $\varphi_D$ satisfies
\begin{equation}\label{eq_support_function_ineq}
\varphi_D\left(\sum_{x \in \P} x \right) \geq \sum_{x \in \P} \varphi_D(x)
\end{equation}
for every primitive collection $\P \in \PC(\wS_{\bf i})$. 
The proof of the claim will split into the cases mention in the definition of small indices. 

\smallskip
\noindent\textbf{Case $1$: $k=0$.} 
In this case, we have $|\GP(\mathbf i)| = \barn$ and so each region $R_i$ does not contain any node in its interior. In addition, 
there exists a unique rigorous path $\gamma_j$ having a peak $t_j$ for each $j\in [\barn]$ and every rigorous path travels along at most $3$ wires by Lemma~\ref{lemma_i+1_once}. 
Namely, each path is of the form:
\[
	\begin{cases}
		\ell_i \to \ell_{i+1} & \text{or} \\
		\ell_i \to \ell_p \to \ell_{i+1} & \\
	\end{cases}
\]
for some $i, p \in [n]$.

We first claim that it satisfies either 
\begin{equation}\label{eq_proof_main_GC_case}
\begin{split}
&\bfw_j + \bfv_j = 0, \text{ or }\\
&\text{there exist }j_1 \neq j_2 \text{ such that } \bfw_j + \bfv_j = \bfw_{j_1} + \bfv_{j_2}. 
\end{split}
\end{equation}
Figures~\ref{fig_lemma_5.1_path1} and~\ref{fig_lemma_5.1_path2} present all possible local pictures of the path $\p_j$  around the maximal peak $t_{j}$
in case where  $t_j$ is not on $\ell_{n+1}$ (and therefore some wire $\ell_q$ passes through the interior of the region $R_i$).
	\begin{figure}[t]
	\begin{subfigure}[c]{0.3\textwidth}
		\centering
		\begin{tikzpicture}[scale= 0.8]
		\tikzset{red line/.style = {line width=0.5ex, red, semitransparent}}
		\tikzset{blue line/.style = {line width=0.5ex, blue, nearly transparent}}
		
		\draw[->] (1,0)--(2,0.5)--(2,2)--(1,2.5) node[above] {$\ell_i$};
		\draw[->] (2,2.5)--(1,2)--(1,1)--(0,0.5) node[below] {$\ell_{i+1}$};
		\draw[->,thick,color=blue] (2,0)--(0,1) node[above left] {$\ell_{q}$};
		
		\draw[red line, ->] (1,0)--(2,0.5) node[red,at start, below, nearly opaque] {$\p_j$}
		--(2,2)--(1.5,2.25)--(1,2)--(1,1)--(0,0.5);

		\node[above,  color=NavyBlue] at (1.5,2.25) {$t_j$};
		\node[above, color = NavyBlue] at (0.5, 0.75) {$t_{j_1}$};
		\node[above, color = NavyBlue] at (1.5, 0.25) {$t_{j_2}$};

		\end{tikzpicture}
	\end{subfigure}
	\begin{subfigure}[c]{0.3\textwidth}
		\centering
		\begin{tikzpicture}[scale= 0.8]
		\tikzset{red line/.style = {line width=0.5ex, red,  semitransparent}}
		\tikzset{blue line/.style = {line width=0.5ex, blue, semitransparent}}
		
		\draw[->] (2,0.5)--(1,1)--(1,2)--(0,2.5) node[left] {$\ell_i$};
		\draw[<-] (1,0)--(0,0.5)--(0,2)--(1,2.5) node[right] {$\ell_{i+1}$};
		\draw[->,thick,color=blue] (2,1)--(0,0) node[below] {$\ell_q$};
		
		\node[above, color=NavyBlue] at (0.5,2.25) {$t_j$};
		\node[above, color = NavyBlue] at (1.5,0.75) {$t_{j_1}$};
		\node[above, color = NavyBlue] at (0.5,0.25) {$t_{j_2}$};

		\draw[red line, ->] (2,0.5)--(1,1) node[red,at start, below, nearly opaque] {$\p_j$}
		--(1,2)--(0.5,2.25)--(0,2)--(0,0.5)--(1,0);
		
		\end{tikzpicture}
	\end{subfigure}

	\caption{Rigorous paths $\p_j$ of the form $\ell_i \to \ell_{i+1}$.}
	\label{fig_lemma_5.1_path1}
\end{figure}
\begin{figure}[b]
	\begin{subfigure}{0.18\textwidth}
		\centering
		\begin{tikzpicture}[scale= 0.8]
		\draw[->] (3,1)--(3,1.5)--(2,2)--(2,2.5)--(1,3)--(1,3.5) node[above] {$\ell_i$};
		\draw[->] (2,3.5)--(2,3)--(1,2.5)--(1,1.5)--(0,1) node[left] {$\ell_{p}$};
		\draw[<-] (2,0)--(1,0.5)--(1,1)--(0,1.5) node[above] {$\ell_{i+1}$};
		\draw[->,thick,color=blue] (3,2)--(2,1.5)--(2,0.5)--(1,0) node[below] {$\ell_{q}$};
		
		\node[above,color=NavyBlue] at (1.5,2.75) {$t_j$};
		\node[above, color=NavyBlue] at (2.5,1.75) {$t_{j_1}$};
		\node[above, color=NavyBlue] at (1.5,0.25) {$t_{j_2}$};

		\draw[line width=0.5ex, red, semitransparent, ->] 
		(3,1)--(3,1.5) node[red,at start, below, nearly opaque] {$\p_j$}
		--(2,2)--(2,2.5)--(1.5,2.75)--(1,2.5)--(1,1.5)--(0.5,1.25)--(1,1)--(1,0.5)--(2,0);
		
		\end{tikzpicture}
	\end{subfigure}~\hspace{2em}
	\begin{subfigure}{0.18\textwidth}
		\centering
		\begin{tikzpicture}[scale= 0.8]
		\draw[->] (3,1)--(3,1.5)--(2,2)--(2,2.5)--(1,3)--(1,3.5) node[above] {$\ell_p$};
		\draw[->] (2,3.5)--(2,3)--(1,2.5)--(1,1.5)--(0,1) node[left] {$\ell_{i+1}$};
		\draw[->,thick,color=blue] (2,0)--(1,0.5) node[below, at start] {$\ell_{q}$} --(1,1)--(0,1.5) ;
		\draw[<-] (3,2)--(2,1.5)--(2,0.5)--(1,0) node[below] {$\ell_{i}$};
		
		\node[above,color=NavyBlue] at (1.5,2.75) {$t_j$};
		\node[above, color=NavyBlue] at (0.5,1.25) {$t_{j_1}$};
		\node[above, color=NavyBlue] at (1.5,0.25) {$t_{j_2}$};
		
		\draw[line width=0.5ex, red,  semitransparent, ->] 
		(1,0)--(2,0.5) node[red,at start, above left, nearly opaque] {$\p_j$}
		--(2,1.5)--(2.5,1.75)--(2,2)--(2,2.5)--(1.5,2.75)--(1,2.5)--(1,1.5)--(0,1);
		
		\end{tikzpicture}
	\end{subfigure}~\hspace{2em}
	\begin{subfigure}{0.18\textwidth}
		\centering
		\begin{tikzpicture}[scale= 0.8]
		
		\draw[->,thick,color=blue] (2,1)--(0,0) node[left] {$\ell_q$};
		\draw[->] (2,0.5)--(1,1)--(1,1.5)--(2,2) node[right] {$\ell_i$};
		\draw [->] (2,1.5)--(1,2)--(1,2.5)--(0,3) node[left] {$\ell_p$};
		\draw [->] (1,3)--(0,2.5)--(0,0.5)--(1,0) node[right] {$\ell_{i+1}$};
		
		\node[above, color=NavyBlue] at (0.5,2.75) {$t_j$};
		\node[above, color=NavyBlue] at (1.5, 0.75) {$t_{j_1}$};
		\node[above, color =NavyBlue] at (0.5,0.25) {$t_{j_2}$};
		
		\draw[line width=0.5ex, red,  semitransparent, ->] 
		(2,0.5)--(1,1) node[red,at start, right, nearly opaque] {$\p_j$}
		--(1,1.5)--(1.5,1.75)--(1,2)--(1,2.5)--(0.5,2.75)--(0,2.5)--(0,0.5)--(1,0);
		
		\end{tikzpicture}
	\end{subfigure}~\hspace{2em}
	\begin{subfigure}{0.18\textwidth}
		\centering
		\begin{tikzpicture}[scale= 0.8]
		
		\draw[->] (1,0)--(2,0.5)--(2,2.5)--(1,3) node[left] {$\ell_i$};
		\draw[->] (2,3)--(1,2.5)--(1,2)--(0,1.5) node[left] {$\ell_p$};
		\draw[->] (0,2)--(1,1.5)--(1,1)--(0,0.5) node[left] {$\ell_{i+1}$};
		\draw[<-,thick,color=blue] (0,1)--(2,0) node[right] {$\ell_q$};
		
		\node[above, color=NavyBlue] at (1.5,2.75) {$t_j$};
		\node[above, color=NavyBlue] at (0.5,0.75) {$t_{j_1}$};
		\node[above, color=NavyBlue] at (1.5,0.25) {$t_{j_2}$};
		
		\draw[line width=0.5ex, red,  semitransparent, ->] 
		(1,0)--(2,0.5) node[red,at start, left, nearly opaque] {$\p_j$}
		--(2,2.5)--(1.5,2.75)--(1,2.5)--(1,2)--(0.5,1.75)--(1,1.5)--(1,1)--(0,0.5);
		
		\end{tikzpicture}
	\end{subfigure}~
%
%
%
%
%
	\caption{Rigorous paths $\p_j$ of the form $\ell_i \to \ell_p \to \ell_{i+1}$.}
	\label{fig_lemma_5.1_path2}
\end{figure}
Note that $\bfw_j + \bfv_j = 0$ only when $t_j$ is lying on $\ell_{n+1}$. Moreover, 
the orientation of the wire $\ell_q$ in each case is uniquely determined as the forbidden pattern~ in Figure \ref{figure_avoiding} should be avoided. 
One can easily see that the path $\ell_{i} \to \ell_q \to \ell_{i+1}$ is a rigorous path corresponding to the vector $\bfw_{j_1}$. Let $t_{j_2}$ be the node described in the figures. Then the nodes $t_j$ and $t_{j_2}$ are on the same column while
$t_{j_1}$ is located on the next column of $t_{j}$, which implies that $j_1 \neq j_2$. 
Thus the claim follows from the definitions of $\bfw_j$'s and $\bfv_j$'s. 
Note that it is straightforward from Figures~\ref{fig_lemma_5.1_path1} and~\ref{fig_lemma_5.1_path2} that the region assigned by $\bfw_j + \bfv_j$ 
is equal to that of $\bfw_{j_1} + \bfv_{j_2}$ since the regions associated to $\bfw_j$ and $\bfw_{j_1}$ 
(respectively $\bfv_j$ and $\bfv_{j_2}$) are differ by $\cham{j}$ (respectively $-\cham{j}$). 

Note that $\wS_{\bf i} = \Sigma_{\bf i}$ in this case.
It remains to show that $D$ is basepoint free.
Since $j_1 \neq j_2$, the cone $\Cone(\bfw_{j_1}, \bfv_{j_2})$ is contained in $\wS_{\bf i}$. Moreover, 
the sum $\bfw_{j_1} + \bfv_{j_2}$ is contained in that cone. Using the fact that the support function $\varphi_D$ is linear on each cone and 
\[
\varphi_{D}(\bfv_j) = -2, \quad \varphi_D(\bfw_j) = 0 \quad \text{ for all }j,
\] we get 
	\begin{equation}\label{eq_pD_inequality_DD_1}
\begin{split}
&\varphi_D(\bfws{j} + \bfv_j) = 0 \geq 
\varphi_D(\bfws{j}) + \varphi_D(\bfv_j) = 0 + (-2), \text{ or }\\
&\varphi_D(\bfws{j} + \bfv_j) = \varphi_D(\bfws{j_1} + \bfv_{j_2}) = \varphi_D(\bfws{j_1}) + \varphi_D(\bfv_{j_2}) = 0 + (-2) \geq 
\varphi_D(\bfws{j}) + \varphi_D(\bfv_j) = 0 + (-2).
\end{split}
\end{equation}
Therefore the inequalities~\eqref{eq_support_function_ineq} hold for primitive collections $\{\{\bfws{j}, \bfv_j\} \mid 1 \leq j \leq \barn\}$, and the result follows.
\vspace{0.3cm}

\smallskip

\noindent\textbf{Case $2$: $\ad_{n-1} = \D$ and $k = 1$.}
In this case, the wiring diagram  is described 
in Figure~\ref{fig_p0_k1_DD} and there is a unique non-canonical path $\widetilde{\p}_0$ (blue path).
\begin{figure}[hbt]
	\centering
	\begin{tikzpicture}[xscale=0.9,yscale=0.6]
	\tikzset{every node/.style = {font = \footnotesize}}
	\tikzset{red line/.style = {line width=0.5ex, red, semitransparent}}
	\tikzset{blue line/.style = {line width=0.5ex, blue, semitransparent}}

	\draw (0,0)--(0,1.5)--(5,4)--(5,6.5);
	\draw (1,0)--(1,1)--(2,1.5)--(2,2)--(1,2.5)--(1,4)--(4,5.5)--(4,6.5);
	\draw (2,0)--(2,1)--(0,2)--(0,6.5);
	\draw (3,0)--(3,2.5)--(2,3)--(2,4)--(1,4.5)--(1,6.5);
	\draw[dashed, black!40] (4,0)--(4,3)--(3,3.5)--(3,4.5)--(2,5)--(2,6.5);
	\draw (5,0)--(5,3.5)--(4,4)--(4,5)--(3,5.5)--(3,6.5);

	\draw (0,7.25)--(0,7.5)--(1,8)--(1,8.25);
	\draw (1,7.25)--(1,7.5)--(0,8)--(0,8.25);
	
	\draw[dotted, draw=gray] (1.5,1.25)--(6,1.25) node[at end, right] {$\barn$};
	\draw[dotted, draw=gray] (0.5,1.75)--(6,1.75) node[at end, right] {$\barn-1$};
	\draw[dotted, draw=gray] (1.5,2.25)--(6,2.25) node[at end, right] {$\barn-2$};	
	\draw[dotted, draw=gray] (2.5,2.75)--(6,2.75) node[at end, right] {$\barn-3$};
	\draw[dotted, draw=gray] (4.5,3.75)--(6,3.75) node[at end, right] {$\barn-n$};	
	\draw[dotted, draw=gray] (1.5,4.25)--(6,4.25) node[at end, right] {$\barn-(n+1)$};	
	\draw[dotted, draw=gray] (3.5,5.25)--(6,5.25) node[at end, right] {$\barn-2n+2$};
	\draw[dotted, draw=gray] (0.5,7.75)--(6,7.75) node [at end, right] {$x$};
	
	\node at (6.7,8.7) {labels of nodes};
	\node at (0,7) {$\vdots$};
	\node at (1,7) {$\vdots$};
	\node at (3,7) {$\vdots$};
	\node at (4,7) {$\vdots$};
	\node at (5,7) {$\vdots$};
	
	\node at (0,8.75) {$\vdots$};
	\node at (1,8.75) {$\vdots$};
	\node at (3,8.75) {$\vdots$};
	\node at (4,8.75) {$\vdots$};
	\node at (5,8.75) {$\vdots$};

	\node[below] at (0,0) {$\ell_{n+1}$};
	\node[below] at (1,0) {$\ell_n$};
	\node[below] at (2,0) {$\ell_{n-1}$};
	\node[below] at (3,0) {$\ell_{n-2}$};
	\node[below] at (4,0) {$\cdots$};
	\node[below] at (5,0) {$\ell_1$};

	\draw[red, line width = 0.5ex, dotted, ->]
	(3.05,0)--(3.05,2.5)--(2.05,3)--(2.05,4)--(1.5,4.3)--(0.95,4)--(0.95,2.5)--(1.49,2.25)--(0.45,1.75)--(1.95,1)--(1.95,0);
	
	\draw[blue, very thick, ->] (3,0)--(3,2.5)--(2,3)--(2,4)--(1.5,4.25)--(1,4)--(1,2.5)--(2,2)--(2,1.5)--(1.5,1.25)--(2,1)--(2,0);

	\node[above] at (0.5,7.75) {$t_x$};
	\node[above] at (1.5,1.25) {$t_{\barn}$};
	\node[above] at (0.5,1.75) {$t_{\barn-1}$};
	\node[above] at (1.5,2.25) {$t_{\barn-2}$};
	\node[above] at (2.5,2.75) {$t_{\barn-3}$};
	\node[above] at (4.5,3.75) {$t_{\barn-n}$};
	\node[above] at (1.5,4.25) {$t_{\barn-(n+1)}$};
	\node[above] at (3.5,5.25) {$t_{\barn-2n+2}$};

	\end{tikzpicture}
	\caption{ Rigorous paths $\p_{\barn-(n+1)} = (\ell_{n-2} \to \ell_n \to \ell_{n+1} \to \ell_{n-1})$ (dotted red) and $\widetilde{\p}_0 = (\ell_{n-2} \to \ell_{n} \to \ell_{n-1})$ (blue) for $\ad_{n-1} = \D$ and $k = 1$.}
	\label{fig_p0_k1_DD}
\end{figure}

Recall from~\eqref{eq_def_tau_P} that $\tau$ is the cone generated by the set $\{ \bfws{\barn-(n+1)}, \bfv_{\barn-2} , \bfws{\barn}\}$. 
To check the basepoint freeness of $D$, we consider the set of primitive collections of $\wS_{\bf i} = \Sigma_{\bf i}^*(\tau)$ 
\[
\begin{split}
\PC(\wS_{\bf i}) &= \{ \{ \bfws{j}, \bfv_j \} \mid 1 \leq j \leq \barn \}  \\
&\quad \cup \{\bfws{\barn-(n+1)}, \bfv_{\barn-2} , \bfws{\barn} \}
\cup \{ \bfw_{\widetilde{\p}_0}, \bfv_{\barn-(n+1)} \} \cup \{ \bfw_{\widetilde{\p}_0}, \bfws{\barn-2} \} \cup \{\bfw_{\widetilde{\p}_0}, \bfv_{\barn}\}
\end{split}
\]
obtained by Proposition~\ref{prop_PC_star_subdivision},
where $\bfw_{\widetilde{\p}_0}$ is the vector generating the ray in the fan $\wS_{\bf i}$ which corresponds to the path $\widetilde{\p}_0$. Note that $\wS_{\mathbf i} = \Sigma_{\mathbf i}^{\ast}(\tau)$ in this case.

Let $t_x$ be the second node from the bottom among nodes lying on the first column  (see Figure~\ref{fig_p0_k1_DD}). 
Similarly to {\bf Case $1$}, we see that \eqref{eq_proof_main_GC_case} holds since
\begin{itemize}
	\item for each node $t_j \neq t_x, t_{\barn - (n+1)}$, the local shape around the node coincides with one of the pictures in 
	Figures~\ref{fig_lemma_5.1_path1} and~\ref{fig_lemma_5.1_path2} and so $(j_1, j_2)$ is uniquely determined, 
	\item for $t_x$, we have 
	\[
		{\bfw}_x + {\bfv}_x = {\bfw}_{\barn - (n+1)} + {\bfv}_{\barn - 1}
	\]
	and therefore we take $(j_1, j_2) = (\barn - (n+1), \barn - 1)$, 
	\item for $t_{\barn - (n+1)}$, we have  
	\[
		\bfw_{\barn-(n+1)} + \bfv_{\barn-(n+1)} = \bfw_{\barn-3} + \bfv_{\barn}.
	\]
	and so we take $(j_1, j_2) = (\barn - 3, \barn)$.
\end{itemize}
(This procedure is necessary since there are two rigorous paths having the maximal peak $t_{\barn - (n+1)}$.)
Since every pair $\{\bfw_{j_1}, \bfv_{j_2} \}$ does not contain any primitive collection in $\mathrm{PC}(\wS_{\bf i})$, 
Lemma~\ref{lem_smooth_fan} implies that 
\[
\mathrm{Cone} (\bfw_{j_1}, \bfv_{j_2}) \in \wS_{\bf i}.
\] 

For the remaining four primitive collections in $\mathrm{PC}(\wS_{\bf i})$, we get the following relations:
\begin{align*}
&\bfws{\barn-(n+1)} + \bfv_{\barn-2} + \bfws{\barn} = \bfw_{\widetilde{\p}_0}, &
&\bfw_{\widetilde{\p}_0}  + \bfv_{\barn-(n+1)} = \bfws{\barn-3} + \bfv_{\barn-2},\\
&\bfw_{\widetilde{\p}_0} + \bfws{\barn-2} = \bfws{\barn-(n+1)} + \bfws{\barn-1}, &
&\bfw_{\widetilde{\p}_0} + \bfv_{\barn} = \bfws{\barn-(n+1)} + \bfv_{\barn-2}.
\end{align*}
Since the map $\varphi_D$ is linear on each cone in ${\wS_{\bf i}}$, we have that
\begin{equation}\label{eq_pD_inequality_DD_2}
\begin{split}
&\varphi_D(\bfw_j + \bfv_j) = \varphi_D(\bfw_{j_1} + \bfv_{j_2}) = \varphi_D(\bfw_{j_1}) + \varphi_D(\bfv_{j_2}) = 0 + (-2) = -2  \\
&\qquad \qquad \qquad \geq \varphi_D(\bfw_j) + \varphi_D(\bfv_j) 
= 0 + (-2), \\
&\varphi_D(\bfws{\barn-(n+1)} + \bfv_{\barn-2} + \bfws{\barn}) = \varphi_D(\bfw_{\widetilde{\p}_0}) = 0 \\
&\qquad \qquad \qquad \geq \varphi_D(\bfws{\barn-(n+1)}) + \varphi_D(\bfv_{\barn-2}) + \varphi_D(\bfws{\barn})
= 0 + (-2) + 0, \\
& \varphi_D(\bfw_{\widetilde{\p}_0}  + \bfv_{\barn-(n+1)} ) = \varphi_D(\bfws{\barn-3} + \bfv_{\barn-2})
= \varphi_D(\bfws{\barn-3}) + \varphi_D( \bfv_{\barn-2}) = 0 + (-2) \\
& \qquad \qquad \qquad \geq \varphi_D(\bfw_{\widetilde{\p}_0}) + \varphi_D(\bfv_{\barn-(n+1)}) = 0 + (-2), \\
&\varphi_D(\bfw_{\widetilde{\p}_0} + \bfws{\barn-2} ) = \varphi_D(\bfws{\barn-(n+1)} + \bfws{\barn-1})
= \varphi_D(\bfws{\barn-(n+1)}) + \varphi_D(\bfws{\barn-1}) = 0 + 0 \\
&\qquad \qquad \qquad \geq \varphi_D(\bfw_{\widetilde{\p}_0}) + \varphi_D(\bfws{\barn-2} ) = 0 + 0, \\
& \varphi_D(\bfw_{\widetilde{\p}_0} + \bfv_{\barn}) = \varphi_D(\bfws{\barn-(n+1)} + \bfv_{\barn-2})
= \varphi_D(\bfws{\barn-(n+1)} ) + \varphi_D( \bfv_{\barn-2}) = 0 + (-2) \\
& \qquad \qquad \qquad \geq \varphi_D(\bfw_{\widetilde{\p}_0}) + \varphi_D(\bfv_{\barn} ) = 0 + (-2).
\end{split}
\end{equation}
Therefore the inequality~\eqref{eq_support_function_ineq} holds for every primitive collection and this proves the theorem for {\bf Case $2$}.
\vspace{0.3cm}

\smallskip
\noindent\textbf{Case $3$: $\ad_{n-1} = \D$ and $k = 2$.} In this case, we have $n \geq 3$ as $0 \leq k \leq n-1$. 
We divide into two cases: $n = 3$ 
($|\mcal{GP}({\bf i})| = \barn + 1$) and $n >3$ ($|\mcal{GP}({\bf i})| = \barn + 2$). See \eqref{eq_set_GP_small_indices}.

When $(k,n) = (2,3)$, we have $\mathbf i = (1,3,2,1,3,2)$ (see Figure~\ref{fig_132132}) and there is exactly one non-canonical rigorous path 
$\widetilde{\p}_2$ with $\bfw_{\widetilde{\p}_2} = \bfw_3 + \bfv_4 + \bfw_5$ as in \eqref{eq_set_GP_small_indices}.
Then the cone $\tau$ defined in \eqref{eq_def_tau_P} is generated by $\{\bfw_3, \bfv_4, \bfw_5\}$
and the set of primitive collections for the fan $\wS_{\bf i} = \Sigma_{\bf i}^{\ast}(\tau)$ is given by
\[
\begin{split}
\PC(\wS_{\bf i}) &= \{ \{ \bfws{j}, \bfv_j \} \mid 1 \leq j \leq 6 \}  \\
&\quad \cup \{ \bfw_3, \bfv_4, \bfw_5 \} \cup 
\{ \bfw_{\widetilde{\p}_2}, \bfv_3\} \cup \{\bfw_{\widetilde{\p}_2}, \bfw_4\} \cup \{\bfw_{\widetilde{\p}_2}, \bfv_5\}
\end{split}
\]
by Proposition~\ref{prop_PC_star_subdivision}.

For the first $6$ primitive collections, we can directly read off the following relations from Figure \ref{fig_132132}:
\begin{equation}\label{eq_rel_case3_1}
\begin{split}
&\bfw_1 + \bfv_1 = \bfw_{\widetilde{\p}_2} + \bfv_4, \quad
\bfw_2 + \bfv_2 = \bfw_3 + \bfv_5, \quad
\bfw_3 + \bfv_3 = \bfw_4 + \bfv_6, \\
& \bfw_j + \bfv_j = 0 \quad \text{ for } j = 4,5,6.
\end{split}
\end{equation}
(For instance for the first equality, both $\bfw_1 + \bfv_1$ and $\bfw_{\widetilde{\p}_2} + \bfv_4$ correspond to the formal sum $ \cham{3} - \cham{4} + \cham{5}$ in Figure \ref{fig_132132}.) Observe that none of the pairs of summands on the right hand side of each relation in \eqref{eq_rel_case3_1}
contain any of the primitive collections, which implies that each pair generates a cone in $\wS_{\bf i}$ by Lemma \ref{lem_smooth_fan}.

For the other primitive collections, we similarly obtain the followings:
\begin{equation}\label{eq_rel_case3_2}
\bfw_3 + \bfv_4 + \bfw_5 = \bfw_{\widetilde{\p}_2},  \quad 
\bfw_{\widetilde{\p}_2} + \bfv_3 = \bfw_5 + \bfv_6, \quad 
\bfw_{\widetilde{\p}_2} + \bfw_4 = \bfw_3 + \bfw_5, \quad 
\bfw_{\widetilde{\p}_2} + \bfv_5 = \bfw_3 + \bfv_4.
\end{equation}
Since the support function $\varphi_D$ is linear on each cone in $\wS_{\bf i}$ and
\[
\varphi_D(\bfv_j) = -2, \quad \varphi_D(\bfws{j}) = \varphi_D(\bfw_{\widetilde{\p}_2})  = 0,
\]
the relations in \eqref{eq_rel_case3_1} and \eqref{eq_rel_case3_2}
imply that the support function $\varphi_D$ satisfies the desired inequalities~\eqref{eq_support_function_ineq}. 
(This conclusion is straightforward since for each relation, the number of $\bfv_j$'s on the left is greater then equal to that on the right, 
cf. \eqref{eq_pD_inequality_DD_2}.)

Now let us consider the case of $k=2$ with $n>3$. 
The main difference from the case $(k,n) = (2,3)$ is that one more non-canonical path $\widetilde{\p}_0$ 
(the blue path in Figure \ref{fig_p0_k2_DD}) appears. 
Therefore, as in \eqref{eq_set_GP_small_indices}, there are two non-canonical paths $\widetilde{\p}_0$
and $\widetilde{\p}_2$ with 
\[
	\bfw_{\widetilde{\p}_0} = \bfw_{\barn-(n+2)} + \bfv_{\barn-4} + \bfw_{\barn-1}, \quad \bfw_{\widetilde{\p}_2} = \bfws{\barn-3} + \bfw_{\barn-1} + \bfv_{\barn-2}.
\]
See Proposition \ref{prop_w_p_i_sum_of_w_and_v}.
For the fan $\wS_{\bf i} = (\Sigma_{\bf i}^{\ast}(\tau))^{\ast}(\tau_{{2}})$ defined in \eqref{eq_def_tau_Q}),
where $\tau = \mathrm{Cone}(\bfw_{\barn-(n+2)}, \bfv_{\barn-4}, \bfw_{\barn-1})$ and 
$\tau_{2} = \mathrm{Cone}(\bfws{\barn-3}, \bfw_{\barn-1}, \bfv_{\barn-2})$
in \eqref{eq_def_tau_P} and \eqref{eq_def_tau_Q},  
the set of primitive collections of $\wS_{\bf i}$ is 
\begin{equation}\label{eq_PC_for_DD_k2}
\begin{split}
\PC(\wS_{\bf i}) &= \{ \{ \bfws{j}, \bfv_j \} \mid 1 \leq j \leq \barn \}  \\
&\quad \cup \{\bfws{\barn-(n+2)}, \bfv_{\barn-4} , \bfws{\barn-1} \}
\cup \{ \bfw_{\widetilde{\p}_0}, \bfv_{\barn-(n+2)} \} \cup \{ \bfw_{\widetilde{\p}_0}, \bfws{\barn-4} \} \cup \{\bfw_{\widetilde{\p}_0}, \bfv_{\barn-1}\} \\
&\quad \quad \cup \{\bfws{\barn-3}, \bfw_{\barn-1}, \bfv_{\barn-2} \} \cup 
\{\bfw_{\widetilde{\p}_2}, \bfv_{\barn-3} \} \cup \{\bfw_{\widetilde{\p}_2}, \bfv_{\barn-1} \} \cup \{\bfw_{\widetilde{\p}_2}, \bfw_{\barn-2}\} \cup 
\{ \bfw_{\widetilde{\p}_2}, \bfw_{\barn-(n+2)}, \bfv_{\barn - 4}\}
\end{split}
\end{equation}
which can be obtained by Proposition \ref{prop_PC_star_subdivision}.

Let $t_{x}$ and $t_{y}$ be the nodes at the second from the bottom among nodes lying on the first column and at the third from the bottom
among nodes lying on the second column, respectively. See Figure~\ref{fig_p0_k2_DD} for nodes $t_x$ and $t_y$. For simplicity, we only draw the case $x > y$, but in general $x$ does not need to be grater than $y$.  Similarly to~\eqref{eq_proof_main_GC_case}, we have
\[
\begin{split}
	&\bfw_j + \bfv_j = 0; \text{ or }\\
	&\text{there exist }j_1 \neq j_2 \text{ such that } 
		\bfw_j + \bfv_j = 
		\begin{cases} \bfw_{j_1} + \bfv_{j_2} &  \text{if} ~j \neq x \\	
			\bfw_{\widetilde{\p}_0} + \bfv_{j_2} & \text{if} ~ j = x
		\end{cases} 
\end{split}
\]
because of the following observation: 
\begin{itemize}
	\item for each node $t_j \neq t_{x}, t_{y}, t_{\barn - (n+1)}$, the local shape around the node coincides with one of the pictures in 
	Figures~\ref{fig_lemma_5.1_path1} and~\ref{fig_lemma_5.1_path2} and so $(j_1, j_2)$ is uniquely determined in each case. 
	\item for $t_{x}$, we have 
	\begin{equation}\label{eq_case3_1}
		{\bfw}_{x} + {\bfv}_{x} = {\bfw}_{\widetilde{\p}_0} + {\bfv}_{\barn - 2},
	\end{equation}
	\item for $t_{y}$, we have 
	\begin{equation}\label{eq_case3_2}
		{\bfw}_{y} + {\bfv}_{y} = {\bfw}_{\barn - (n+2)} + {\bfv}_{\barn - 3}
	\end{equation}
	and therefore we take $(j_1, j_2) = (\barn - (n+2), \barn - 3)$, 
	\item for $t_{\barn - (n+2)}$, we have  
	\begin{equation}\label{eq_case3_3}
		\bfw_{\barn-(n+2)} + \bfv_{\barn-(n+2)} = \bfw_{\barn-5} + \bfv_{\barn - 1}
	\end{equation}
	and so we take $(j_1, j_2) = (\barn - 5, \barn - 1)$.
\end{itemize}
In particular, we have $j_1 \neq j_2$.

For showing that $D$ is basepoint free, we apply the same procedure as in the previous cases as follows. 
For the first $\barn$ primitive collections, we can similarly show that each pair $\{\bfw_{j_1}, \bfv_{j_2}\}$ as well as 
$\{\bfw_{\widetilde{\p}_0}, \bfv_{j_2}\}$ does not contain any primitive collection of $\wS_{\bf i}$ listed in \eqref{eq_PC_for_DD_k2} and so it 
generates a cone in $\wS_{\bf i}$ using Lemma \ref{lem_smooth_fan}. More precisely, it is rather straightforward that $\{ \bfw_{j_1}, \bfv_{j_2}\}$ does not 
contain any of \eqref{eq_PC_for_DD_k2} since $j_1 \neq j_2$. For $\{ \bfw_{\widetilde{\p}_0}, \bfv_{j_2}\}$, we need to show that 
\[
	j_2 \neq  \barn - (n+2), ~\barn- 1
\] 
which follows from our observations above that $j_2 = \barn - 2$. 

For the remaining nine primitive collections in \eqref{eq_PC_for_DD_k2}, we obtain the following relations:
\begin{equation}\label{eq_relation_DD}
\begin{array}{ll}
\bfw_{\barn-(n+2)} + \bfv_{\barn-4} + \bfw_{\barn-1} = \bfw_{\widetilde{\p}_0}, 
& \bfw_{\widetilde{\p}_0} + \bfv_{\barn-(n+2)} = \bfw_{\barn-5} + \bfv_{\barn-4}, \\
\bfw_{\widetilde{\p}_0} + \bfw_{\barn-4} = \bfw_{\barn-(n+2)} + \bfw_{\barn-3}, 
& \bfw_{\widetilde{\p}_0} + \bfv_{\barn-1} = \bfw_{\barn-(n+2)} + \bfv_{\barn-4}, \\
\bfws{\barn-3} + \bfw_{\barn-1} + \bfv_{\barn-2} = \bfw_{\widetilde{\p}_2}, 
& \bfw_{\widetilde{\p}_2} + \bfv_{\barn-3} = \bfw_{\barn-1} + \bfv_{\barn}, \\
\bfw_{\widetilde{\p}_2} + \bfv_{\barn-1}= \bfw_{\barn-3} + \bfv_{\barn-2}, 
& \bfw_{\widetilde{\p}_2} + \bfw_{\barn-2} = \bfw_{\barn-3} + \bfw_{\barn-1}, \\
\bfw_{\widetilde{\p}_2} + \bfw_{\barn-(n+2)} + \bfv_{\barn-4} = \bfw_{\widetilde{\p}_0} + \bfw_{\barn-3} + \bfv_{\barn-2}.
\end{array}
\end{equation}
See Figure~\ref{fig_p0_k2_DD}. 
We can check that the right hand side of each relation in \eqref{eq_relation_DD} generates a cone in $\wS_{\bf i}$ in a similar fashion. 
Combining all relations \eqref{eq_case3_1}, \eqref{eq_case3_2}, \eqref{eq_case3_3}, and \eqref{eq_relation_DD}, 
\begin{figure}
	\centering
	\begin{tikzpicture}[xscale=0.9,yscale=0.6]
	\tikzset{every node/.style = {font = \footnotesize}}
	\tikzset{red line/.style = {line width=0.5ex, red, semitransparent}}
	\tikzset{blue line/.style = {line width=0.5ex, blue, semitransparent}}
	
	\draw (0,0)--(0,2.5)--(6,5.5)--(6,7.5);
	\draw (1,0)--(1,1)--(3,2)--(3,3.5)--(2,4)--(2,5.5)--(5,7)--(5,7.5);
	\draw (2,0)--(2,1)--(1,1.5)--(1,2.5)--(0,3)--(0,7.5);
	\draw (3,0)--(3,1.5)--(2,2)--(2,3)--(1,3.5)--(1,7.5);
	\draw (4,0)--(4,4)--(3,4.5)--(3,5.5)--(2,6)--(2,7.5);
	\draw[dashed, black!40] (5,0)--(5,4.5)--(4,5)--(4,6)--(3,6.5)--(3,7.5);
	\draw (6,0)--(6,5)--(5,5.5)--(5,6.5)--(4,7)--(4,7.5);
	
	\draw (0,8.25)--(0,8.5)--(1,9)--(1,9.25);	
	\draw (1,8.25)--(1,8.5)--(0,9)--(0,9.25);
	
	\draw (1,10)--(1,10.25)--(2,10.75)--(2,11);
	\draw (2,10)--(2,10.25)--(1,10.75)--(1,11);

	\draw[dotted, draw=gray] (1.5,1.25)--(7,1.25) node[at end, right] {$\barn$};
	\draw[dotted, draw=gray] (2.5,1.75)--(7,1.75) node[at end, right] {$\barn-1$};
	\draw[dotted, draw=gray] (0.5,2.75)--(7,2.75) node[at end, right] {$\barn-2$};	
	\draw[dotted, draw=gray] (1.5,3.25)--(7,3.25) node[at end, right] {$\barn-3$};
	\draw[dotted, draw=gray] (2.5,3.75)--(7,3.75) node[at end, right] {$\barn-4$};
	\draw[dotted, draw=gray] (3.5,4.25)--(7,4.25) node[at end, right] {$\barn-5$};	
	\draw[dotted, draw=gray] (5.5,5.25)--(7,5.25) node[at end, right] {$\barn-(n+1)$};
	\draw[dotted, draw=gray] (2.5,5.75)--(7,5.75) node[at end, right] {$\barn-(n+2)$};	
	\draw[dotted, draw=gray] (3.5,6.75)--(7,6.75) node[at end, right] {$\barn-2n+2$};	
	\draw[dotted, draw=gray] (0.5,8.75)--(7,8.75) node[at end, right] {$x$};
	\draw[dotted, draw=gray] (1.5,10.5)--(7,10.5) node[at end, right] {$y$};
	
	%
	%
	\node at (7.7,11.7) {labels of nodes};
	
	\foreach \x in {0,1,2,4,5,6}
		\node at (\x,8) {$\vdots$};

	\foreach \x in {0,1,2,4,5,6}
		\node at (\x,9.75) {$\vdots$};

	\foreach \x in {0,1,2,4,5,6}	\node at (\x,11.5) {$\vdots$};

	\node[below] at (0,0) {$\ell_{n+1}$};
	\node[below] at (1,0) {$\ell_n$};
	\node[below] at (2,0) {$\ell_{n-1}$};
	\node[below] at (3,0) {$\ell_{n-2}$};
	\node[below] at (4,0) {$\ell_{n-3}$};
	\node[below] at (5,0) {$\cdots$};
	\node[below] at (6,0) {$\ell_1$};

	\draw[blue, very thick, ->] (4,0)--(4,4)--(3,4.5)--(3,5.5)--(2.5,5.75)--(2,5.5)--(2,4)--(3,3.5)--(3,2)--(2.5,1.75)--(3,1.5)--(3,0);
	
	\draw[red, line width = 0.5ex, dotted, ->]
	(2.95,0)--(2.95,1.5)--(1.95,2)--(1.95,3)--(1.5,3.2)--(0.6,2.75)--(1.05,2.5)--(1.05,1.5)--(2.05,1)--(2.05,0);
	%
	%
	%
	%
	%
	
	\node[above] at (0.5, 8.75) {$t_x$};
	\node[above] at (1.5,10.5) {$t_y$};
	\node[above] at (1.5,1.25) {$t_{\barn}$};
	\node[above] at (2.5,1.75) {$t_{\barn-1}$};
	\node[above] at (0.5,2.75) {$t_{\barn-2}$};
	\node[above] at (1.5,3.25) {$t_{\barn-3}$};
	\node[above] at (2.5,3.75) {$t_{\barn-4}$};
	\node[above] at (3.5, 4.25) {$t_{\barn-5}$}; 
	\node[above] at (5.5, 5.25) {$t_{\barn-(n+1)}$};
	\node[above] at (2.5,5.75) {$t_{\barn-(n+2)}$};
	\node[above] at (4.5,6.76) {$t_{\barn-2n+2}$};


	\end{tikzpicture}
	\caption{When $\ad_{n-1} = \D$ and $k = 2$, rigorous paths $\widetilde{\p}_2 = (\ell_{n-2} \to  \ell_{n+1} \to \ell_{n-1})$ (dotted red) and $\widetilde{\p}_0 = (\ell_{n-3} \to \ell_{n} \to \ell_{n-2})$ (blue).}
	\label{fig_p0_k2_DD}
\end{figure}
and the linearity of the support function $\varphi_D$ each cone together with the informations 
\[
\varphi_D(\bfv_j) = -2, \quad \varphi_D(\bfws{j}) = \varphi_D(\bfw_{\widetilde{\p}_0}) = \varphi_D(\bfw_{\widetilde{\p}_2}) = 0,
\]
we see that $\varphi_D$ satisfies the desired inequalities~\eqref{eq_support_function_ineq}.
(As mentioned at the end of the first part of {\bf Case $3$}, the conclusion immediately follows from that 
the number of $\bfv_j$'s on the left is greater then equal to that on the right for each relation.)
This completes the proof for {\bf Case $3$}.
\smallskip

\noindent\textbf{Case~$4$: $\ad_{n-1} = \A$ and $0 < k < n-1$.}
In this case, there is one non-canonical path $\widetilde{\p}_0$ with $\bfw_{\widetilde{\p}_0} = \bfws{\barn-(n+k)}+ \bfv_{\barn-n}+ \bfws{\barn-k+1}$
and the fan $\wS_{\bf i}$ is given by $\Sigma_{\bf i}^{\ast}(\tau)$, where $\tau = \Cone(\bfws{\barn-(n+k)}, \bfv_{\barn-n}, \bfws{\barn-k+1})$. 
(The picture for this case is described in Figure~\ref{fig_p0_AD}.)
By Proposition~\ref{prop_PC_star_subdivision}, we have that
\begin{equation}\label{eq_case4}
\begin{split}
\PC(\wS_{\bf i}) &= \{ \{ \bfws{j}, \bfv_j \} \mid 1 \leq j \leq \barn \}  \\
&\quad \cup \{\bfws{\barn-(n+k)}, \bfv_{\barn-n}, \bfws{\barn-k+1} \}
\cup \{ \bfw_{\widetilde{\p}_0}, \bfv_{\barn-(n+k)} \} \cup \{ \bfw_{\widetilde{\p}_0}, \bfws{\barn-n} \} \cup \{\bfw_{\widetilde{\p}_0}, \bfv_{\barn-k+1}\}.
\end{split}
\end{equation}

Let $t_x$ be the node at the second from the bottom among nodes lying on the $(n-k)$th column (in the painted region in Figure~\ref{fig_p0_AD}).
We can similarly prove that \eqref{eq_proof_main_GC_case} holds because
\begin{itemize}
	\item for each node $t_j \neq t_x$, the local shape around the node coincides with one of the pictures in 
	Figures~\ref{fig_lemma_5.1_path1} and~\ref{fig_lemma_5.1_path2}, so $(j_1, j_2)$ is uniquely determined, 
	\item for $t_x$, we have 
	\[
		\bfw_x + \bfv_x = \bfw_{\barn-(n+k)} + \bfv_{\barn-n+1}.
	\]
	so that we take $(j_1, j_2) = (\barn - (n+k), \barn - n +1)$.
\end{itemize}
In particular we have $j_1 \neq j_2$. Moreover, one can easily see that $\{\bfw_{j_1}, \bfv_{j_2}\}$ does not contain any of the primitive collections of 
$\wS_{\bf i}$ and so it generates a cone in $\wS_{\bf i}$ for every $\{\bfw_{j_1}, \bfv_{j_2}\}$. Furthermore, the inequalities in \eqref{eq_support_function_ineq} hold for each $\{\bfw_{j_1}, \bfv_{j_2}\}$. (Indeed, the inequalities \eqref{eq_support_function_ineq} are equalities 
since the left and right hand sides of the relations contain the same number (one) of $\bfv_j$'s.)

For the remaining four primitive collections in \eqref{eq_case4}, we have
\begin{equation}\label{eq_relation_AD}
\begin{array}{ll}
\bfws{\barn-(n+k)} +  \bfv_{\barn-n} + \bfws{\barn-k+1} = \bfw_{\widetilde{\p}_0},
& \bfw_{\widetilde{\p}_0} + \bfv_{\barn-(n+k)} = \bfws{\barn-n+1} + \bfv_{\barn-n+2}, \\
\bfw_{\widetilde{\p}_0} + \bfws{\barn-n} = \bfws{\barn-(n+k)} + \bfws{\barn-n+1}, 
& \bfw_{\widetilde{\p}_0} + \bfv_{\barn-k+1} = \bfws{\barn-(n+k)} + \bfv_{\barn-n}.
\end{array}
\end{equation}
Since the support function $\varphi_D$ is linear on each cone and
\[
\varphi_D(\bfv_j) = -2, \quad \varphi_D(\bfws{j}) = \varphi_D(\bfw_{\widetilde{\p}_0})  = 0,
\]
the relations in~\eqref{eq_relation_AD} imply that the support function $\varphi_D$ satisfies the desired inequalities~\eqref{eq_support_function_ineq}. \end{proof}

We finalize this section by presenting an example of a smooth projective toric variety $X_{\wS_{\bf i}}$ such that $\{ \mathbf{u}_{\rho} \mid \rho \in \wS_{\bf i}(1)\} = \{ \mathbf{u}_{\rho} \mid \rho \in \Sigma_{\Delta_{\bf i}(\lambda)}(1)\}$ but $X_{\wS_{\bf i}}$ is \textit{not} a small desingularization of $X_{\Delta_{\mathbf i}(\lambda)}$. 
\begin{example}\label{example_not_small}
Suppose that 
\[
\mathbf i = \isI{(\D,\A,\A,\A,\D,\D)}{0,0,0,0,0,3} = 
(4,3,4,2,3,4,1,2,3,4,5,4,6,5,4,3,2,1,4,3,2) \in \Rn{7}.
\]	
Since $\ind_{\A}(\mathbf i) = 9$ and $\ind_\D(C_\A(\mathbf i)) = 3 \neq 0$, the word $\mathbf i$ does not have small indices.
By Theorem~\ref{thm_DD_AD_combine} (also, see Example~\ref{example_DD} and Figure~\ref{fig_example_DD})  the number of rigorous paths is $21 +3 = 24$. 
Following Definition~\ref{def_Bi} we may find $\{\bfv_j, \bfw_j \mid 1 \leq j \leq 21\}$. 
For the remaining three non-canonical paths $\widetilde{\p}_0$, $\widetilde{\p}_2$, $\widetilde{\p}_3$ (see Definition~\ref{def_path_j}), we have the following relations:
\[
\widetilde{\p}_0 = \bfw_{12} + \bfv_{15} + \bfw_{19}, \quad
\widetilde{\p}_2 = \bfw_{17} +\bfv_{18} + \bfw_{20}, \quad
\widetilde{\p}_3 = \bfw_{16} + \bfv_{17} + \bfv_{18}+ \bfw_{19}+\bfw_{21}. 
\]
By setting 
\[
\tau \coloneqq \Cone( \bfw_{12}, \bfv_{15}, \bfw_{19}),\quad
\tau_2 \coloneqq \Cone(\bfw_{17}, \bfv_{18}, \bfw_{20}),\quad
\tau_3 \coloneqq \Cone(\bfw_{16}, \bfv_{17}, \bfv_{18}, \bfw_{19},\bfw_{21}),
\]
we may define the fan $\wS_{\bf i}$ by
\[
\wS_{\bf i} = \left(\left(\Sigma^{\ast}(\tau)\right)^{\ast}(\tau_{2})\right)^{\ast}(\tau_{3})
\]
Then the set of primitive collections is given by
	\[
	\begin{split}
	\PC(\wS_{\bf i}) &= \{ \{ \bfws{j}, \bfv_{j}\} \mid 1 \leq j \leq 21\} \\
	& \quad \cup \{\bfws{12}, \bfv_{15}, \bfws{19}\} \cup \{\bfw_{\widetilde{\p}_0}, \bfv_{12}\} \cup\{\bfw_{\widetilde\p_0}, \bfws{15}\} \cup \{\bfw_{\widetilde\p_0},\bfv_{19}\} \\
	& \quad \quad \cup \{\bfws{17}, \bfv_{18}, \bfw_{20} \} \cup \{\bfw_{\widetilde\p_2}, \bfv_{17}\} \cup \{\bfw_{\widetilde\p_2}, \bfw_{18}\} \cup \{\bfw_{\widetilde\p_2}, \bfv_{20}\} \\
	& \quad \quad \quad \cup \{\bfw_{16}, \bfv_{17}, \bfv_{18}, \bfw_{19},\bfw_{21}\} 
	\cup \{\bfw_{\widetilde{\p}_3}, \bfv_{16} \}
		\cup \{\bfw_{\widetilde{\p}_3}, \bfw_{17} \}
		\cup \{\bfw_{\widetilde{\p}_3}, \bfw_{18} \}	
		\cup \{\bfw_{\widetilde{\p}_3}, \bfv_{19} \}	
		\cup \{\bfw_{\widetilde{\p}_3}, \bfv_{21} \}	
	\\	
	&\quad \quad \quad \quad \cup	\{\bfw_{12}, \bfv_{15},\bfw_{\widetilde{\p}_3}\}
	\cup \{\bfw_{17}, \bfw_{20}, \bfw_{\widetilde{\p}_3}\} \cup 
	\{ \bfw_{\widetilde{\p}_2}, \bfw_{\widetilde{\p}_3}\}. 
	\end{split}
	\]
	Considering the collection $ \P = \{\bfw_{\widetilde{\p}_3}, \bfv_{19}\}$, we have the relation
	\[
	 \bfw_{\widetilde{\p}_3}  + \bfv_{19} = \mathbf{e}_{16} 
	 = \bfw_{16} + \bfv_{17} + \bfv_{18} + \bfw_{21},
	\]
	where $\{\mathbf{e}_1,\dots, \mathbf{e}_{21}\}$ is the set of standard basis vectors in $\R^{21}$. 
	Since the set $\{ \bfw_{16}, \bfv_{17}, \bfv_{18}, \bfw_{21}\}$ does not contain any primitive collection in $\PC(\wS_{\bf i})$, the summation $\bfw_{16} + \bfv_{17} + \bfv_{18} + \bfw_{21}$ is contained in $\Cone(\bfw_{16}, \bfv_{17}, \bfv_{18}, \bfw_{21}) \in \wS_{\bf i}$. 
	For the divisor $D = \sum_{j=1}^{21} 2 D_j$ as in the proof of Theorem~\ref{thm_main} and its support function $\varphi_D$, we have that
	\[
	\varphi_D(\bfw_{\widetilde{\p}_3}  + \bfv_{19} ) = 0 + (-2) + (-2) + 0 \not\geq \varphi_D(\bfw_{\widetilde{\p}_3}) + \varphi_D(\bfv_{19}) = 0 + (-2).
	\]
	Therefore the collection $\P$ does not satisfy the inequality, so that $D$ is not a basepoint free divisor on $\wS_{\bf i}$. Hence, one cannot say that the toric variety $X_{\wS_{\bf i}}$ is a small toric desingularization of the toric variety $X_{\Delta_{\mathbf i}(\lambda)}$ even though $\wS_{\bf i}$ is a smooth polytopal fan such that 
	\[
	\{ \bfu_{\rho} \mid \rho \in \wS_{\bf i}(1)\} = \{ \bfu_{\rho} \mid \rho \in \Sigma_{\Delta_{\bf i}(\lambda)}(1)\}.
	\] 
	
	Note that we can choose other vectors to construct a Bott manifold (and there exist finitely many choices in this case). But one can check that none of them defines a small toric resolution of the toric variety~$X_{\Delta_{\mathbf i}(\lambda)}$. The authors do not know whether there exists a small toric resolution of~$X_{\Delta_{\mathbf i}(\lambda)}$.
\end{example}

\appendix 
\section{Dynkin diagram automorphisms and string polytopes}
\label{appendix_automorphisms_and_string_polytopes}
Let $G$ be a connected semisimple algebraic group of rank $n$ over $\mathbb C$ and $\mathfrak{g}$ its Lie algebra. 
Fixing a Cartan subalgebra $\mathfrak{t}$ of $\frak{g}$ and an enumeration of the simple roots $\alpha_1,\dots,\alpha_n$, 
we have the Chevalley generators $\{ e_i, f_i, \alpha_i^{\vee} \mid 1 \leq i \leq n \}$ 
and the Weyl group $W$ generated by reflections $s_i$ through the hyperplanes orthogonal to the simple roots $\alpha_i$.
Here $\alpha_i^{\vee}$ is the coroot of $\alpha_i$. 
The weight lattice $\Lambda$ is the set of all $\lambda \in \mathfrak{t}^*$ such that $\langle \lambda, \alpha_i^{\vee} \rangle \in \mathbb Z$
and $\Lambda$ has a $\mathbb Z$-basis consisting of the fundamental weights $\varpi_1,\dots,\varpi_n$, 
which are determined by the relation $\langle \varpi_i, \alpha_j^{\vee} \rangle = \delta_{i,j}$. 
We call a weight $\lambda = \lambda_1 \varpi_1 + \cdots + \lambda_n \varpi_n$ {\em dominant} if $\lambda_i \geq 0$ for all $i = 1,\dots,n$.
Let $\Lambda_{+}$ denote the set of dominant integral weights. 

For a dominant weight $\lambda$, we denote a (finite-dimensional) irreducible representation of $G$ with highest weight $\lambda$ by $V_{\lambda}$. 
Then $V_{\lambda}$ has a remarkable basis $\mathcal B_{\lambda}$ consisting of the nonzero vectors $b v_{\lambda}$, 
where $b$ lies in the specialization at $q=1$ of the Lusztig canonical basis for the quantized enveloping algebra $U_q(\mathfrak g)$ of $\mathfrak{g}$ over $\mathbb C(q)$
(for details, see \cite{Kash90} and \cite[Section 3]{Kav15}). 
Denote by $\tilde{e}_i, \tilde{f}_i \colon \mathcal B_{\lambda} \to \mathcal B_{\lambda} \cup \{ 0 \}$ 
the raising and lowering Kashiwara operators for $V_{\lambda}$ corresponding to the simple root $\alpha_i$. 

Now, we define the string parametrization for elements of a \emph{dual crystal basis} $\mathcal B_{\lambda}^*$ of the dual representation $V_{\lambda}^*$. 
This depends on a reduced word $\mathbf i = (i_1,\dots,i_\barn) \in [n]^\barn$ for the longest element $w_0 \in W$, $w_0 = s_{i_1} \cdots s_{i_\barn}$, where $\barn= \ell(w_0)$. 
The set of reduced words for $w\in W$ will be denoted by $R(w)$. 

\begin{definition} 
	For a reduced word $\mathbf i = (i_1,\dots,i_\barn) \in R(w_0)$, we define a map 
	$\Phi_{\mathbf i} \colon \mathcal B_{\lambda}^* \to \mathbb Z^\barn_{\geq 0}$ by $\Phi_{\mathbf i} (b) = (t_1, \cdots, t_\barn)$, where 
	\[
	\begin{split}
	t_1 \coloneqq &~ \mbox{max}~ \{ a \mid \tilde{f}_{i_1}^a (b) \neq 0 \}, \\
	t_2 \coloneqq &~ \mbox{max}~ \{ a \mid \tilde{f}_{i_2}^a \tilde{f}_{i_1}^{t_1} (b) \neq 0 \}, \\
	\cdots & \\
	t_\barn \coloneqq &~ \mbox{max}~ \{ a \mid \tilde{f}_{i_\barn}^a \cdots \tilde{f}_{i_2}^{t_2} \tilde{f}_{i_1}^{t_1} (b) \neq 0 \}.
	\end{split}
	\]
	This map is called \emph{string parametrization} of $\mathcal B_{\lambda}^*$ with respect to $\mathbf i$. 
\end{definition}

\begin{proposition}[{\cite[Proposition 1.5]{Li}} and {\cite[Proposition 3.5]{BeZe01}}]\label{rational polyhedral convex cone}
	There exists a (unique) rational polyhedral convex cone $\mathcal C_{\mathbf i} \subset \Lambda_{\mathbb R} \times \mathbb R^\barn$
	such that the union $\displaystyle \bigcup_{\lambda \in \Lambda_+} \{ (\lambda, \Phi_{\mathbf i} (b)) \mid b \in  \mathcal B_{\lambda}^*  \}$ 
	is the intersection of $\mathcal C_{\mathbf i}$ with the lattice $\Lambda \times \mathbb Z^\barn$. 
\end{proposition}

The projection of $\mathcal C_{\mathbf i}$ to the second factor $\mathbb R^\barn$ is also a rational polyhedral convex cone 
and we call it \emph{string cone} $C_{\mathbf i}$ associated to $\mathbf i \in R(w_0)$.
Since the highest weight $G$-module of the weight $0$ is trivial, $\mathcal C_{\mathbf i}$ intersects with $\{ 0 \} \times \mathbb R^\barn$ only at the origin. 
Thus the slice of the cone $\mathcal C_{\mathbf i}$ at a fixed $\lambda \in \Lambda_+$ is a rational polyhedral polytope in $\mathbb R^\barn$. 

\begin{definition} 
	For a dominant weight $\lambda \in \Lambda_+$ and a reduced word $\mathbf i = (i_1,\dots,i_\barn) \in R(w_0)$, 
	the slice of $\mathcal C_{\mathbf i}$ at $\lambda$ is 
	the \emph{string polytope} $\Delta_{\mathbf i}(\lambda) = \{ \mathbf t \mid (\lambda, \mathbf t) \in \mathcal C_{\mathbf i} \} \subset \mathbb R^\barn$. 
\end{definition}

It follows from \cite{Li} that the string polytope $\Delta_{\mathbf i}(\lambda)$ can be obtained by intersecting the string cone $C_{\mathbf i}$ with the \emph{$\lambda$-cone}:
\begin{equation}\label{eq_ieqs_string_polytope}
\Delta_{\mathbf i}(\lambda) = C_{\mathbf i} \cap \{ \mathbf t \in \mathbb{R}^\barn_{\geq 0} \mid l_j( \mathbf t)  \leq \langle \lambda, \alpha_{i_j}^{\vee} \rangle \text{ for } 1 \leq j \leq \barn \},
\end{equation}
where $l_1, \dots, l_\barn$ are linear functions defined by  
\begin{equation}\label{eq_linear_ftns_l_j} 
l_j(\mathbf t) \coloneqq t_{j} + \langle t_{j+1} \alpha_{i_{j+1}} + \cdots + t_\barn \alpha_{i_\barn}, \alpha_{i_j}^{\vee} \rangle \quad \text{ for } 1 \leq j \leq \barn.
\end{equation}

In this case where $G = \SL_{n+1}(\C)$, the inequalities defining the string cone $C_{\bf i}$ are described explicitly in \cite[Proposition 3.14]{BeZe01} and can be written from the Gleizer--Postnikov's paths in the wiring diagram in \cite{GlPo00}. 
Also, the inequalities defining the $\lambda$-cone can be read off from the wiring diagram in \cite{Rusinko08}. See Section~\ref{secDescriptionOfStringPolytopes}.

\begin{example}\label{example_string_polytope_3}
	Let $G = \SL_3(\mathbb{C})$, and $\lambda = 2 \varpi_1 + 2\varpi_2$. Let $\mathbf{i} = (1,2,1)$. Then the linear functions $l_1,l_2,l_3$ are given by
	\[
	l_1(\textbf t) = t_1 - t_2 + 2 t_3, \quad
	l_2(\textbf t) = t_2 - t_3, \quad 
	l_3(\textbf t) = t_3.
	\]
	The string cone $C_{\bf i}$ is the set of points $(t_1,t_2,t_3) \in \mathbb{R}^3$ satisfying
	\[
	t_1 \geq 0, \quad t_2 \geq t_3 \geq 0.
	\]
	The string polytope $\Delta_{\bf i}(\lambda)$ is given by the set of points $\mathbf t=(t_1, t_2, t_3) \in \mathbb{R}^3$ satisfying:
	\[
	\begin{split}
	0 \leq &~t_1 \leq t_2 - 2t_3 + 2, \\
	t_3 \leq &~t_2 \leq t_3 + 2, \\
	0 \leq &~t_3 \leq 2,
	\end{split}
	\]
	which is described in Figure~\ref{string_polytope_3}.
\end{example}

\begin{figure}[H]
		\tdplotsetmaincoords{70}{100}
		\begin{tikzpicture}[tdplot_main_coords, scale=0.8]
		
		\coordinate (1) at (0,0,0);
		\coordinate (2) at (2,0,0);
		\coordinate (3) at (4,2,0);
		\coordinate (4) at (2,4,2);
		\coordinate (5) at (0,4,2);
		\coordinate (6) at (0,2,2);
		\coordinate (7) at (0,2,0);
		
		\begin{scope}[color=gray!50, thin]
		\foreach \xi in {0,1,2,3,4,5} { \draw (\xi, 5,0) -- (\xi, 0,0) -- (\xi, 0, 3); }%
		\foreach \yi in {0,1,2,3,4,5} {\draw (5,\yi,0) -- (0,\yi,0) -- (0,\yi,3);}%
		\foreach \zi in {0,1,2,3} {\draw (5,0,\zi) -- (0,0,\zi) -- (0,5,\zi);}%
		\end{scope}

		\draw[->] (0,0,0) -- (5.5,0,0) node[anchor = north] {$t_1$};
		\draw[->] (0,0,0) -- (0,5.5,0) node[anchor = south] {$t_2$} ;
		\draw[->] (0,0,0) -- (0,0,3.5) node[anchor = south] {$t_3$};	
		
		\draw[thick] (1)--(2)--(3)--(4)--(5)--(6)--cycle;
		\draw[thick] (4)--(6)--(2);
		\draw[thick, dashed] (5)--(7)--(3);
		\draw[thick, dashed] (7)--(1);
		
		\foreach \x in {1,...,7}{
			\node[circle,fill=black,inner sep=1pt] at (\x) {};
		}
		
		\end{tikzpicture}
		\caption{The string polytope $\Delta_{(1,2,1)}(2\varpi_1  + 2\varpi_2)$ in Example~\ref{example_string_polytope_3}.} 
		\label{string_polytope_3}
\end{figure}
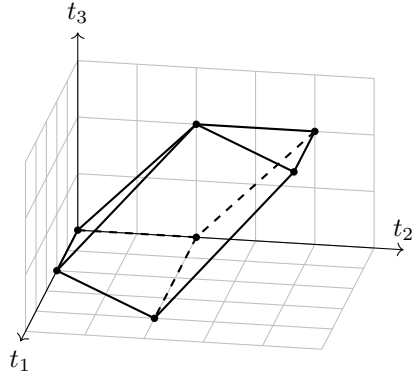

\begin{definition} 
	Let $\mathfrak g$ be a semisimple Lie algebra with Cartan matrix $C=(c_{i,j})_{1\leq i, j \leq n}$. 
	A bijection $\theta \colon [n] \to [n]$ satisfying $c_{\theta(i),\theta(j)}=c_{i,j}$ for all $i, j$ is called a \emph{Dynkin diagram automorphism}. 
\end{definition}


	A Dynkin diagram automorphism naturally induces a Lie algebra automorphism $\theta \colon \mathfrak g \to \mathfrak g$
	such that $\theta(e_i)= e_{\theta(i)}$, $\theta(f_i)= f_{\theta(i)}$, $\theta(\alpha_i^{\vee})= \alpha_{\theta(i)}^{\vee}$ for all $i$
	(we will the same notation for simplicity). 
	We define a $\mathbb C$-linear automorphism $\theta^* \colon \mathfrak t^* \to \mathfrak t^*$ by $\langle \theta^*(\lambda), t \rangle = \langle \lambda, \theta^{-1}(t) \rangle$ for $\lambda \in \mathfrak t^*$ and $t \in \mathfrak t$. 
	Then we get $\theta^* (\varpi_i) = \varpi_{\theta(i)}$ since $\langle \theta^*(\varpi_i), \alpha_{\theta(j)}^{\vee} \rangle = \langle \varpi_i, \theta^{-1} (\alpha_{\theta(j)}^{\vee}) \rangle =  \langle \varpi_i, \alpha_{j}^{\vee} \rangle = \delta_{i, j}$. 

\begin{example}\label{example_Dynkin diagram automorphisms}
	\begin{enumerate}
		\item For $\mathfrak g=\mathfrak{sl}_{n+1}(\mathbb C)$, the \defi{involution} $\iota \colon [n] \to [n]$ defined by $\iota (i) = n + 1 - i  \text{ for }1 \leq i \leq n$
		is a Dynkin diagram automorphism (see Figure~\ref{fig_Dynkin_diagram_An}). 
		
		\item A non-trivial Dynkin diagram automorphism exists only when $\mathfrak g$ is a Lie algebra of type $A_n$ ($n \geq 2$), $D_n$, or $E_6$ (see Figures~\ref{fig_Dynkin_diagram_Dn} and~\ref{fig_Dynkin_diagram_E6}). 
		All these algebras except $D_4$ have a unique non-trivial Dynkin diagram automorphism of order 2. 
		Since $D_4$ also has a Dynkin diagram automorphism of order 3, the group of its Dynkin diagram automorphisms is isomorphic to the symmetric group $\mathfrak{S}_{3}$ (see Figure~\ref{fig_Dynkin_diagram_D4}).
	\end{enumerate}
\end{example}
\begin{figure}[h]
	\begin{subfigure}[b]{0.3\textwidth}
		\centering
		\def\x{-0.15}
		\begin{tikzpicture}[inner sep = 0.45mm]

		\node (1) at (0,0) [circle, draw] {};
		\node (2) at (1,0) [circle, draw] {};
		\node (22) at (2,0) [circle, draw] {};
		\node (3) at (3,0) [circle, draw] {};
		\node (4) at (4,0) [circle, draw] {};
		
		\node [below] at (0,\x) {\small$1$};
		\node [below] at (1,\x) {\small$2$};
		\node [below] at (2,\x) {\small$3$};
		\node [below] at (3,\x) {\small$n-1$};
		\node [below] at (4,\x) {\small$n$};
		
		\path (1) edge (2) 
		(2) edge (22)
		(3) edge (4);
		\path[dashed] (22) edge (3);
		\end{tikzpicture}
		\caption{$A_n$.}
		\label{fig_Dynkin_diagram_An}
	\end{subfigure}
	\begin{subfigure}[b]{0.3\textwidth}
		\centering
		\def\x{-0.15}
		\begin{tikzpicture}[scale = 1, inner sep = 0.45mm]
		\node (1) at (0,0) [circle, draw] {};
		\node (2) at (1,0) [circle, draw] {};
		\node (3) at (2,0) [circle, draw] {};
		\node (4) at (3,0) [circle, draw] {};
		\node (5) at (4,0.3) [circle, draw] {};
		\node (6) at (4,-0.3) [circle, draw] {};
		
		\node [below] at (0,\x) {\small$1$};
		\node [below] at (1,\x) {\small$2$};
		\node [below] at (2,\x) {\small$n-3$};
		\node [below] at (3,\x) {\small$n-2$};
		\node [right] at (5.east) {\small$n-1$};
		\node [right] at (6.east) {\small$n$};
		
		\path (1) edge (2) 
		(3) edge (4)
		(4) edge (5)
		(4) edge (6);
		\path[dashed] (2) edge (3);
		\end{tikzpicture}
		\caption{$D_n$ ($n \geq 4$).}
		\label{fig_Dynkin_diagram_Dn}	
	\end{subfigure}
	\\ [1em]
	\begin{subfigure}[t]{0.3\textwidth}
		\centering
		\begin{tikzpicture}[inner sep = 0.45mm]
		\node (1) at (0,0) [circle, draw] {};
		\node (2) at (1,0) [circle, draw] {};
		\node (22) at (2,0) [circle, draw] {};
		\node (3) at (3,0) [circle, draw] {};
		\node (4) at (4,0) [circle, draw] {};
		\node (5) at (2,1) [circle, draw] {};
		
		\node [below] at (1.south) {\small $1$};
		\node [below] at (2.south) {\small$3$};
		\node [below] at (22.south) {\small$4$};
		\node [below] at (3.south) {\small$5$};
		\node [below] at (4.south) {\small$6$};
		\node [right] at (5.east) {\small$2$};
		
		\path (1) edge (2) 
		(2) edge (22)
		(3) edge (4)
		(22) edge (3)
		(5) edge (22);
		\end{tikzpicture}
		\caption{$E_6$.}
		\label{fig_Dynkin_diagram_E6}
	\end{subfigure}
	\begin{subfigure}[t]{0.3\textwidth}
		\centering
		\begin{tikzpicture}[inner sep = 0.45mm]
		
		\foreach \a in {1,2,3}{
			\draw (\a*360/3+60:1cm) node (\a) [circle, draw] {};	
		}
		\node (0) at (0,0) [circle,draw] {};
		
		\node [below] at (1.south) {\small $1$};
		\node [below] at (0.south) {\small $2$};
		\node [right] at (3.east) {\small $3$};
		\node [right] at (2.east) {\small $4$};
		
		\path (0) edge (1)
		(0) edge (2)
		(0) edge (3);
		
		%
		%
		%
		\end{tikzpicture}
		\caption{$D_4$.}
		\label{fig_Dynkin_diagram_D4}
	\end{subfigure}
	
	\caption{Dynkin diagrams.}
	\label{fig_Dynkin_diagram}
\end{figure}
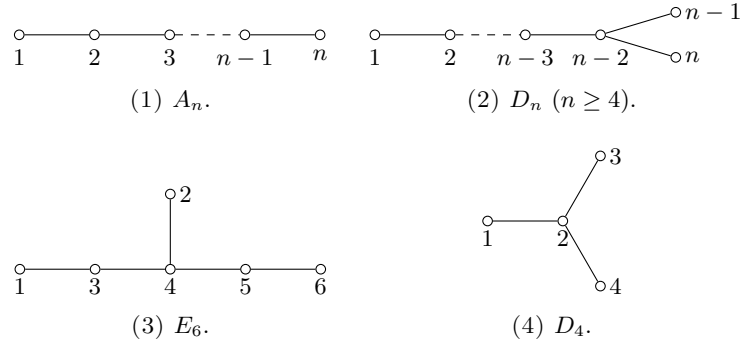

Note that the group of Dynkin diagram automorphisms of a semisimple Lie algebra $\mathfrak g$ is isomorphic to the group of outer automorphisms of $\mathfrak g$ (\cite[Section 4 of Chapter 4]{OV}). 

\begin{proposition}\label{same string polytopes under Dynkin diagram automorphism}
	Let $\theta$ be a Dynkin diagram automorphism of $\mathfrak g$. 
	If we consider $\theta(\mathbf i) = (\theta(i_1),\dots, \theta(i_\barn))$ for a reduced word $\mathbf i = (i_1,\dots,i_\barn) \in R(w_0)$, 
	then we have the same string polytopes $\Delta_{\theta(\mathbf i)}(\theta^*(\lambda)) = \Delta_{\mathbf i}(\lambda)$ for any dominant weight $\lambda \in \Lambda_+$.
\end{proposition}

\begin{proof}
	The induced Lie algebra automorphism $\theta \colon \mathfrak g \to \mathfrak g$ also induces a $\mathbb C(q)$-algebra automorphism $\theta \colon U_q(\mathfrak g) \to U_q(\mathfrak g)$ 
	preserving the $\mathbb C(q)$-subalgebra $U_q^-(\mathfrak g)$ generated by the Chevalley generators $\{ f_i \mid i= 1, \dots, n\}$ corresponding to negative roots. 
	By \cite[Lemma 2.3.2]{NaSa03}, we obtain a $\mathbb C$-linear isomorphism $\bar{\theta} \colon V_{\lambda} \to V_{\theta^*(\lambda)}$ induced from $\theta \colon U_q^-(\mathfrak g) \to U_q^-(\mathfrak g)$, 
and on the crystal basis $\mathcal B_{\lambda}$ of $V_{\lambda}$ we have $\bar{\theta} \circ \tilde{e}_{i} = \tilde{e}_{\theta(i)} \circ \bar{\theta}$ and $\bar{\theta} \circ \tilde{f}_{i} = \tilde{f}_{\theta(i)} \circ \bar{\theta}$ for each $i$. 
	Because the crystal structures of the crystal bases $\mathcal B_{\lambda}$ and $\mathcal B_{\theta^*(\lambda)}$ are invariant under $\bar{\theta}$, 
	we conclude that $\Delta_{\theta(\mathbf i)}(\theta^*(\lambda)) = \Delta_{\mathbf i}(\lambda)$. 
\end{proof}

\section{Reduced words of the longest element in $\mathfrak{S}_5$ having small indices}
\label{appendix_S5}
In this section, we analyze reduced words in $\Rn{5}$, and present elements in $\Rn{5}$ which have small indices.
Using the hook length formula~\eqref{eq_hook_length_formula}, there are $768$ many reduced words of the longest element in $\mathfrak{S}_5$. 
By the result~\cite[\S 3]{Bedard99}, there are $62$ reduced words up to $2$-moves. Furthermore, considering the involution in Example~\ref{example_Dynkin diagram automorphisms}(1) and Proposition~\ref{same string polytopes under Dynkin diagram automorphism}, it is enough to consider 31 elements in $\Rn{5}$ to study combinatorics of the string polytopes $\Delta_{\mathbf i}(\lambda)$. 
In Table~\ref{table_S5}, we consider these 31 elements and check whether they have small indices or not. 
The number on the first column is the index given in~\cite[Table~1]{Bedard99}.

\begin{table}[b]
	\begin{tabular}{c c c c c}
		\toprule
		\# & reduced word $\mathbf i$ & $\is$ & small indices & $ |\GP(\mathbf i)|$ \\
		\midrule 
$1$	&	$(     1,2,1,3,2,1,4,3,2,1	        )$	&	$\mathbf i = \isI{(\D,\D,\D,\D)}{0,0,0,0}$  & $\fullmoon$ & $10$
\\
$2$	&	$(	2,1,2,3,2,1,4,3,2,1	)$	&      $\mathbf i = \isI{(\D,\A,\D,\D)}{0,0,0,0}$ & $\fullmoon$ &	$10$
 \\
$3$	&	$(	1,2,3,2,1,2,4,3,2,1	)$	&	$\mathbf i \sim \isI{(\D,\A,\D,\A)}{0,0,0,6}$ & $\times$ & $11$
 \\
$4$	&	$(	1,2,1,3,2,4,3,2,1,2	)$	&	$\mathbf i = \isI{(\D,\D,\D,\D)}{0,0,0,1}$ & $\fullmoon$ &$11$
 \\
$5$	&	$(	2,1,3,2,3,1,4,3,2,1	)$	&	$\mathbf i \sim \isI{(\D,\D,\D,\A)}{0,0,0,4}$ & $\times$ &$11$
 \\
$6$	&	$(	2,1,2,3,2,4,3,2,1,2	)$	&	$\mathbf i =  \isI{(\D,\A,\D,\D)}{0,0,0,1}$ & $\fullmoon$ &$11$
 \\ 
$7$	&	$(	1,3,2,3,1,2,4,3,2,1	)$	&	$\mathbf i \sim \isI{(\D,\A,\D,\A)}{0,0,0,5}$ & $\times$ &$11$
 \\
$8$	&	$(	1,2,3,2,1,4,3,2,3,1	)$	&	$\mathbf i \sim \isI{(\D,\D,\D,\D)}{0,0,1,1}$ & $\times$& $14$
\\
$9$	&	$(	1,2,1,3,4,3,2,3,1,2	)$	&	$\mathbf i  \sim \isI{(\D,\D,\D,\D)}{0,0,0,2}$ & $\fullmoon$ & $12$
\\
$10$	&	$(	2,1,3,2,1,4,3,4,2,1	)$	&	$\mathbf i \sim \isI{(\D,\A,\D,\D)}{0,0,1,1}$ & $\times$&$13$
\\
$11$	&	$(	2,3,2,1,2,3,4,3,2,1	)$	&      $\mathbf i = \isI{(\D,\D,\A,\D)}{0,0,0,0}$ & $\fullmoon$  &$10$
\\
$12$	&	$(	2,1,3,2,3,4,3,2,1,2	)$	&	$\mathbf i \sim \isI{(\D,\D,\A,\D)}{0,0,1,1}$ &$\times$ & $13$
\\
$13$	&	$(	2,1,2,3,4,3,2,3,1,2	)$	&	$\mathbf i \sim \isI{(\D,\A,\D,\D)}{0,0,0,2}$ & $\fullmoon$ & $12$
\\
$14$	&	$(	1,3,2,3,1,4,3,2,3,1	)$     &      $ \mathbf i \sim \isI{(\D,\D,\D,\D)}{0,0,2,1}$ &$\times$& $13$
\\
$15$	&	$(	3,2,1,2,3,2,4,3,2,1	)$	&      $\mathbf i \sim \isI{(\D,\A,\D,\A)}{0,0,0,4}$ &$\times$&	$11$
\\
$16$	&	$(	1,2,3,2,4,3,2,1,2,3	)$	&	$\mathbf i \sim \isI{(\D,\D,\A,\A)}{0,0,0,6}$ & $\times$ &$14$
\\
$17$	&	$(	1,2,1,4,3,4,2,3,1,2	)$     &	$\mathbf i \sim \isI{(\D,\D,\D,\D)}{0,0,0,3}$ & $\times$ & $12$
\\
$18$	&	$(	1,2,3,4,3,2,1,2,3,2	)$	&	$\mathbf i = \isI{(\D,\D,\A,\D)}{0,0,3,3}$ &$\times$& $13$
\\
$19$	&	$(	2,3,2,1,2,4,3,4,2,1	)$	&	$\mathbf i  \sim \isI{(\D,\D,\A,\D)}{0,0,0,1}$ & $\fullmoon$ & $11$
\\
$20$	&	$(	2,1,3,2,4,3,4,2,1,2	)$	&	$\mathbf i \sim \isI{(\D,\A,\D,\D)}{0,0,1,2}$ & $\times$& $16$
\\
$21$	&	$(	3,2,3,1,2,3,4,3,2,1	)$	&	$\mathbf i  = \isI{(\D,\A,\A,\D)}{0,0,0,0}$ & $\fullmoon$ & $10$
\\
$22$	&	$(	2,1,2,4,3,4,2,3,1,2	)$	&	$\mathbf i \sim \isI{(\D,\A,\D,\D)}{0,0,0,3}$ & $\times$ & $12$
\\
$23$	&	$(	1,3,2,1,4,3,4,2,3,1	)$	&	$\mathbf i \sim \isI{(\D,\D,\D,\D)}{0,0,2,2}$ & $\times$ & $15$
\\
$24$	&	$(	3,2,1,2,3,4,3,2,3,1	)$	&	$\mathbf i \sim \isI{(\D,\D,\D,\D)}{0,0,3,1}$ & $\times$ & $14$
\\
$25$	&	$(	1,3,2,3,4,3,2,1,2,3	)$	&	$\mathbf i \sim \isI{(\D,\A,\A,\D)}{0,0,2,2}$ & $\times$ & $13$
\\
$26$	&	$(	1,2,3,4,3,2,3,1,2,3	)$	&	$\mathbf i \sim \isI{(\D,\D,\D,\D)}{0,0,1,3}$ & $\times$ & $14$
\\
$27$	&	$(	1,2,4,3,4,2,1,2,3,2	)$     &	$\mathbf i  \sim \isI{(\D,\D,\D,\D)}{0,0,0,4}$ & $\times$& $13$
\\
$28$	&	$(	3,2,3,1,2,4,3,4,2,1	)$	&	$\mathbf i  \sim \isI{(\D,\A,\A,\D)}{0,0,0,1}$ & $\fullmoon$ & $11$
\\
$30$	&	$(	2,1,3,4,3,2,3,4,1,2	)$	&	$\mathbf i \sim \isI{(\D,\A,\D,\D)}{0,0,1,3}$ & $\times$ & $17$
\\
$31$	&	$(	2,1,4,3,2,3,4,3,1,2	)$	&	$\mathbf i \sim \isI{(\D,\A,\D,\D)}{0,0,0,4}$ & $\times$ & $15$
\\
$34$	&	$(	1,2,4,3,4,2,3,1,2,3	)$	&      $\mathbf i \sim \isI{(\D,\D,\D,\D)}{0,0,1,4}$ & $\times$ &	$14$\\
\bottomrule
	\end{tabular}
\caption{Reduced words in $\Rn{5}$.}
\label{table_S5}
\end{table}

It has been known from~\cite[Example~5.7]{AB} that the string polytopes $\Delta_{\mathbf i}(\lambda)$ are integral for $n \leq 4$ and $\lambda = \sum_{i=1}^n 2 \varpi_i$. 
Moreover, one can check that $\Delta_{\mathbf i}(\varpi_i)$ are also integral for all $1 \leq i \leq n$ using the computer program SAGE.
Indeed, for $n \leq 3$, we already proved  that $X_{\Delta_{\mathbf i}(\lambda)}$ admits a small toric resolution (in Theorem~\ref{thm_main}) so that $\Delta_{\mathbf i}(\lambda)$ is integral (in Corollary~\ref{cor_integral_polytope}).
We may address the following question.
\begin{Question}
When does the toric variety $X_{\Delta_{\mathbf i}(\lambda)}$ admit a small toric resolution for $\mathbf i \in \Rn{5}$ and a regular dominant integral weight $\lambda$? 
\end{Question}


\begin{thebibliography}{BCFKvS00}
	
	\bibitem[AB04]{AB}
	Valery Alexeev and Michel Brion, \emph{Toric degenerations of spherical
		varieties}, Selecta Math. (N.S.) \textbf{10} (2004), no.~4, 453--478.
	
	\bibitem[And13]{An13}
	Dave Anderson, \emph{Okounkov bodies and toric degenerations}, Math. Ann.
	\textbf{356} (2013), no.~3, 1183--1202.
	
	\bibitem[Aur07]{Aur}
	Denis Auroux, \emph{Mirror symmetry and {$T$}-duality in the complement of an
		anticanonical divisor}, J. G\"{o}kova Geom. Topol. GGT \textbf{1} (2007),
	51--91.
	
	\bibitem[B99]{Bedard99}
	Robert Bédard, \emph{On commutation classes of reduced words in {W}eyl
		groups}, European J. Combin. \textbf{20} (1999), no.~6, 483--505.
	
	\bibitem[BB05]{BB05Combinatorics}
	Anders Björner and Francesco Brenti, \emph{Combinatorics of {C}oxeter
		groups}, Graduate Texts in Mathematics, vol. 231, Springer, New York, 2005.
	
	\bibitem[BCFKvS00]{BCKV}
	Victor~V. Batyrev, Ionu\c{t} Ciocan-Fontanine, Bumsig Kim, and Duco van
	Straten, \emph{Mirror symmetry and toric degenerations of partial flag
		manifolds}, Acta Math. \textbf{184} (2000), no.~1, 1--39.
	
	\bibitem[BF19]{BF}
	Lara Bossinger and Ghislain Fourier, \emph{String cone and superpotential
		combinatorics for flag and {S}chubert varieties in type {A}}, J. Combin.
	Theory Ser. A \textbf{167} (2019), 213--256.
	
	\bibitem[BP15]{BP2015ToricTop2015}
	Victor~M. Buchstaber and Taras~E. Panov, \emph{Toric topology}, Mathematical
	Surveys and Monographs, vol. 204, American Mathematical Society, Providence,
	RI, 2015.
	
	\bibitem[BZ93]{BeZe93}
	Arkady Berenstein and Andrei Zelevinsky, \emph{String bases for quantum groups
		of type {{${\sf A}_r$}}}, I. {M}. {G}elfand {S}eminar, Adv. Soviet Math.,
	vol.~16, Amer. Math. Soc., Providence, RI, 1993, pp.~51--89.
	
	\bibitem[BZ96]{BeZe96}
	\bysame, \emph{Canonical bases for the quantum group of type {$A_r$} and
		piecewise-linear combinatorics}, Duke Math. J. \textbf{82} (1996), no.~3,
	473--502.
	
	\bibitem[BZ01]{BeZe01}
	\bysame, \emph{Tensor product multiplicities, canonical bases and totally
		positive varieties}, Invent. Math. \textbf{143} (2001), no.~1, 77--128.
	
	\bibitem[Cal02]{Cal}
	Philippe Caldero, \emph{Toric degenerations of {S}chubert varieties},
	Transform. Groups \textbf{7} (2002), no.~1, 51--60.
	
	\bibitem[CKLP19]{CKLP}
	Yunhyung Cho, Yoosik Kim, Eunjeong Lee, and Kyeong-Dong Park, \emph{On the
		combinatorics of string polytopes}, arXiv preprint arXiv:1904.00130 (2019).
	
	\bibitem[CLS11]{CLS11toric}
	David~A. Cox, John~B. Little, and Henry~K. Schenck, \emph{Toric varieties},
	Graduate Studies in Mathematics, vol. 124, American Mathematical Society,
	Providence, RI, 2011.
	
	\bibitem[CO06]{COtoric}
	Cheol-Hyun Cho and Yong-Geun Oh, \emph{Floer cohomology and disc instantons of
		{L}agrangian torus fibers in {F}ano toric manifolds}, Asian J. Math.
	\textbf{10} (2006), no.~4, 773--814.
	
	\bibitem[Ewa96]{Ewald96Combinatorial}
	G\"unter Ewald, \emph{Combinatorial convexity and algebraic geometry}, Graduate
	Texts in Mathematics, vol. 168, Springer-Verlag, New York, 1996.
	
	\bibitem[FOOO09]{FOOO}
	Kenji Fukaya, Yong-Geun Oh, Hiroshi Ohta, and Kaoru Ono, \emph{Lagrangian
		intersection {F}loer theory: anomaly and obstruction. {P}art {I} and {II}},
	AMS/IP Studies in Advanced Mathematics, vol.~46, American Mathematical
	Society, Providence, RI, 2009.
	
	\bibitem[GK94]{GrKa94}
	Michael Grossberg and Yael Karshon, \emph{Bott towers, complete integrability,
		and the extended character of representations}, Duke Math. J. \textbf{76}
	(1994), no.~1, 23--58.
	
	\bibitem[GL96]{GL}
	Nicolae Gonciulea and Venkatramani Lakshmibai, \emph{Degenerations of flag and
		{S}chubert varieties to toric varieties}, Transform. Groups \textbf{1}
	(1996), no.~3, 215--248.
	
	\bibitem[GP00]{GlPo00}
	Oleg Gleizer and Alexander Postnikov, \emph{Littlewood-{R}ichardson
		coefficients via {Y}ang-{B}axter equation}, Internat. Math. Res. Notices
	(2000), no.~14, 741--774.
	
	
	\bibitem[HK15]{HaKa15}
	Megumi Harada and Kiumars Kaveh, \emph{Integrable systems, toric degenerations
		and {O}kounkov bodies}, Invent. Math. \textbf{202} (2015), no.~3, 927--985.
\bibitem[HY18]{HYstring}
Megumi Harada and Jihyeon~Jessie Yang, \emph{Singular string polytopes and
	functorial resolutions from {N}ewton-{O}kounkov bodies}, Illinois J. Math.
\textbf{62} (2018), no.~1-4, 271--292.
	
	\bibitem[Kas90]{Kash90}
	Masaki Kashiwara, \emph{Crystalizing the {$q$}-analogue of universal enveloping
		algebras}, Comm. Math. Phys. \textbf{133} (1990), no.~2, 249--260.
	
	\bibitem[Kav15]{Kav15}
	Kiumars Kaveh, \emph{Crystal bases and {N}ewton--{O}kounkov bodies}, Duke Math.
	J. \textbf{164} (2015), no.~13, 2461--2506.
	
	\bibitem[KK12]{KaKho12}
	Kiumars Kaveh and Askold Georgievich Khovanskii, \emph{Newton--{O}kounkov bodies, semigroups
		of integral points, graded algebras and intersection theory}, Ann. of Math.
	(2) \textbf{176} (2012), no.~2, 925--978.
	
	\bibitem[KM05]{KM}
	Mikhail Kogan and Ezra Miller, \emph{Toric degeneration of {S}chubert varieties
		and {G}elfand-{T}setlin polytopes}, Adv. Math. \textbf{193} (2005), no.~1,
	1--17.
	
	\bibitem[Lit98]{Li}
	Peter Littelmann, \emph{Cones, crystals, and patterns}, Transform. Groups
	\textbf{3} (1998), no.~2, 145--179.
	
	\bibitem[LM09]{LM}
	Robert Lazarsfeld and Mircea Musta\c{t}\u{a}, \emph{Convex bodies associated to
		linear series}, Ann. Sci. \'{E}c. Norm. Sup\'{e}r. (4) \textbf{42} (2009),
	no.~5, 783--835.
	
	\bibitem[NNU10]{NNU10}
	Takeo Nishinou, Yuichi Nohara, and Kazushi Ueda, \emph{Toric degenerations of
		{G}elfand--{C}etlin systems and potential functions}, Adv. Math. \textbf{224}
	(2010), no.~2, 648--706.
	
	\bibitem[NNU12]{NNU12}
	\bysame, \emph{Potential functions via toric degenerations}, Proc. Japan Acad.
	Ser. A Math. Sci. \textbf{88} (2012), no.~2, 31--33.
	
	\bibitem[NS03]{NaSa03}
	Satoshi Naito and Daisuke Sagaki, \emph{Three kinds of extremal weight vectors
		fixed by a diagram automorphism}, J. Algebra \textbf{268} (2003), no.~1,
	343--365.
	
	\bibitem[OV90]{OV}
	Arkadiĭ L'vovich Onishchik and Èrnest Borisovich Vinberg, \emph{Lie groups and algebraic groups},
	Springer Series in Soviet Mathematics, Springer-Verlag, Berlin, 1990,
	Translated from the Russian and with a preface by D. A. Leites.
	
	\bibitem[Rus08]{Rusinko08}
	Joe Rusinko, \emph{Equivalence of mirror families constructed from toric
		degenerations of flag varieties}, Transform. Groups \textbf{13} (2008),
	no.~1, 173--194.
	
	\bibitem[Sat00]{Sato_Towards_2000}
	Hiroshi Sato, \emph{Toward the classification of higher-dimensional toric
		{F}ano varieties}, Tohoku Math. J. (2) \textbf{52} (2000), no.~3, 383--413.
	
	\bibitem[Ste19]{Steinert19}
	Christian Steinert, \emph{Reflexivity of {N}ewton--{O}kounkov bodies of partial
		flag varieties}, arXiv preprint arXiv:1902.07105 (2019).
	
	\bibitem[Tit69]{Ti}
	Jacques Tits, \emph{Le probl\`eme des mots dans les groupes de {C}oxeter},
	Symposia {M}athematica ({INDAM}, {R}ome, 1967/68), {V}ol. 1, Academic Press,
	London, 1969, pp.~175--185.
	
\end{thebibliography}

\providecommand{\bysame}{\leavevmode\hbox to3em{\hrulefill}\thinspace}
\providecommand{\MR}{\relax\ifhmode\unskip\space\fi MR }
\providecommand{\MRhref}[2]{%
	\href{http://www.ams.org/mathscinet-getitem?mr=#1}{#2}
}
\providecommand{\href}[2]{#2}

\end{document}